\documentclass[11pt]{amsart}
\usepackage[utf8]{inputenc}
\usepackage{amsmath,amsthm,amsfonts,amssymb,amscd,mathtools}
\usepackage{verbatim}
\usepackage{mathrsfs}
\usepackage{multicol}
\usepackage{tabularx}
\usepackage[usenames,dvipsnames]{color}
\usepackage{enumerate}
\usepackage{graphicx}
\usepackage{color,xcolor}
\usepackage{bbm}
\usepackage[all]{xy}
\usepackage{hyperref}
\usepackage[capitalise]{cleveref}
\usepackage{pb-diagram,tikz-cd}

\theoremstyle{definition}
\newtheorem{thm}{Theorem}[subsection]
\newtheorem{prop}[thm]{Proposition}
\newtheorem{cor}[thm]{Corollary}
\newtheorem{con}[thm]{Conjecture}
\newtheorem{lem}[thm]{Lemma}

\newtheorem{ex}[thm]{Example}

\numberwithin{equation}{subsection}

\def\cal#1{\text{$\mathcal{#1}$}}
\def\ord#1^#2{#1$^{\text{#2}}$}
\def\lie#1{\mathfrak{#1}}
\def\tlie#1{\tilde{\mathfrak{#1}}}

\def\wt{{\rm wt}}

\def\supp{{\rm supp}}

\def\qch{{\rm qch}}

\def\opl_#1^#2{\text{\scriptsize$\bigoplus\limits_{\text{\normalsize$#1$}}^{\text{\normalsize$#2$}}$}}
\def\otm_#1^#2{\text{\scriptsize$\bigotimes\limits_{\text{\footnotesize$#1$}}^{\text{\footnotesize$#2$}}$}}
\def\wcal#1{{\mbox{$\widetilde{\cal #1}$}}}

\def\bs#1{\boldsymbol{#1}}
\def\endd{\hfill$\diamond$}

\begin{document}

%\title[A graph approach to prime simple modules I]{A graph approach to prime simple modules\\ for quantum affine Algebras I \\ \hfil \\ General highest-$\ell$-weight criteria and  \\ totally ordered graphs for type $A$}
\title[On the primality   of totally ordered $q$-factorization graphs]{On the primality  of totally ordered\\ $q$-factorization graphs}
\author[A. Moura and C. Silva]{Adriano Moura and Clayton  Silva}

\address{Departamento de Matemática, Universidade Estadual de Campinas, Campinas - SP - Brazil, 13083-859.}
\email{aamoura@ime.unicamp.br}
\email{ccris22@gmail.com}

%\begin{comment}
\thanks{This work was developed as part of the Ph.D. project of the second author, which was supported by a PICME grant.
	The work of the first author was partially supported by CNPq grant 304261/2017-3 and Fapesp grant 2018/23690-6. }

\begin{abstract}
We introduce the combinatorial notion of a $q$-fatorization graph intended as a tool to study and express results related to the  classification of prime simple  modules for quantum affine algebras. These are directed graphs equipped with three decorations: a coloring and a weight map on vertices, and an exponent map on arrows (the exponent map can be seen as a weight map on arrows). Such graphs do not contain oriented cycles and, hence, the set of arrows induces a partial order on the set of vertices. In this first paper on the topic, beside setting the theoretical base of the concept, we establish several criteria for deciding whether or not a tensor product of two simple modules is a highest-$\ell$-weight module and use such criteria to prove, for type $A$, that a simple module whose $q$-factorization graph has a totally ordered vertex set is prime.
\end{abstract}
%\end{comment}

\maketitle

\section{Introduction}

The simple finite-dimensional modules for an affine Kac-Moody algebra  $\tilde{\mathfrak{g}}$ were classified by Chari and Pressley \cite{cha:intrep,cp:loopgroups} in terms of tensor products of simple evaluation modules, which  are built from simple finite-dimensional $\mathfrak{g}$-modules. Moreover, the factorization of such simple $\tilde{\mathfrak{g}}$-modules in terms of evaluation modules is unique, up to permutation of the factors. In fact, the finite-dimensional simple evaluation modules are exactly the finite-dimensional prime simple $\tilde{\mathfrak{g}}$-modules, that is, those that cannot be factored as a non-trivial tensor product.

As in the classical case, the simple finite-dimensional modules for the associated Drinfeld-Jimbo quantum group $U_q(\tilde{\mathfrak{g}})$ were also classified by Chari and Pressley \cite{cp:qaa,cp:book}. However, in this context, the classification was described in terms of their highest-$\ell$-weights (or Drinfeld polynomials), with no mention to prime simple modules, except in the case the underlying finite-dimensional simple Lie algebra $\lie g$ is of type $A_1$. In that case, the simple prime modules are again evaluation modules and every simple module can be uniquely expressed as a tensor product of prime ones (up to reordering). Thus, the question about finding a description of the simple modules in terms of tensor products of prime ones beyond rank one has intrigued the specialists since the early days of the study of the finite-dimensional representation theory of quantum affine algebras. The situation is indeed much more complicated since evaluation modules exist only for type $A$ but, even in that case, it is known \cite{cp:fact} that there are  prime simple modules which are not evaluation modules. The classification of prime simple  modules remains open after more than three decades since the early works on the topic.

As further studies were made, several examples of families of  prime simple modules started to appear in the literature such as the Kirillov-Reshetikhin (KR) modules or, more generally, minimal affinizations \cite{cha:minr2}, and certain snake modules \cite{muyou:path} (which contain the examples in the aforementioned \cite{cp:fact}). However, the most important advent related to this topic was a theory introduced by Hernandez and Leclerc \cite{hele:cluster} connecting the finite-dimensional representations of quantum affine algebras to cluster algebras. In particular, as a consequence of their main conjecture, in principle, all real prime simple modules can be computed using the machinery of cluster mutations since they correspond to the cluster variables of certain explicitly prescribed cluster algebras. A real module is a simple module whose tensor square is also simple (only the trivial module is real in the classical setting, but they abound in the quantum setting).  However, describing all cluster variables is not exactly a simple task in general. The combinatorics of cluster mutations for type $A$ was rephrased in \cite{cdfl} in tableau-theoretic language and the resulting algorithm could be used to produce examples of prime, real, and non-real modules. Most of the HL conjecture was proved for simply laced $\lie g$ in \cite{qin}, which built up on \cite{nak:cluster}. We refer to the survey \cite{hele:clustersur} for an account on the status of the conjecture and the related literature. The papers \cite{bc,bcm,cdm:DemD,dll:clusna} have explicitly identified prime modules in certain HL subcategories and provided alternate proofs for parts of the HL conjecture. See also \cite{baku,kkop:moncat2,naoi:sw} for  recent developments related to HL subcategories as well as \cite{cmy} for a study of primality from a homological perspective. 

The motivation for the present work is the problem of classifying the Drinfeld polynomials whose associated simple modules are prime and similarly for real modules. We will focus here on the former, leaving our first answers regarding the latter to appear in \cite{mosi}. 
As it is clear from the above considerations, this is a difficult problem, so our goal is to gradually obtain general results towards such classification. In this sense, most of the original results of the present paper consist of criteria for deciding whether certain tensor products are highest-$\ell$-weight modules or not. We use such criteria for proving the main result of the present paper, \Cref{t:toto}, as well as the main results of \cite{mosi}. Such criteria allowed us to expand the number of examples  of families of prime and real simple modules compared to the existing literature.  In particular, they recover the primality and reality of minimal affinizations for all types and, for type $A$, the primality of snake modules arising from prime snakes, skew representations, and certain minimal affinizations by parts.

In order to describe Drinfeld polynomials which correspond to simple prime modules in an efficient manner, we propose a graph theoretical language based on the notion of $q$-factorization. The notion of $q$-factorization is already present in the literature and is based on the solution of this classification for $\lie g$ of rank one. More precisely, for each simple root of $\lie g$, one considers the subalgebra of $U_q(\tlie g)$ generated by the corresponding loop-like generators and then the associated restriction of the Drinfeld polynomial. Since this subalgebra is of type $A_1^{(1)}$, this restricted polynomial can then be factorized according to the decomposition of the associated simple module as a tensor product of prime modules. Each of these factors is said to be a $q$-factor of the original Drinfeld polynomial $\bs\pi$. Using the $q$-factorization, we define a decorated oriented graph $G(\bs\pi)$ which we call the $q$-factorization graph of $\bs\pi$. The set of vertices of $G(\bs\pi)$ is the multiset of $q$-factors ($q$-factors with multiplicities give rise to as many vertices). Each vertex is given two decorations: a ``color'' (the simple root which originated the vertex) and a ``weight'' (the degree of the polynomial).
Given two vertices $\bs\omega$ and $\bs\omega'$, $G(\bs\pi)$ contains the arrow 
\begin{equation*}
	\begin{tikzcd}
		\bs\omega \arrow[r] & \bs\omega'
	\end{tikzcd} 
\end{equation*}
if and only if the tensor product $L_q(\bs\omega)\otimes L_q(\bs\omega')$ of the associated simple modules is reducible and highest-$\ell$-weight. Since $L_q(\bs\omega)$ is a KR module for every vertex $\bs\omega$ and KR modules are real, it follows that  $G(\bs\pi)$ has no loops.  Moreover, if a tensor product of KR modules is not highest-$\ell$-weight, the tensor product in the opposite order is. Hence, the determination of the arrows is equivalent to the solution of the problem of classifying the reducible tensor products of KR modules. Such classification gives rise to a decoration for the arrows:  a positive integer which we call the exponent of the arrow. We recall that the roots of the polynomial $\bs\omega$ form a $q$-string. Let us say $a$ is the center of such string and similarly let $a'$ be the center of the string associated to $\bs\omega'$. Let us say that $\bs\omega$ is $i$-colored and has weight $r$ while $\bs\omega'$ is $j$-colored and has weight $s$. Then, there exists a finite set of positive integers $\mathscr R_{i,j}^{r,s}$ such that  $L_q(\bs\omega)\otimes L_q(\bs\omega')$ is reducible and highest-$\ell$-weight if and only if $a=a'q^m$ for some $m\in\mathscr R_{i,j}^{r,s}$. This number $m$ is then defined to be the exponent of the arrow.  We visually express this set of data by the picture
\begin{equation*}
	\begin{tikzcd}
		\stackrel{r}{i} \arrow[r,"m"] & \stackrel{s}{j}
	\end{tikzcd} 
\end{equation*}
If $G(\bs\pi)$ is connected, this data determines $\bs\pi$ uniquely up to uniform shift of all centers. Primeness and reality of the underlying simple modules are independent of such shift. Thus, the classification of prime simple  modules can be rephrased as a classification of such decorated graphs. For instance, the result for $\lie g$ of type $A_1$ can be phrased as: $L_q(\bs\pi)$ is prime if and only if $G(\bs\pi)$ has a single vertex. Also, for general $\lie g$, if $G(\bs\pi)$ has two vertices, then $L_q(\bs\pi)$ is prime if and only if $G(\bs\pi)$ is connected. This is not true in general: although $G(\bs\pi)$ is connected if $L_q(\bs\pi)$ is prime (\Cref{p:primeconect}), the converse is far from true. Henceforth, we say $G(\bs\pi)$ is prime  if $L_q(\bs\pi)$ is prime. We remark that, by definition, $G(\bs\pi)$ has no oriented cycles and, therefore, the structure of arrows induce a natural partial order on the set of vertices of $G(\bs\pi)$. For instance, in the above picture, $\bs\omega\succ\bs\omega'$.

A precise description of the elements belonging to $\mathscr R_{i,j}^{r,s}$ can be read off the results of \cite{ohscr:simptens} for nonexceptional $\lie g$ as well as for type $G$. Some of our results were proved without using such precise description and, hence, they are proved for all types. For instance, the main result of \cite{cp:fact} describes a family of prime simple modules for type $A_2$. In the graph language that we are introducing here, this can be simply described by saying that  $G(\bs\pi)$ is prime if it is an oriented line (all arrows in the same direction):
\begin{equation*}
	\begin{tikzcd}
		\circ \arrow[r] & \circ \arrow[r] & \cdots 
	\end{tikzcd} 
\end{equation*}
In \Cref{gencp}, we prove that this is true for all $\lie g$. Even for type $A_2$, this does not cover all prime simple modules. For instance, one of the main results we present in \cite{mosi} characterize all nonoriented lines with three vertices which are prime for type $A$.  Another fact we prove here for all $\lie g$ concerns the case that $G(\bs\pi)$ is a tree, i.e., there are no (nonoriented) cycles. In that case, we prove that, if $G(\bs\pi)$ is prime, then every connected subgraph of $G(\bs\pi)$ is also prime.  This not true if $G(\bs\pi)$ is not a tree, and we give a counter example in \cite{mosi}, which is a paper dedicated to the study of several results concerning trees. 

Beside the collection of criteria for deciding whether certain tensor products are highest-$\ell$-weight modules or not, the main result of the present paper (\Cref{t:toto}) states that, if $\lie g$ is of type $A$, $L_q(\bs\pi)$ is prime if $G(\bs\pi)$ is a totally ordered graph, i.e., the partial order on the set of vertices is a total order. In particular, this is the case if $G(\bs\pi)$ is a tournament, i.e., if any pair of vertices is linked by an arrow.  Thus, for type $A$, \Cref{t:toto} is a strong generalization of the aforementioned \Cref{gencp}. After solving the purely combinatorial problem of classifying all the totally ordered $q$-factorization graphs, \Cref{t:toto} would then provide an explicit family of simple prime modules. We do not address this combinatorial problem here beyond type $A_2$, restricting ourselves to presenting a family of examples of $q$-factorization graphs  with arbitrary number of vertices for type $A$ which are  afforded by tournaments in \Cref{ex:tour}. For type $A_2$, \Cref{p:totopartial} implies that a totally ordered $q$-factorization graph must be a tree and, hence, we are back to the context of \cite{cp:fact} and \Cref{gencp}.

The reason \Cref{t:toto} is proved only for type $A$ is that, differently from the proof of \Cref{gencp}, the argument used here explicitly utilizes the description of the sets  $\mathscr R_{i,j}^{r,s}$. Therefore, if the same approach is to be used for other types,  a case-by-case analysis would have to be employed. Thus, we leave the analysis for other types to appear elsewhere. 

The paper is organized as follows. In \Cref{ss:graph}, we review the basic terminology and notation about directed graphs which we shall use, while in \Cref{ss:cuts} we recall the concept of cuts of a graph as well as the definitions of special types of graphs such as trees and tournaments. The basic notation about classical and quantum affine algebras is fixed in \Cref{ss:clalg}, while the notions of Drinfeld polynomials, $\ell$-weights and $q$-factorization is reviewed in \Cref{ss:lwl}. This is sufficient to formalize the first part of the definition of $q$-factorization graphs. Thus, in \Cref{ss:prefact}, we define the concept of pre-factorization graph. 
\Cref{ss:hopf} closes  \Cref{s:prel} by collecting some basic general facts about Hopf algebras and their representations theory.

The second part of the definition of $q$-factorization graphs, given in \Cref{ss:fact}, concerns the sets $\mathscr R_{i,j}^{r,s}$, which are explained, alongside the definition of prime modules, in \Cref{ss:primem}. The required representation theoretic background for these subsections is reviewed in Sections \ref{ss:repth} and \ref{ss:tpdu}. The statements of our main results and conjectures are presented in \Cref{ss:main}, while \Cref{ss:ex} brings a few illustrative examples such as the aforementioned family of tournaments. The other two examples interpret the notions of snake and skew modules from the perspective of $q$-factorization graphs. 

\Cref{s:hlwc} brings the statements and proofs of the several criteria for deciding whether certain tensor products are highest-$\ell$-weight modules or not. Its several subsections split them by the nature of the statements. Perhaps it is worth calling attention to those criteria which are most  used or play more crucial roles in the proof of \Cref{t:toto} as well as in the proofs of the main results from \cite{mosi}: \Cref{l:hlwquot},  \Cref{p:killhlw}, and \Cref{p:killdualhlw}. 

\Cref{ss:toto} is completely dedicated to the proof of \Cref{t:toto}. In \Cref{ss:removeA}, we continue the trend initiated in \Cref{ss:remove}, this time with specific conditions for type $A$. Namely, these sections establish criteria for removing a vertex from a $q$-factorization graph in a way that its tensor product with another graph remains associated to a highest-$\ell$-weight tensor product. We continue  by  collecting a few technical lemmas concerned with arithmetic relations among the elements of $\mathscr R_{i,j}^{r,s}$ in \Cref{ss:arith}.  The key technical part of the proof of \Cref{t:toto} is \Cref{l:totordopparrow}. All the criteria are then brought together to finalize the proof in \Cref{ss:ptoto}.

\section{Preliminaries}\label{s:prel}

Throughout the paper, let $\mathbb C$ and  $\mathbb Z$ denote the sets of complex numbers and integers, respectively. Let also $\mathbb Z_{\ge m} ,\mathbb Z_{< m}$, etc. denote the obvious subsets of $\mathbb Z$. Given a ring $\mathbb A$, the underlying multiplicative group of units is denoted by $\mathbb A^\times$. 
The symbol $\cong$ means ``isomorphic to''. We shall use the symbol $\diamond$ to mark the end of remarks, examples, and statements of results whose proofs are postponed. The symbol \qedsymbol\ will mark the end of proofs as well as of statements whose proofs are omitted.

\subsection{Directed Graphs}\label{ss:graph} In this section we fix notation regarding the basic concepts of graph theory.

A directed graph is a pair $G = (\mathcal V_G,\mathcal A_G)$ where $\mathcal V_G$ is a set and $\mathcal A_G$ is a subset of $\mathcal V_G\times \mathcal V_G$ such that
$$(v,v')\in \mathcal A_G \ \Rightarrow\ (v',v)\notin \mathcal A_G.$$
We will typically simplify notation and write $\mathcal V$ and $\mathcal A$ instead of $\mathcal V_G$ and $\mathcal A_G$.
An element of $\mathcal V$ is called a vertex and an element $(v,v')$ of $\mathcal A$ is called an arrow from $v$ to $v'$. We shall also say $v'$ is the head of the arrow $(v,v')$ while $v$ is its  tail. Given $a\in \mathcal A$, we write $t_a$ for its tail end $h_a$ for its head.  
As usual, the picture
\begin{equation*}
	\begin{tikzcd}
		v \arrow[r] & v'
	\end{tikzcd} 
\end{equation*}
will mean that $(v,v')\in \mathcal A$. A loop in $G$ is an element $a\in \mathcal A$ such that  $t_a=h_a$. We will only consider graphs with no loops, so, henceforth, this is implicitly assumed. We also assume $G$ is finite, i.e., $\mathcal V$ is a finite set.

Given a subset $\mathcal V'$ of $\mathcal V$, the subgraph $G'=G_{\mathcal V'}$ of $G$ associated to $\mathcal V'$ is the pair $(\mathcal V',\mathcal A')$ with $$\mathcal A'=\{a\in \mathcal A:t_a,h_a\in \mathcal V'\}.$$ In terms of pictures, $G_{\mathcal V'}$ is obtained from $G$ by deleting the elements of $\mathcal V\setminus \mathcal V'$ as well as all the arrows starting at or heading to an element of  $\mathcal V\setminus \mathcal V'$. It will often be convenient to write $G\setminus\mathcal V'$ instead of $G_{\mathcal V'}$.

Let $\mathscr P(\mathcal V)$ be the power set of $\mathcal V$ and $\pi:\mathcal A\to \mathscr P(\mathcal V)$ be given by $\pi(a)=\{t_a,h_a\}$. The (non-directed) graph associated to $G$ is the pair $(\mathcal V,\mathcal E)$ where $\mathcal E=\pi(\mathcal A)$. The elements of $\mathcal E$ will be referred to as edges. By a  (non-directed) path  of length $m\in\mathbb Z_{\ge 0}$ in $G$, we mean  a sequence $\rho=e_1,\dots,e_m$ of edges in $\mathcal E$ such that  
\begin{equation*}
	\#(e_j\cap e_{j+1})=1  \quad\text{for all}\quad 1\le j<m \quad\text{and}\quad e_{j-1}\cap e_j\cap e_{j+1}=\emptyset  \quad\text{for all}\quad 1< j<m .
\end{equation*}
This is equivalent to saying that there exists an underlying sequence of vertices $v_1,\dots,v_{m+1}$ such that $e_j = \{v_j,v_{j+1}\}$ for all $1\le j\le m$. This sequence is unique if $m>1$. If $v_1=v_{m+1}$, we say $\rho$ is a cycle based on $v_1$. In that case, if $m=\min\{j>1: v_j=v_1\}$, we say $\rho$ is an $m$-cycle. Note there does not exist $m$-cycles for $m\le 2$.

We shall often write $\rho =e_1\cdots e_m$ instead of  $\rho=e_1,\dots,e_m$ and set $\ell(\rho)=m$. We also write $e\in\rho$ to mean that $e=e_j$ for some $1\le j\le m$. 
Suppose  $\rho'=e_1'\dots e_{m'}'$ is another path such that $e_m\cap e'_1\ne \emptyset $ and either
\begin{equation*}
	e_m = e'_1 \quad\text{or}\quad e_{m-1}\cap e_m\cap e'_1 = \emptyset = e_m\cap e'_1\cap e'_2. 
\end{equation*}
Then, the sequence obtained from $e_1\cdots e_me'_1\cdots e'_{m'}$ after successive deletion of any appearance of a substring of the form $ee, e\in\mathcal E$, is a path which we denote by $\rho*\rho'$. The path $\rho^- := e_m\cdots e_1$ will be referred to as the reverse path of $\rho$. In particular, $\rho*\rho^-$ is the empty sequence.

If $m=\ell(\rho)>1$, 
\begin{equation*}
	v\in e_1\setminus e_2, \qquad\text{and}\qquad v'\in e_m\setminus e_{m-1},  
\end{equation*}
we say $\rho$ is a is a path from $v$ to $v'$. If $\ell(\rho)=1$, say, $\rho=e_1=\pi(a)$ for some $a\in\mathcal A$,  $\rho$ can be regarded as a path from $t_a$ to $h_a$ and vice versa.  We let $\mathscr P_{v,v'}$ be the set of all paths from $v$ to $v'$ and $\mathscr P_G$ be the set of all paths in $G$.
If $\rho\in \mathscr P_{v,v'}$ and $\rho'\in \mathscr P_{v',v''}$,  then $\rho*\rho'\in \mathscr P_{v,v''}$.

A subpath $\rho'$ of $\rho$ is subsequence such that 
\begin{equation*}
	e_i,e_j\in\rho' \quad\text{with}\qquad i<j \qquad\Rightarrow\qquad e_k\in\rho' \quad\text{for all}\quad i\le k\le j.
\end{equation*}
We say $\rho$ is a simple path if no  subpath is a cycle. If $\rho=e_1\cdots e_m$ is a path from $v$ to $v', e_j=\pi(a_j)$, and $m>1$, the signature of $\rho$ is the element $\sigma_\rho=(s_1,\dots,s_m)\in\mathbb Z^m$ given by 
\begin{equation*}
	s_1 = \begin{cases} -1,& \text{if } v=t_{a_1},\\ 1,& \text{if } v=h_{a_1},\end{cases} \qquad\text{and}\qquad 
	s_{j+1} = \begin{cases} s_j,& \text{if } t_{a_{j+1}}=h_{a_j} \text{ or } t_{a_j}=h_{a_{j+1}},\\ -s_j,& \text{otherwise,}\end{cases} 
\end{equation*}
for all $1\le j<m$.  If $m=1$, the signature will be $1$ or $-1$ depending on whether it is regarded as a path from $h_{a_1}$ to $t_{a_1}$ or the other way round, respectively. We shall say $\rho$ is monotonic or directed if $s_i=s_j$ for all $1\le i,j\le m$. In that case, if $s_j=1$ for all $1\le j\le m$, we say it is increasing. Otherwise, it is decreasing.  If $\rho$ is increasing, we set $h_\rho = h_{a_1}$ and $t_{\rho}=t_{a_m}$. If it is decreasing, then $t_\rho = t_{a_1}$ and $h_{\rho}=h_{a_m}$. If $s_{j+1}=-s_j$ for all $1\le j<m$, we say $\rho$ is alternating.  Clearly, $\sigma_{\rho^-}=(-s_m,\dots,-s_1)$. We shall refer to a monotonic cycle as an oriented cycle.  
We will denote by $\mathscr{P}^+_{v,v'}$ (resp. $\mathscr{P}^-_{v,v'}$) be the set of increasing (resp. decreasing) monotonic paths from $v$ to $v'$. For instance
\begin{equation*}
	\begin{tikzcd}
		v_1 \arrow[r,"a_1"] & v_2 \arrow[r,"a_2"] & v_3 \in \mathscr{P}^-_{v_1,v_3}
	\end{tikzcd} \quad\text{while}\quad
	\begin{tikzcd}
		v_1 & \arrow[swap,l,"a_1"]  v_2 & \arrow[swap,l,"a_2"]  v_3 \in \mathscr{P}^+_{v_1,v_3}.
	\end{tikzcd} 
\end{equation*}
On the other hand, 
\begin{tikzcd}
	v_1  \arrow[r,"a_1"] & v_2 & \arrow[swap,l,"a_2"]  v_3 \in \mathscr{P}{v_1,v_3},
\end{tikzcd} 
but is neither in $\mathscr{P}^+_{v_1,v_3}$ nor in $\mathscr{P}^-_{v_1,v_3}$.

A graph $G$ is said to be connected if, for every pair of vertices $v\ne v'$, there exists a path from $v$ to $v'$. 
If $G$ is connected, we can consider the distance function $d:\mathcal V\to\mathbb Z$ defined by, $d(v,v)=0$ for all $v\in \mathcal V$ and
\begin{equation*}
	d(v,v') = \min\{\ell(\rho): \rho\text{ is a path from } v \text{ to }v'\} \qquad\text{if}\qquad v\ne v'.
\end{equation*}
If $d(v,v')=1$ we say $v$ and $v'$ are adjacent.  Also, for two subsets $\mathcal V_1,\mathcal V_2\subseteq \mathcal V$, define
\begin{equation*}
	d(\mathcal V_1,\mathcal V_2) = \min\{d(v_1,v_2):v_1\in \mathcal V_1, v_2\in \mathcal V_2\}.
\end{equation*}
Set $d(v,v')=\infty$ if $v$ and $v'$ belong to distinct connected components.

\begin{ex}
	The following path $\rho=e_1\cdots e_5$ from $v$ to $v'$ has signature $(1,-1,-1,1,-1)$ and contains the $3$-cycle $e_2e_3e_4$. The circles denote other arbitrary elements in $\mathcal V$.
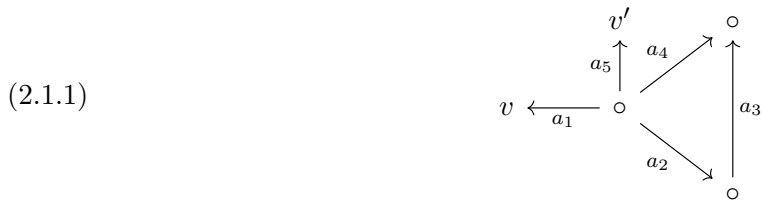
\begin{equation}\label{e:pathex}
	\begin{tikzcd}
		 & v'  &  \circ \\
		v &  \arrow[l,"a_1"]  \arrow[u,"a_5"] \circ \arrow[ur,"a_4"] \arrow[swap,dr,"a_2"]\\
		    &   & \circ \arrow[swap,uu,"a_3"]
	\end{tikzcd} 
\end{equation}
Note $d(v,v')\le 2$ since $a_1a_5$ is a path from $v$ to $v'$. The subpath $e_2e_3$ is decreasing, while $e_3e_4e_5$ is an alternating subpath. The path $e_1e_5$ is alternating while $e_3e_2$ is increasing, but they are not subpaths of $\rho$. 
\endd
\end{ex}

Every path gives rise to a subgraph associated to the set
\begin{equation*}
	\mathcal V^\rho = \{v\in \mathcal V: v\in e\text{ for some } e\in \rho \}.
\end{equation*}
Given $v\in \mathcal V$, set $\mathcal A_v=\{v'\in\mathcal V:d(v,v')=1\}$, 
\begin{equation}\label{e:adjset}
	\mathcal A_v^1 = \{v'\in\mathcal A_v: (v',v)\in\mathcal A \}, \qquad\text{and}\qquad \mathcal A_v^{-1} = \{v'\in\mathcal A_v: (v,v')\in\mathcal A \}.
\end{equation}
The valence of $v$ is defined as $\#\mathcal A_v$. If this number is $0$, we say $v$ is an isolated vertex, if it is $1$, we say $v$ is monovalent, and if it is at least $3$, we say $v$ is multivalent.   Set
\begin{equation*}
	\mathring{G} = \{v\in\mathcal V: \#\mathcal A_v>1\} \qquad\text{and}\qquad \partial G = G\setminus\mathring{G}.
\end{equation*}
Elements of $\partial G$ will be referred to as boundary vertices while those of $\mathring{G}$ will be referred to as inner vertices. A vertex $v$ is said to be a source if there are no incoming arrows towards it or, equivalently,
%it is not isolated and
\begin{equation*}
	\mathcal A_v \subseteq \mathcal A_v^{-1},
	%   (v,v')\in\mathcal A \qquad\text{for all}\qquad v'\in\mathcal A_v.
\end{equation*}
while it is a sink if
\begin{equation*}
	\mathcal A_v \subseteq \mathcal A_v^1.
	%(v',v)\in\mathcal A \qquad\text{for all}\qquad v'\in\mathcal A_v
\end{equation*}
In particular, isolated vertices are sinks and sources at the same time and a non-isolated vertex cannot be a sink and a source concomitantly. We will say a vertex  is extremal if it is either a sink or a source. Note the middle circle in \eqref{e:pathex} is a source, the upper one is a sink, and the lower one is neither.

\subsection{Cuts and Special Kinds of Graphs}\label{ss:cuts} A cut of a directed graph $G$ is a pair of subgraphs $(G',G'')$ such that
\begin{equation*}
	\mathcal V = \mathcal V'\sqcup\mathcal V'', \quad \mathcal A'=\{a\in\mathcal A: h_a,t_a\in\mathcal V'\}, \quad\text{and}\quad \mathcal A''=\{a\in\mathcal A: h_a,t_a\in\mathcal V''\}.
\end{equation*}
The set
\begin{equation*}
	\mathcal A\setminus  (\mathcal A'\cup\mathcal A'')
\end{equation*}
is called the associated cut-set. Note the cut can be recovered from its cut-set if $G$ is connected. Elements of the cut-set are said to cross the cut. An element  $a\in\mathcal A$ is said to be a bridge if the number of connected components of $(\mathcal V,\mathcal A\setminus\{a\})$ is larger than that of $G$. If $G$ is connected, this is equivalent to saying that $\{a\}$ is the cut-set of a cut. We shall say a cut $(G',G'')$ is connected if both $G'$ and $G''$ are connected.

A connected graph with no cycles is said to be a tree. We shall refer to a tree with no multivalent vertex as a line.
We will say $G$ is a monotonic line if $\mathcal V=\mathcal V^\rho$ for some simple monotonic path $\rho$. 
Note every tree with more than one vertex has at least two monovalent vertices, a fact which is false in general, as seen in the following examples. 
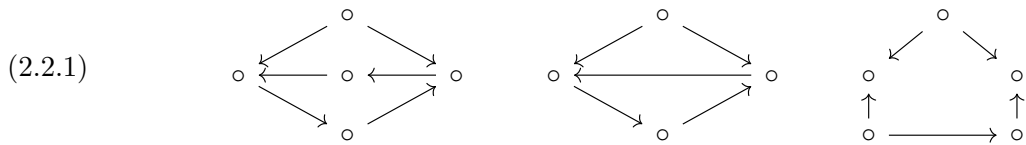
\begin{equation}\label{e:linefus}
	\begin{tikzcd}[row sep=small]
	&	\circ \arrow[dl]\arrow[dr]&  \\ 
	\circ \arrow[dr]& \arrow[l] \circ& \arrow[l]\circ\\
	& \circ \arrow[ur]& 
	\end{tikzcd} \qquad
	\begin{tikzcd}[row sep=small]
		&	\circ \arrow[dl]\arrow[dr]&  \\ 
		\circ \arrow[dr]& & \arrow[ll]\circ\\
		& \circ \arrow[ur]& 
	\end{tikzcd} \qquad
	\begin{tikzcd}[column sep=small, row sep=small]
		&	\circ \arrow[dl]\arrow[dr]&  \\
		\circ & & \circ\\
		\circ\arrow[u] \arrow[rr]& & \circ \arrow[u]
	\end{tikzcd}
\end{equation}
Note that, in the first two graphs, the subgraphs obtained by removing the upper vertex are formed by directed cycles, while the cycle corresponding to the subgraphs obtained by removing the lower vertex are not. Note also that an arrow $a$ is a bridge if and only if it is not contained in a cycle. In particular, the above graphs are bridgeless. A forest is a graph whose connected components are trees or, equivalently, every arrow is a bridge. We also recall that a tournament is a graph whose underlying set of edges is complete, i.e., $\{v,v'\}\in\mathcal E$ for every $v,v'\in\mathcal V, v\ne v'$. In that case, the underlying non-directed graph is said to be complete. None of the above graphs is a tournament, but the middle one is missing only one arrow to become a tournament. 

Let us record some elementary properties of trees. 

\begin{lem}\label{l:disctree}
	The following are equivalent for a graph $G$. 
	\begin{enumerate}[(i)]
		\item $G$ is a tree.
		\item  $G$ is connected and the graph obtained by removing any edge has two connected components.
		\item $\#\mathscr P_{v,v'}=1$ for all vertices $v,v'\in G$. \hfil\qed
	\end{enumerate}	
\end{lem}

%\begin{proof}
%	If $G$ is a tree, then $G$ is connected. So let $e=\{a,b\}$ be any edge of $G$.
%	Then if $G-e$ is connected, there are at least two paths in $G$ from $a$ and $b$.
%	Hence $G-e$ is disconnected and so the vertices in $G-e$ may be partitioned into two
%	subsets:
%	\begin{itemize}
	%		\item[i)] vertex $a$ and those vertices that can be reached from $a$ by a path in $G-e$;
	%		\item[ii)] vertex $b$ and those vertices that can be reached from $b$ by a path in $G-e$.
	%	\end{itemize}
%	There two connceted components are trees because a cycle in either component
%	would also be in $G$. \\
%	
%	On the other hand, if $G-e$ has two connected components, then $G$ is connected. Since the two connected components of $G-e$ are trees, if $G$ contains a cycle, then $e=\{a,b\}$ must be an edge of the cycle. But then $G-e$ is connected, which is a contradiction.  \\
%	
%\end{proof}

In light of (iii) of the above lemma, given vertices $v,v'$ in a tree, we denote by $[v,v']$ the set of vertices of the unique element of $\mathscr P_{v,v'}$. In particular, $[v',v]=[v,v']$. Evidently, if $v\ne v'$,   $\#([v,v']\cap\mathcal A_v)= 1$.
Given $m\in\mathbb Z_{> 0}$, set
\begin{equation}
	\mathcal A_v^{\pm m}=\{v'\in\mathcal V: d(v,v')=m \text{ and } [v,v']\cap\mathcal A_{v'}^{\mp 1} \neq\emptyset\}.
\end{equation}
This clearly coincides with the sets defined in \eqref{e:adjset} when $m=1$.   Set also $\mathcal A_v^0=\{v\}$ and
\begin{equation}
	\mathcal A_v^\pm = \bigcup_{m\in\mathbb Z_{>0}} \mathcal A_v^{\pm m}.
\end{equation}

\begin{lem}\label{l:elemtree}
	Assume $G$ is a tree.
	
	\begin{enumerate}[(a)]
		\item $\partial G\ne\emptyset, \#\partial G=1$ iff $G$ is a singleton, and $\#\partial G=2$ iff $G$ is a nontrivial path.
		\item If $H$ is a subgraph, then $H$ is a tree. Moreover, if $H$ is connected and proper, $\partial G\setminus \mathcal V_H\ne\emptyset$.
		\item  If $H$ is a connected subgraph and $k=\#\mathcal V_G-\#\mathcal V_H$, there exist $v_1,\dots,v_k\in\mathcal V_G$ such that
		$v_j\in\partial (G\setminus \{v_i:i<j\})$ and $H=G\setminus\{v_i:1\le i\le k\}$.
		\item If $\mathcal V_1\cup\mathcal V_2$ is a nontrivial partition of $\mathcal V_G$ such that $G_{\mathcal V_i}$ is connected for $i=1,2$, there exists unique $(v_1,v_2)\in\mathcal V_1\times\mathcal V_2$ such that $d(v_1,v_2)=1$. 
		\item For all $v\in\mathcal V$, the sets $\mathcal A_v^m, m\in\mathbb Z$ are disjoint and $\mathcal V = \mathcal A_v^+\cup\mathcal A_v^0\cup\mathcal A_v^-$.\hfil\qed
	\end{enumerate}
\end{lem}

We will be interested in graphs with no oriented cycles. In that case, the set of arrows $\mathcal A$  induces a partial order on $\mathcal V$ by the transitive extension of the strict relation
\begin{equation*}
	h_a\prec t_a \quad\text{for}\quad a\in\mathcal A.
\end{equation*} 
Note
\begin{equation}
	\mathcal P_{v,v'}^+ \ne\emptyset \ \Leftrightarrow\ v \prec v' \quad\text{and}\quad  \mathcal P_{v,v'}^- \ne\emptyset   \ \Leftrightarrow\ v' \prec v.
\end{equation}
Set $D(v,v')=0$ if $v=v'$,
\begin{equation}
	D(v,v') = \min\{\ell(\rho): \rho\in\mathscr{P}^+_{v,v'}\} \ \text{if}\  v \prec v', \quad D(v,v') = -\min\{\ell(\rho): \rho\in\mathscr{P}^-_{v,v'}\} \ \text{if}\  v' \prec v,
\end{equation}
and $D(v,v')=\infty$ if $v$ and $v'$ are not comparable by $\preceq$. Given $m\in\mathbb Z$, set
\begin{equation}\label{e:Neibor}
	\mathcal N^m_G(v) = \{v\in\mathcal V: D(v,v')=m\} \quad\text{and}\quad \mathcal N^\pm_G(v) = \bigcup_{m\in\mathbb Z_{>0}} \mathcal N_G(v)^{\pm m}.
\end{equation}
If no confusion arises, we simplify notation and write $\mathcal N^m(v)$ and $\mathcal N^\pm(v)$.

We shall say $G$ is a totally ordered graph if $\preceq$ is a total order on $\mathcal V$. The following lemma is easily established.

\begin{lem}\label{l:totop}\hfil
	\begin{enumerate}[(a)]
		\item Every totally ordered graph is connected  and has a unique sink and a unique source.
		\item If $G$ is a  totally ordered graph and $v\in\mathcal V$ is an extremal vertex, the subrgraph associated to $\mathcal V\setminus\{v\}$ is also totally ordered.
		\item A totally ordered tree is an monotonic line.
		\item Every tournament with no oriented cycles is totally ordered.\qed
	\end{enumerate}
\end{lem}

Only the last graph in \eqref{e:linefus} does not contain an directed cycle so $\preceq$ is defined, but it is not totally ordered. The following are examples of totally ordered graphs:
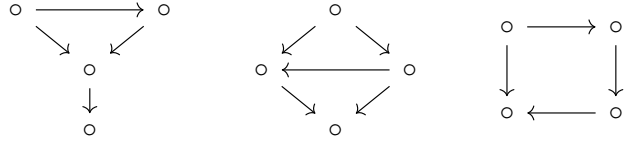
\begin{equation*}
	\begin{tikzcd}[column sep=small,row sep=small]
			\circ \arrow[rr]\arrow[dr] & & \circ \arrow[dl] \\ 
		&  \circ \arrow[d]& \\
		& \circ & 
	\end{tikzcd} \qquad
	\begin{tikzcd}[column sep=small,row sep=small]
		&	\circ \arrow[dl]\arrow[dr]&  \\ 
		\circ \arrow[dr]& & \arrow[ll]\circ \arrow[dl]\\
		& \circ & 
	\end{tikzcd} \qquad
	\begin{tikzcd}[column sep=small,row sep=small]
		\circ \arrow[rr] \arrow[dd]& & \circ \arrow[dd]\\ \\
		\circ & & \circ \arrow[ll]
	\end{tikzcd}
\end{equation*}

\subsection{Classical and Quantum Algebras}\label{ss:clalg}
Let $I$ be the set of nodes of a  finite-type connected Dynkin diagram. By regarding $I$ as the set of vertices of the undirected graph whose  edges are the sets of adjacent nodes of the diagram, we can use the notions of graph theory from the previous sections. By abuse of language, we refer to any subset $J$ of $I$ as a subdiagram (subgraph). In particular,  we have defined $d(i,j)$ and $[i,j]$ for all $i,j\in I$ as well as $\partial J$ and $\mathring{J}$ for any $J\subseteq I$. Let also $\bar J$ be the minimal connected sudiagram of $I$ containing $J$. This is well defined since $I$ is a tree.

Let $\lie g$ be the simple Lie algebra over $\mathbb C$ corresponding to the given Dynkin diagram, fix a Cartan subalgebra $\lie h$ and a set of positive roots $R^+$ and let
$\lie g_{\pm \alpha},\alpha\in R^+$, and $\lie g=\lie n^-\oplus\lie h\oplus\lie n^+$ be the associated root spaces and triangular decomposition.
The simple roots will be denoted by $\alpha_i$, the fundamental weights by
$\omega_i$, $i\in I$, while $Q,P,Q^+,P^+$ will denote the root and weight
lattices with corresponding positive cones, respectively.
Let also $h_\alpha\in\lie h$  be the co-root associated to $\alpha\in R^+$. If $\alpha=\alpha_i$ is simple, we often simplify notation and write
$h_i$. Let $C = (c_{i,j})_{i,j\in I}$ be the Cartan matrix of $\lie g$, i.e., $c_{i,j}=\alpha_j(h_i)$, and $d_i, i\in I$, be such that $d_ic_{i,j}=d_jc_{j,i}, i,j\in I$. The Weyl group is denoted by $\cal W$  and its longest element by $w_0$.  We also denote by $w_0$ the involution on $I$ induced by $w_0$ and set $i^*=w_0(i)$. The dual Coxeter number and the lacing number of $\lie g$ will be denoted by $h^\vee$ and $r^\vee$, respectively. In particular, $r^\vee=\max\{d_i:i\in I\}$.

For a subdiagram $J\subseteq I$,  let $\lie g_J$ be the subalgebra of $\lie g$ generated by the corresponding simple root vectors, $\lie h_J=\lie h\cap\lie g_J$, and so on. Let also $Q_J$ be the subgroup of $Q$ generated by $\alpha_j, j\in J$, $Q^+_J=Q^+\cap Q_J$, and
$R^+_J=R^+\cap Q_J$. 
Given $\lambda\in P$, let $\lambda_J$ denote the
restriction of $\lambda$ to $\lie h_J^*$. 
If $\mu\in P$ and $J\subseteq I$, define also 
$${\rm supp}(\mu)=\{i\in I:\mu(h_i)\ne 0\}.$$

For a Lie algebra $\lie a$ over $\mathbb C$, let $\tlie a=\lie a\otimes  \mathbb
C[t,t^{-1}]$ be its loop algebras and identify $\lie a$ with the subalgebra $\lie a\otimes 1$.
Then, $\tlie g = \tlie n^-\oplus \tlie h\oplus \tlie
n^+$ and $\tlie h$ is an abelian subalgebra.

Let $\mathbb F$ be an algebraically closed field of characteristic zero, fix $q\in\mathbb F^\times$ which is not a root of $1$, and set $q_i=q^{d_i}, i\in I$. Let also $U_q(\lie g)$ and $U_q(\tlie g)$ be the associated Drinfeld-Jimbo quantum groups over $\mathbb F$. We use the notation as in  \cite[Section 1.2]{Moura}.  In particular, the Drinfeld loop-like generators of $U_q(\tlie g)$ are denoted by $x_{i,r}^\pm, h_{i,s}, k_i^{\pm 1}, i\in I, r,s\in\mathbb Z, s\ne 0$. Also, $U_q(\lie g)$ is the subalgebra of $U_q(\tlie g)$ generated by $x_i^\pm = x_{i,0}^\pm, k_i^{\pm 1}, i\in I$, and the subalgebras $U_q(\lie n^\pm), U_q(\lie h), U_q(\tlie n^\pm), U_q(\tlie h)$   are defined in the expected way. 

Given $J\subseteq I$, let $U_q(\lie a_J)$, with $\lie a=\lie g, \tlie g,\tlie h$, etc., be the respective quantum groups associated to $\lie a_J$. Let also $U_q(\lie a)_J$ be the the subalgebra of $U_q(\tlie g)$ generated by the generators corresponding to $J$. It is well known that there is an algebra isomorphism
\begin{equation*}
	U_q(\lie a)_J\cong U_{q_J}(\lie a_J) \qquad\text{where}\qquad q_J=q^{d_J} \qquad\text{with}\qquad d_J = \min\{d_j:j\in J\}.
\end{equation*}
This is a Hopf algebra isomorphism only if $\lie a\subseteq\lie g$. We shall always implicitly identify $U_q(\lie a)_J$ with $U_{q_J}(\lie a_J)$ without further notice. When $J=\{j\}$ is a singleton, we simply write  $U_q(\lie a)_j$ instead of $U_q(\lie a)_{\{j\}}$, and so on.

\subsection{The $\ell$-Weight Lattice}\label{ss:lwl}
The $\ell$-weight lattice of $U_q(\tlie g)$ is the multiplicative group $\mathcal P$ of $n$-tuples of rational
functions $\bs\varpi = (\bs\varpi_i(u))_{i\in I}$ with values
in $\mathbb F$  such that $\bs\varpi_i(0)=1$ for all $i\in I$. The elements of the submonoid $\mathcal P^+$ of
$\mathcal P$ consisting of $n$-tuples of polynomials will be referred to as dominant $\ell$-weights or Drinfeld polynomials. 
If $\bs\pi,\bs\omega\in\mathcal P^+$ satisfy $\bs\pi\bs\omega^{-1}\in\mathcal P^+$, we shall say $\bs\omega$ divides $\bs\pi$ and write $\bs\omega|\bs\pi$.

Given
$a\in\mathbb F^\times$ and $\mu\in P$, let $\bs{\omega}_{\mu,a}\in\mathcal P$ be the element whose $i$-th rational function is
\begin{equation*}
	(1-au)^{\mu(h_i)}, \quad i\in I.
\end{equation*}
In the case that $\mu=\omega_i$ for some $i$, we simplify notation and write $\bs\omega_{i,a}$. Since  $\mathcal P$ is a (multiplicative) free abelian group on the set $\{\bs{\omega}_{i,a}:i\in I,a\in\mathbb F^\times\}$, there exists a unique  group homomorphism  $\wt:\mathcal P \to P$ determined by setting $\wt(\bs\omega_{i,a})=\omega_i$. Set
\begin{equation*}
	\supp(\bs\varpi) = \supp(\wt(\bs\varpi)), \quad\bs\varpi\in\mathcal P.
\end{equation*}
There exists an injective map $\mathcal P\to (U_q(\tlie h))^*$ (see \cite{Moura}) and, hence, we identify $\mathcal P$  with its image in $(U_q(\tlie h))^*$. 

Given $i\in I, a\in\mathbb F^\times, m\in\mathbb Z_{\ge 0}$,  define $q_i=q^{d_i}$ and
\begin{equation*}
	\bs\omega_{i,a,r} = \prod_{p=0}^{r-1} \bs\omega_{i,aq_i^{r-1-2p}}.
\end{equation*}
Note that $\wt(\bs\omega_{i,a,r}) = r\omega_i$. % {\color{Apricot} \qquad\text{and}\qquad \wt(\bs\alpha_{i,a}) = \alpha_i. }
We shall refer to Drinfeld polynomials of the form $\bs\omega_{i,a,r}$ as polynomials of Kirillov-Reshetikhin (KR) type. 
Every Drinfeld polynomial can be written uniquely as a product of KR type polynomials such that, for every two factors supported at $i$, say $\bs\omega_{i,a,r}$ and $\bs\omega_{i,b,s}$, the following holds
\begin{equation}\label{defqfact}
	\frac{a}{b} \ne q_i^{r+s-2p} \qquad\text{for all}\qquad 0\le p<\min\{r,s\}. 
\end{equation}
Such factorization is said to be the  $q$-factorization of $\bs\pi$ and the corresponding factors are called the $q$-factors of $\bs\pi$.
By abuse of language, whenever we mention the set of $q$-factors of $\bs\pi$ we actually mean the associated multiset of $q$-factors counted with multiplicities in the $q$-factorization. We shall say that $\bs\pi,\bs\pi'\in\mathcal P^+$ have dissociate $q$-factorizations if the set of $q$-factors of $\bs\pi\bs\pi'$ is the union of the sets of $q$-factors of $\bs\pi$ and $\bs\pi'$. It will also be convenient to work with factorizations in KR type polynomials which not necessarily satisfy \eqref{defqfact}. Such factorization will be referred to as pseudo $q$-factorizations and the associated factors as the corresponding pseudo $q$-factors.

For $\bs\varpi\in\mathcal P$ and $J\subseteq I$, let $\bs\varpi_J$ be the associated
$J$-tuple of rational functions and let $\cal
P_J=\{\bs\varpi_J:\bs\varpi\in \mathcal P\}$. Similarly define $\cal
P_J^+$. Notice that $\bs\varpi_J$ can be regarded as an element of
the $\ell$-weight lattice of $U_q(\tlie g)_J$. Let $\pi_J:\mathcal P\to
\cal P_J$ denote the map $\bs\varpi\mapsto\bs\varpi_J$. If $J=\{j\}$
is a singleton, we write $\pi_j$ instead of $\pi_J$.

Given $i\in I, a\in\mathbb F^\times$, the following elements are known as simple $\ell$-roots:
\begin{equation}
	\bs\alpha_{i,a} = (\bs\omega_{i,aq_i,2})^{-1}\prod_{j\ne i} \bs\omega_{j,aq_i,-c_{j,i}}.
\end{equation}
The subgroup of $\mathcal P$ generated by them is called the $\ell$-root lattice of $U_q(\tlie g)$ and will be denoted by $\mathcal Q_q$. Let also $\cal Q_q^+$ be the submonoid generated by the simple $\ell$-roots. Quite clearly, $\wt(\bs\alpha_{i,a})=\alpha_i$. Define a partial order on $\mathcal P$ by
$$\bs\varpi\le\bs\omega \quad{if}\quad \bs\omega\bs\varpi^{-1}\in\cal Q_q^+.$$

\subsection{Pre-Factorization Graphs}\label{ss:prefact}
Given a set $I$, an $I$-coloring of a graph $G=(\mathcal V,\mathcal A)$ is a function $c:\mathcal V\to I$. Given an $I$-coloring $c$ and $i\in I$, let $\mathcal V_i=\{x\in \mathcal V:c(x)=i\}$. 
By a colored graph we will mean an oriented graph $G$ with a choice of coloring $c:\mathcal V\to I$. 

We shall also decorate the vertices and arrows of graphs by positive integers. We will refer to a function $\lambda:\mathcal V\to \mathbb Z_{>0}$ as a weight  and to a function $\epsilon:\mathcal A\to \mathbb Z_{>0}$ as an exponent on $G$. The number $\epsilon(a)$ will be referred to as the exponent of $a$. We will always assume  that $\epsilon$ satisfies the following compatibility condition:
\begin{equation}\label{e:expcomp}
	\epsilon_\rho = \epsilon_{\rho'} \qquad\text{for all}\qquad \rho,\rho'\in\mathscr P_{v,v'}, v,v'\in \mathcal V,
\end{equation}
where, if $\rho=e_1\cdots e_m$  is  such that $\sigma_\rho=(s_1,\dots,s_m)$ and $e_j=\pi(a_j)$,
\begin{equation*}
	\epsilon_\rho := \sum_{j=1}^m s_j\epsilon(a_j).
\end{equation*}
Evidently, $\epsilon_{\rho^-}=-\epsilon_{\rho}$ and one easily checks  $\epsilon_{\rho*\rho'} = \epsilon_{\rho} + \epsilon_{\rho'}$. 
Set
\begin{equation}\label{e:incpaths}
	\mathscr P^+_G = \{\rho\in\mathscr P_G:\epsilon_\rho>0\} \quad\text{and}\quad \mathscr P^-_G = \{\rho\in\mathscr P_G:\epsilon_\rho<0\}.
\end{equation}

We shall refer to the data $(G,c,\lambda,\epsilon)$ formed by a colored oriented graph, a weight $\lambda$, and an exponent $\epsilon$ on $G$  as a pre-factorization graph. We shall abuse of language and simply say $G$ is a pre-factorization graph. We locally illustrate the structures maps of  a pre-factorization graph with the following picture  
\setlength{\unitlength}{.4cm}
\begin{equation*}
	\begin{tikzcd}
		\stackrel{r}{i} \arrow[r,"m"] & \stackrel{s}{j}
		\end{tikzcd} 
\end{equation*}
where $i$ and $j$ are the colors at the corresponding vertices, $r$ and $s$ are their associated weights and $m$ is the exponent associated to the given arrow.
The following is an obvious consequence of \eqref{e:expcomp}.

\begin{lem}\label{l:nocycle}
	If $G$ is a pre-factorization graph, then $G$ contains no oriented cycles.\hfill\qed
\end{lem}

In particular, only the last graph in \eqref{e:linefus} can be equipped with a pre-factorization graph structure.
If $I$ is as in \Cref{ss:clalg} and $G$ is a connected pre-factorization graph, for each choice of $(v_0,a)\in \mathcal V\times\mathbb F^\times$, we can associate a Drinfeld polynomial  as follows. Define
\begin{equation}
	a_{v_0}=a \quad\text{and}\quad a_{v}=aq^{\epsilon_\rho}  \quad\text{if}\quad \rho\in\mathscr P_{v_0,v}.
\end{equation}
Condition \eqref{e:expcomp} guarantees this is well-defined. Then, define
\begin{equation}\label{e:polytograph}
	\bs\pi_{G,v_0,a}=\prod_{v\in \mathcal V} \bs\omega_{c(v),a_v,\lambda(v)}.
\end{equation}
One can easily check that, 
\begin{equation}
	\bs\pi_{G,v_0,a'q^{-\epsilon_\rho}} = \bs\pi_{G,v_0',a'} \qquad\text{for all}\qquad (v_0',a')\in \mathcal V\times\mathbb F^\times,\ \rho\in\mathscr P_{v_0,v_0'}.
\end{equation}
Therefore, up to a uniform modification on the centers of the factors in the right-hand side of \eqref{e:polytograph}, the definition is independent of the choice of $(v_0,a)$. We will often write $\bs\pi_G$ to shorten notation when the knowledge of precise centers is not relevant. 

\begin{ex}\label{ex:pigamma}
	Assume $\lie g$ is of type $A_2$, so $I=\{1,2\}$, and consider the following pre-factorization graph
	\begin{equation*}
		\begin{tikzcd}
			\stackrel{2}{1} \arrow[r,"3"] & \stackrel{2}{2} & \stackrel{1}{1}\arrow[swap,l,"4"]
		\end{tikzcd} 
	\end{equation*}
	If we select the middle vertex to define $\bs\pi=\bs\pi_G$, we get
	\begin{equation*}
		\bs\pi = \bs\omega_{2,a,2}\ \bs\omega_{1,aq^3,2}\ \bs\omega_{1,aq^4}
	\end{equation*}
	Note that, in this case, the factors in \eqref{e:polytograph} are the $q$-factors of $\bs\pi$. However, this may not be the case as the following trivial example shows:
	\begin{equation*}
		\begin{tikzcd}
			\stackrel{1}{1} \arrow[r,"2"] & \stackrel{1}{1}
		\end{tikzcd} 
	\end{equation*}
	In this case, if we choose the first vertex as the base for the definition, the factors in \eqref{e:polytograph} are $\bs\omega_{1,a}$ and $\bs\omega_{1,aq^2}$,  which combine to form a single $q$-factor.
	\endd 
\end{ex} 

\subsection{Hopf Algebra Facts}\label{ss:hopf} We recall some general facts about Hopf algebras (see for instance \cite{egno} and references therein). 

Given a Hopf algebra $\mathcal H$ over $\mathbb F$, its category $\mathcal C$ of finite-dimensional representations is an abelian monoidal category and we denote the (right) dual of a module  $V$ by $V^*$. More precisely, the action of $\mathcal H$ of $V^*$ is given by
\begin{equation}\label{rightdual}
	(hf)(v) = f(S(h)v) \quad\text{for}\quad h\in\mathcal H, f\in V^*, v\in V.
\end{equation}
The evaluation map $V^*\otimes V\to \mathbb F$ is a module map, where $\mathbb F$ is regarded as the trivial module by using the counit map.
Moreover if, $v_1,\dots, v_n$ is a basis of $V$ and $f_1, \dots, f_n$ is the corresponding dual basis, there exists a unique a homomorphism of modules
\begin{equation*}
	\mathbb F\to V\otimes V^*, \qquad 1\mapsto \sum_{i=1}^n v_i\otimes f_i,
\end{equation*}
called the coevaluation map. We denote the evaluation and coevaluation maps associated to a module $V$ by $\operatorname{ev}_V$ and $\operatorname{coev}_V$, respectively, or simply by ${\rm ev}$ and ${\rm coev}$ if no confusion arises. In particular, 
\begin{equation*}%\label{e:tptrivial}
	\operatorname{Hom}_{\mathcal H}(\mathbb F, V\otimes V^*) \ne 0 \qquad\text{and}\qquad \operatorname{Hom}_{\mathcal H}(V^*\otimes V, \mathbb F)\ne 0.
\end{equation*}
If the antipode is invertible, the notion of left dual module is obtained by replacing $S$ by $S^{-1}$ in \eqref{rightdual}. The left dual of $V$ will be denoted by $^*V$ and we have
\begin{equation*}
	^*(V^*)\cong (^*V)^* \cong V.
\end{equation*}

Given $\mathcal H$-modules $V_1,V_2,V_3$, we have
\begin{equation}\label{e:frobrec}
	\begin{aligned}
		\operatorname{Hom}_{\mathcal C}(V_1\otimes V_2, V_3)\cong \operatorname{Hom}_{\mathcal C}(V_1, V_3\otimes V_2^*), \quad
		\operatorname{Hom}_{\mathcal C}(V_1, V_2\otimes V_3)\cong \operatorname{Hom}_{\mathcal C}(V_2^*\otimes V_1, V_3),
	\end{aligned}
\end{equation}
and 
\begin{equation}\label{e:dualtp}
	(V_1\otimes V_2)^*\cong V_2^*\otimes V_1^*.
\end{equation}
For instance, an isomorphism for the first statement in \eqref{e:frobrec} is given by 
\begin{equation*}%\label{e:frobrecp}
	f\mapsto (f\otimes\operatorname{id}_{V_2^*})\circ(\operatorname{id}_{V_1}\otimes\operatorname{coev}_{V_2})\circ\gamma_{V_1}
\end{equation*}
and has the inverse 
\begin{equation*}
	g\mapsto\gamma'_{V_3}\circ(\operatorname{id}_{V_3}\otimes\operatorname{ev}_{V_2})\circ(g\otimes\operatorname{id}_{V_2}),
\end{equation*}
where $\gamma_V:V\rightarrow V\otimes\mathbb{F}$ and $\gamma'_V:V\otimes\mathbb{F}\rightarrow V$ are the canonical maps.
Note also that every short exact sequence
\begin{equation*}
	0\to V_1\to V_2\to V_3\to 0
\end{equation*}
gives rise to another short exact sequence of the form
\begin{equation}\label{e:dualses}
	0\to V_3^*\to V_2^*\to V_1^*\to 0.
\end{equation}

We shall use the following lemma in the same spirit as in \cite{kkko} (a proof can also be found in \cite{sil}).

\begin{lem}\label{l:subtens}
	Let $V_1,V_2,V_3\in\mathcal C$ and suppose $M$ is a submodule of $V_1\otimes V_2$ and $N$ is a submodule of $V_2\otimes V_3$ such that $$M\otimes V_3\subseteq V_1\otimes N.$$
	Then, there exists a submodule $W$ of $V_2$ such that 
	\begin{equation*}
		M\subseteq V_1\otimes W\qquad\textrm{and}\qquad W\otimes V_3\subseteq N.
	\end{equation*}
	Similarly, if $V_1\otimes N\subseteq M\otimes V_3,$ there exists a submodule $W$ of $V_2$ such that 
	\begin{equation*}
		N\subseteq W\otimes V_3\qquad\textrm{and}\qquad V_1\otimes W\subseteq M.
	\end{equation*}\qed
\end{lem}

\begin{lem}\label{c:nonzeromorph}
	Let $V_1,V_2,V_3, L_1,L_2$ be $\mathcal H$-modules and assume  $V_2$ is simple. If
	\begin{equation*}
		\varphi_1:L_1\rightarrow V_1\otimes V_2\quad\textrm{and}\quad\varphi_2: V_2\otimes V_3\rightarrow L_2
	\end{equation*}
	are nonzero homomorphisms, the composition
	$$L_1\otimes V_3 \xrightarrow{\varphi_1\otimes \operatorname{id}_{V_3}} V_1\otimes V_2\otimes V_3 \xrightarrow{\operatorname{id}_{V_1}\otimes \varphi_2} V_1\otimes L_2
	$$
	does not vanish. Similarly, if 
	\begin{equation*}
		\varphi_1:V_1\otimes V_2\rightarrow L_1\quad\textrm{and}\quad\varphi_2:L_2\rightarrow V_2\otimes V_3
	\end{equation*}
	are nonzero homomorphisms, the composition
	$$V_1\otimes L_2\xrightarrow{\operatorname{id}_{V_1}\otimes\varphi_2} V_1\otimes V_2\otimes V_3 \xrightarrow{\varphi_1\otimes\operatorname{id}_{V_3}} L_1\otimes V_3$$
	does not vanish.
\end{lem}

\begin{proof}
	We will write down the details for the first claim only, as the second can be proved similarly. Assume  
	\begin{equation*}
		(\operatorname{id}_{V_1}\otimes\varphi_2)\circ(\varphi_1\otimes\operatorname{id}_{V_3})=0,
	\end{equation*}
	i.e., $$\operatorname{Im}(\varphi_1)\otimes V_3=\operatorname{Im}(\varphi_1\otimes\operatorname{id}_{V_3})\subseteq\operatorname{Ker}(\operatorname{id}_{V_1}\otimes\varphi_2)=V_1\otimes\operatorname{Ker}(\varphi_2).$$   \Cref{l:subtens} implies there exists a submodule $W\subseteq V_2$ such that 
	\begin{equation*}
		\operatorname{Im}(\varphi_1)\subseteq V_1\otimes W\quad\textrm{and}\quad W\otimes V_3\subseteq\operatorname{Ker}(\varphi_2).
	\end{equation*}
	Since $V_2$ is simple, either $W=0$ or $W=V_2$. If $W=0$, then $\operatorname{Im}(\varphi_1)\subseteq V_1\otimes W=0$, which is a contradiction, since  $\varphi_1$ is nonzero. On the other hand, if $W=V_2$, it follows that $V_2\otimes V_3=\operatorname{Ker}(\varphi_2)$, yielding a contradiction, since $\varphi_2$ is nonzero.	
\end{proof}

\section{Representation Theory and $q$-Factorization Graphs}

We start this section reviewing the relevant representation theoretic background for our purposes. This will lead to the main definition of the paper: that of $q$-factorization graphs. We then state the main results of the paper and end the section with a few illustrative examples.

\subsection{Finite-Dimensional Representations}\label{ss:repth}
Let $\cal C$ be the category of all finite-dimensional (type-$1$) weight modules of $U_q(\lie g)$. Thus, a finite-dimensional $U_q(\lie g)$-module $V$ is in $\cal C$ if 
\begin{equation*}
	V=\bigoplus_{\mu\in P}^{} V_\mu \qquad\text{where}\qquad  V_\mu=\{v\in V: k_iv=q_i^{\mu(h_i)}v \text{ for all } i\in I\}.
\end{equation*}
The following theorem summarizes the basic facts about $\cal C$.

\begin{thm}\label{t:ciuqg} Let $V$ be an object of $\cal C$. Then:
	\begin{enumerate}[(a)]
		\item $\dim V_\mu = \dim V_{w\mu}$ for all $w\in\cal W$.
		\item $V$ is completely reducible.
		\item For each $\lambda\in P^+$, the $U_q(\lie g)$-module $V_q(\lambda)$ generated by a vector $v$ satisfying
		$$x_i^+v=0, \qquad k_iv=q^{\lambda(h_i)}v, \qquad (x_i^-)^{\lambda(h_i)+1}v=0,\quad\forall\ i\in I,$$
		is irreducible and finite-dimensional. If $V\in\cal C$ is
		irreducible, then $V$ is isomorphic to $V_q(\lambda)$  for some
		$\lambda\in P^+$.
			\hfill\qedsymbol
	\end{enumerate}
\end{thm}

If $J\subseteq I$ we shall denote by $V_q(\lambda_J)$ the simple
$U_q(\lie g)_J$-module of highest weight $\lambda_J$. Since $\cal
C$ is semisimple, it is easy to see that, if $\lambda\in P^+$ and
$v\in V_q(\lambda)_\lambda$ is nonzero, then $U_q(\lie g)_Jv\cong
V_q(\lambda_J)$.

Let $\wcal C$ the category of all finite-dimensional $\ell$-weight modules of $U_q(\tlie g)$.
Thus, a finite-dimensional $U_q(\tlie g)$-module $V$ is in $\wcal C$ if
$$V=\bigoplus_{\bs\varpi\in\mathcal P}^{} V_{\bs\varpi}$$
where
$$v\in V_{\bs\varpi} \quad\Leftrightarrow\quad \exists\ k\gg 0 \quad\text{s.t.}\quad (\eta-\bs\varpi(\eta))^kv=0 \quad\text{for all}\quad \eta\in U_q(\tlie h).$$
$V_{\bs\varpi}$ is called the $\ell$-weight space of $V$ associated to $\bs\varpi$. Note that if $V\in\wcal C$, then $V\in\cal C$ and
\begin{equation*}
	V_\mu = \bigoplus_{\bs\varpi:\wt(\bs\varpi)=\mu}^{}
	V_{\bs\varpi}.
\end{equation*}
If $V\in \wcal C$, the qcharacter of $V$ is the following element of the group ring $\mathbb Z[\mathcal P]$:
\begin{equation*}\label{e:qchtp}
	\qch(V) = \sum_{\bs\varpi\in\mathcal P} \dim(V_{\bs\varpi})\bs\varpi.
\end{equation*}

A nonzero vector $v\in V_{\bs\varpi}$ is
said to be a highest-$\ell$-weight vector if
$$\eta v=\bs\varpi(\eta)v \quad\text{for every}\quad \eta\in U_q(\tlie h)
\quad\text{and}\quad x_{i,r}^+v=0 \quad\text{for all}\quad i\in I, r\in\mathbb Z.$$
$V$ is said to be a highest-$\ell$-weight module if it is generated by a
highest-$\ell$-weight vector.
Evidently, every highest-$\ell$-weight module has a maximal proper submodule and, hence, a unique irreducible quotient. In particular, 
if two simple modules are highest-$\ell$-weight, then they are isomorphic if and only if the highest $\ell$-weights are the same. This is also equivalent to saying that they have the same qcharacter. The following was proved in \cite{cp:qaa}.

\begin{thm}%\label{t:weyl}
	Every simple object of $\wcal C$ is a highest-$\ell$-weight module. There exists a simple object of $\wcal C$ of highest $\ell$-weight $\bs\pi$ if and only if $\bs\pi\in\mathcal P^+$.\hfill\qedsymbol
\end{thm}

It follows that $\qch(V)$ completely determines the irreducible factors of $V$. We shall denote by $L_q(\bs\pi)$ any representative of the isomorphism class of simple modules with highest $\ell$-weight $\bs\pi$.  For $J\subseteq I$, we shall denote by $L_q(\bs\pi_J)$ the simple
$U_q(\tlie g)_J$-module of highest weight $\bs\pi_J$.  

If $V$ is a highest-$\ell$-weight module with highest-$\ell$-weight vector $v$ and $J\subset I$, we let $V_J$ denote the $U_q(\tlie g)_J$-submodule of $L_q(\bs{\pi})$ generated by $v$. Evidently, if $\bs\pi$ is the highest-$\ell$-weight of $V$, then $V_J$ is highest-$\ell$-weight with highest $\ell$-weight $\bs\pi_J$. Moreover, we have the following well-known facts:

\begin{equation}\label{e:weightsJ}
	V_J = \bigoplus_{\eta\in Q_J^+} V_{\wt(\bs\pi)-\eta}  = \bigoplus_{\eta\in Q_J} V_{\wt(\bs\pi)+\eta}.
\end{equation}

\begin{lem}\label{sl32sl2}
	If $V$ is simple, so is $V_J$.\hfill\qed
\end{lem}

\subsection{Tensor Products and Duality for $U_q(\tlie g)$-modules}\label{ss:tpdu}
It is well known that $U_q(\tlie g)$ is a Hopf algebra with invertible antipode. 
For the proof of following proposition, see \cite[Propositions 1.5 and 1.6]{cha:minr2} (part (b) has not been proved there, but the proof is similar to that of part (c)).

\begin{prop}\label{p:Cartaninv}\hfill\\\vspace{-10pt}
	\begin{enumerate}[(a)]
		\item Given $a\in \mathbb C^\times$, there exists a unique Hopf algebra automorphism $\tau_a$ of $U_q(\tlie g)$ such that
		$$\tau_a(x_{i,r}^{\pm})=a^rx_{i,r}^{\pm},\quad \tau_a(h_{i,s})=a^rh_{i,s},\quad \tau_a(k_{i}^{\pm})=k_{i}^{\pm}, \quad i\in I,\ r,s\in\mathbb Z, s\ne 0.$$
		
		\item There exists a unique Hopf algebra automorphism $\sigma$ of $U_q(\tlie g)$ such that
		$$\sigma(x_{i,r}^{\pm})= x_{i^*,r}^{\pm},\quad \sigma(h_{i,s})= h_{i^*,s},\quad \sigma(k_{i}^{\pm})=k_{i^*}^{\pm}, \quad i\in I,\ r,s\in\mathbb Z, s\ne 0.$$
		
		\item There exists a unique algebra automorphism $\kappa$ of $U_q(\tlie g)$ such that
		$$\kappa(x_{i,r}^{\pm})=-x_{i,-r}^{\mp},\quad \kappa(h_{i,s})=-h_{i,-s},\quad \kappa(k_{i}^{\pm1})=k_{i}^{\mp1}, \quad i\in I,\ r,s\in\mathbb Z, s\ne 0.$$
		Moreover $(\kappa\otimes \kappa)\circ \Delta=\Delta^{\text{op}}\circ \kappa$,
		where $\Delta^{\text{op}}$ is the opposite comultiplication of $U_q(\tlie g)$.\footnote{The automorphism $\kappa$ is most often denoted by $\hat\omega$ in the literature and its restriction to $U_q(\lie g)$, typically denoted by $\omega$, is referred to as the Cartan automorphism of $U_q(\lie g)$. We chose to modify the notation to avoid visual confusion with our most often used symbol for a Drinfeld polynomial: $\bs\omega$.}\hfil\qed
	\end{enumerate}
\end{prop}

Given $\bs\pi\in\cal P^+$, define $\bs\pi^{\tau_a}\in \cal P^+$ by $\bs\pi^{\tau_a}_i(u)=\bs\pi_i(au).$ One easily checks that the pullback
$L_q(\bs\omega)^{\tau_a}$ of $L_q(\bs\pi)$ by $\tau_a$ satisfies
\begin{equation}\label{e:tau_a}
	L_q(\bs\pi)^{\tau_a}\cong L_q(\bs\pi^{\tau_a}).
\end{equation}
Define also $\bs\pi^\sigma,\bs\pi^*\in\mathcal P^+$ by
\begin{equation}\label{e:dualDpolidef}
	\bs\pi_i^\sigma(u) = \bs\pi_{i^*}(u) \quad\text{for}\quad i\in I, \quad\text{and}\quad \bs\pi^* = (\bs\pi^\sigma)^{\tau_{q^{-r^\vee h^\vee}}} = (\bs\pi^{\tau_{q^{-r^\vee h^\vee}}})^\sigma.
\end{equation}
It is well-known that 
\begin{equation}\label{e:dualDpoli}
	L_q(\bs\pi)^* \cong L_q(\bs\pi^*).
\end{equation}

We denote by $V^\sigma$ and $V^\kappa$ the pull-back of $V$ by $\sigma$ and $\kappa$, respectively. In particular,
\begin{equation}\label{e:tppb}
	(V_1\otimes V_2)^\sigma\cong V_1^{\tau_a}\otimes V_2^{\tau_a}, \quad (V_1\otimes V_2)^\sigma\cong V_1^\sigma\otimes V_2^\sigma, \quad\text{and}\quad (V_1\otimes V_2)^\kappa\cong V_2^\kappa\otimes V_1^\kappa.
\end{equation}
Also, for any short exact sequence
\begin{equation*}
	0\to V_1\to V_2\to V_3\to 0,
\end{equation*}
we have  short exact sequences
\begin{equation}\label{e:cinvses}
	0\to V_1^f\to V_2^f\to V_3^f\to 0 \qquad\text{with}\qquad f=\tau_a,\sigma,\kappa.
\end{equation}
Moreover, if $\bs\pi\in\cal P^+$ with $\bs\pi_i(u) = \prod_j (1-a_{i,j}u)$,
where $a_{i,j}\in \mathbb F$, and $\bs\pi^-\in\cal P^+$ is defined by 
$\bs\pi^-_i(u) = \prod_j (1-a_{i,j}^{-1}u)$, we have
\begin{equation}\label{e:kappapb}
	V_q(\bs{\pi})^\sigma \cong V_q(\bs\pi^\sigma) \quad\text{and}\quad V_q(\bs{\pi})^\kappa \cong V_q(\bs\pi^\kappa)
	\quad\text{where}\quad \bs{\pi}^\kappa = (\bs{\pi}^-)^*.
\end{equation}

It was proved in \cite{fr:qchar} that
\begin{equation}\label{e:multqchar}
	\qch(V\otimes W) = \qch(V)\qch(W).
\end{equation}
In particular, we have:

\begin{prop}\label{sinter}
	Let $\bs{\pi}, \bs{\varpi}\in\mathcal{P}^+$. Then, $L_q(\bs{\pi})\otimes L_q(\bs{\varpi})$ is simple if and only if $L_q(\bs{\varpi})\otimes L_q(\bs{\pi})$ is simple and, in that case, $L_q(\bs{\pi})\otimes L_q(\bs{\varpi}) \cong L_q(\bs\pi\bs\varpi)\cong L_q(\bs{\varpi})\otimes L_q(\bs{\pi})$. 	\hfil\qed
\end{prop}

It turns out that the determination of the simplicity of tensor products can be reduced to that of two-fold tensor products. This is the main result of \cite{Hernandez1}:

\begin{thm}\label{irred}
	If $S_1,\cdots, S_n$ are simple $U_q(\tlie g)$-modules, the tensor product $$S_1\otimes\cdots\otimes S_n$$ is simple if, and only if, $S_i\otimes S_j$ is simple for all $1\leq i<j\leq n$.\hfil\qed
\end{thm}

Given a connected subdiagram $J$, since $U_q(\tlie g)_J$ is not a sub-coalgebra of $U_q(\tlie g)$, if $M$ and $N$ are $U_q(\tlie g)_J$-submodules of $U_q(\tlie g)$-modules $V$ and $W$, respectively, it is in general not true that $M\otimes N$ is a $U_q(\tlie g)_J$-submodule of $V\otimes W$.  Recalling that we have an algebra isomorphism $U_q(\tlie g)_J\cong U_{q_J}(\tlie g_J)$, we shall denote by $M\otimes_J N$ the $U_q(\tlie g)_J$-module obtained by using the coalgebra structure from $ U_{q_J}(\tlie g_J)$. The next result describes a special situation on which $M\otimes N$ is a submodule isomorphic to $M\otimes_J N$. Recall the notation defined in the paragraph preceding \Cref{sl32sl2}.

\begin{prop}[{\cite[Proposition 2.2]{cp:minsl}}]\label{p:subJtp}
	Let $V$ and $W$ be finite-dimensional highest-$\ell$-weight modules with highest $\ell$-weights $\bs\pi,\bs\varpi\in\mathcal P^+$, respectively, and let $J\subseteq I$ be a connected subdiagram. Then, $V_J\otimes W_J$ is a $U_q(\tlie g)_J$-submodule of $V\otimes W$ isomorphic to $V_J\otimes_J W_J$ via the identity map.\hfill\qed
\end{prop}

\begin{cor}\label{c:sJs}
	In the notation of \Cref{p:subJtp}, if $V\otimes W$ is highest-$\ell$-weight, so is $V_J\otimes W_J$. Moreover, if $V\otimes W$ is simple, so is $V_J\otimes W_J$.
\end{cor}

\begin{proof}
	As shown in the proof of  \Cref{p:subJtp}, we have
	\begin{equation}\label{e:sJs}
		V_J\otimes W_J = \opl_{\eta\in Q_J^+}^{} (V\otimes W)_{\wt(\bs\pi)+\wt(\bs\varpi)-\eta}.
	\end{equation}
	Thus, if $V\otimes W$ is highest-$\ell$-weight, any nonzero vector in $V_J\otimes W_J$  is a linear combination of vectors of the form
	$x_{i_1,r_1}^-\cdots x_{i_l,r_l}^-(v\otimes w)$
	for some $l\ge 0, i_k\in I, r_k\in\mathbb Z, 1\le k\le l$. But the weight of such vector is 
	\begin{equation*}
		\wt(\bs\pi)+\wt(\bs\varpi)-\sum_{k=1}^l \alpha_{i_k}
	\end{equation*}
	and, hence, we must have $i_k\in J$ for all $1\le k\le l$, which implies the first claim.	 
	The second claim follows from the first together with  \Cref{sl32sl2}. 
\end{proof}

\subsection{Simple Prime Modules and $q$-Factors}\label{ss:primem}

A finite-dimensional $U_q(\tlie g)$-module $V$ is said to be prime if it is not isomorphic to a tensor product of two non trivial modules. Evidently, any 
finite-dimensional simple module can be written as a tensor product of (simple) prime modules. If a prime module $P$ appears in some factorization of a simple module $S$, we shall say that $P$ is a prime factor of $S$. 

In particular, in light of \eqref{e:multqchar}, in order understand the qcharacters of the simple modules, it suffices to understand those of the simple prime modules. However, the only case the classification of simple prime modules is completely understood is for $\lie g=\lie{sl}_2$. In that case, the classification is given by the following theorem, proved in \cite{cp:qaa}.

\begin{thm}\label{t:primesl2}
	If $\lie g=\lie{sl}_2,\bs\pi\in\mathcal P^+,$ and the $q$-factors of $\bs\pi$ are $\bs\pi^{(j)}, 1\le j\le m$, then
	\begin{equation*}
		L_q(\bs\pi)\cong L_q(\bs\pi^{(1)})\otimes\cdots\otimes L_q(\bs\pi^{(m)}).
	\end{equation*}
	Moreover, up to re-ordering, $L_q(\bs\pi)$ has a unique factorization as tensor product of prime modules.
	In particular, $L_q(\bs\pi)$ is prime if and only if it has a unique $q$-factor.\hfill\qed 
\end{thm}

If $\bs\pi\in\mathcal P^+$ has a unique $q$-factor, the module $L_q(\bs\pi)$ is called a Kirillov-Reshetikhin (KR) module. 
It is well-known (see \cite{cha:braid,ohscr:simptens} and references therein) that, given $(i,r),(j,s)\in I\times  \mathbb Z_{>0}$, there exists a finite set $\mathscr R_{i,j}^{r,s} \subseteq \mathbb Z_{>0}$
such that
\begin{equation}\label{defredset}
	L_q(\bs\omega_{i,a,r})\otimes L_q(\bs\omega_{j,b,s}) \text{ is reducible}\qquad\Leftrightarrow\qquad \frac{a}{b} = q^m \text{ with } |m|\in \mathscr R_{i,j}^{r,s}.
\end{equation}
Moreover, in that case,
\begin{equation}\label{e:krhwtp}
	L_q(\bs\omega_{i,a,r})\otimes L_q(\bs\omega_{j,b,s}) \text{ is  highest-$\ell$-weight}\qquad\Leftrightarrow\qquad m>0.
\end{equation}
It follows from \Cref{sinter} and \eqref{e:dualtp} that
\begin{equation}\label{e:redsim}
	\mathscr R_{j,i}^{s,r} = \mathscr R_{i,j}^{r,s} = \mathscr R_{i^*,j^*}^{r,s} .
\end{equation}

\begin{thm}\label{t:krredsets}
	If $\lie g$ is of type $A$  and $i,j\in I, r,s\in\mathbb Z_{>0}$, we have 
	\begin{equation*}
		\mathscr R_{i,j}^{r,s} = \{r+s+d(i,j)-2p: - d([i,j],\partial I)\le p<\min\{r,s\}  \}. 
	\end{equation*}\hfill\qed
\end{thm}

The above was essentially proved in \cite{cha:braid} and can be read off the results of \cite{ohscr:simptens}, from where the description for other types can also be extracted (see also \cite{jamo:tpkr}).

Given a connected subdiagram $J$ such that $[i,j]\subseteq J$, let $\mathscr R_{i,j,J}^{r,s}$ be determined by 
\begin{equation*}
	V_q((\bs\omega_{i,a,r})_J)\otimes V_q((\bs\omega_{j,b,s})_J)  \text{ is reducible}\qquad\Leftrightarrow\qquad \frac{a}{b} = q^m \text{ with } |m|\in \mathscr R_{i,j,J}^{r,s}.
\end{equation*}
Note this is not the same set obtained by considering the corresponding module for the algebra $U_{q_J}(\tlie g_J)\cong U_q(\tlie g)_J$. Indeed, if we denote the latter by $\mathscr R_{i,j}^{r,s}[J]$, we have
\begin{equation*}
	m\in \mathscr R_{i,j}^{r,s}[J] \quad\Leftrightarrow\quad d_Jm\in \mathscr R_{i,j,J}^{r,s}.
\end{equation*}
Note also that \Cref{c:sJs} implies
\begin{equation}\label{incredsets}
	\mathscr R_{i,j,J}^{r,s}\subseteq \mathscr R_{i,j,K}^{r,s} \quad\text{if}\quad J\subseteq K.
\end{equation}
Finally, set 
\begin{equation}\label{e:redsetsl2}
	\mathscr R_{i}^{r,s} =  \mathscr R_{i,i,\{i\}}^{r,s}
\end{equation}

\begin{cor}\label{c:krtpsl2}
	For every $i\in I, r,s\in\mathbb Z_{>0}$, $\mathscr R_{i}^{r,s}=\{d_i(r+s-2p):0\le p<\min\{r,s\}\}$.\hfil\qed
\end{cor}

\begin{prop}\label{p:qftp}
	If $\bs\pi,\bs\varpi\in\mathcal P^+$ are such that $L_q(\bs\pi)\otimes L_q(\bs\varpi)$ is simple, then they have dissociate $q$-factorizations.
\end{prop}

\begin{proof}
	If the $q$-factorizations are not dissociate, it follows from \Cref{t:primesl2} that there exists $i\in I$ and $q$-factors $\bs\omega$ of $\bs\pi$ and $\bs\omega'$ of $\bs\pi'$, both supported at $i$, such that $L_q(\bs\omega)\otimes L_q(\bs\omega')$ is reducible. Moreover, writing $\bs\pi = \widetilde{\bs\pi}\bs\omega$ and $\bs\pi'= \widetilde{\bs\pi}'\bs\omega'$, it follows that
	\begin{equation*}
		L_q(\bs\pi_i)\otimes L_q(\bs\pi'_i)\cong L_q(\widetilde{\bs\pi}_i)\otimes L_q(\bs\omega_i)\otimes L_q(\bs\omega'_i)\otimes L_q(\widetilde{\bs\pi}'_i)
	\end{equation*}
	which is reducible, yielding a contradiction with \Cref{c:sJs}.	
\end{proof}

\begin{cor}\label{c:qftp}
	Let $\bs\pi\in\mathcal P^+$.  $L_q(\bs\pi)$ is prime if and only if for every decomposition $\bs\pi = \bs\omega\bs\varpi, \bs\omega,\bs\varpi\in\mathcal P^+$, such that $\bs\omega$ and $\bs\varpi$ have dissociate $q$-factorzations, $L_q(\bs\omega)\otimes L_q(\bs\varpi)$ is reducible.
\end{cor}

\begin{proof}
	If $L_q(\bs\pi)$ is not prime, by definition there exists a nontrivial decomposition $\bs\pi = \bs\omega\bs\varpi$ such that  $L_q(\bs\omega)\otimes L_q(\bs\varpi)$ is simple and \Cref{p:qftp} implies $\bs\omega$ and $\bs\varpi$ have dissociate $q$-factorzations. If $L_q(\bs\pi)$ is prime, by definition,  $L_q(\bs\omega)\otimes L_q(\bs\varpi)$ is reducible for any nontrivial decomposition $\bs\pi = \bs\omega\bs\varpi$.
\end{proof}

Given $\bs\pi\in\mathcal P^+$, consider a nontrivial $2$-set partition of its set of $q$-factors and let $\bs\omega$ and $\bs\varpi$ be the products of the $q$-factors in each of the parts. The above corollary tells us that the task of deciding the primality of $L_q(\bs\pi)$ can be phrased as a task of testing the reducibility of $L_q(\bs\omega)\otimes L_q(\bs\varpi)$ for every such partition. Thus, one can think of organizing the level of complexity of the task by the number of $q$-factors of $\bs\pi$. The answer for the two first levels is given by:

\begin{cor}\label{c:p2qf}
	Every KR module is prime. Moreover, if $\bs\pi\in\mathcal P^+$ has exactly two $q$-factors, say $\bs\omega_{i,a,r}$ and $\bs\omega_{j,b,s}$, then $V_q(\bs\pi)$ is prime if and only if $\frac{a}{b}=q^m$ with $|m|\in\mathscr R_{i,j}^{r,s}$. \hfil\qed
\end{cor}

\subsection{Factorization Graphs}\label{ss:fact}

Recall the definition of the sets $\mathscr R_{i,j}^{r,s}$ in  \eqref{defredset}, as well as \eqref{e:redsetsl2}, and \eqref{e:incpaths}. We shall say that a pre-factorization graph $G$ is a $q$-factorization graph if,
for every $i\in I$, 
\begin{equation}\label{e:areqfact}
	v,v'\in \mathcal V_i, \ \rho\in\mathscr P_{v,v'}  \qquad\Rightarrow\qquad |\epsilon_\rho| \notin \mathscr R_i^{\lambda(v),\lambda(v')}
\end{equation}
and
\begin{equation}\label{e:notherarrow}
	\rho\in\mathscr P_{v,v'}\cap\mathscr P_G^+  \quad\text{with}\quad \epsilon_\rho \in \mathscr R_{c(v),c(v')}^{\lambda(v),\lambda(v')} \qquad\Rightarrow\qquad (v',v)\in\mathcal A.
\end{equation}
Condition \eqref{e:areqfact} ensures that the factors in the right-hand side of \eqref{e:polytograph} are the $q$-factors of $\bs\pi$. On the other hand, \eqref{e:notherarrow} guarantees that no pre-factorzation graph can be obtained by adding an arrow to $G$.
We will refer to a pre-factorization graph satisfying \eqref{e:notherarrow} as a pseudo $q$-factorization graph.

We shall now see that any pseudo $q$-factorization  of a Drinfeld polynomial gives rise  to a pseudo $q$-factorization graph which is a $q$-factorization graph if and only if it is the $q$-factorization.   
Thus, fix a Drinfeld polynomial $\bs\pi$ and let $\mathcal V$ be the corresponding multiset of pseudo $q$-factors. For $i\in I$, let
\begin{equation*}
	\mathcal V_i = \{\bs\omega\in \mathcal V: \supp(\bs\omega)=\{i\}\}.
\end{equation*} 
This gives rise to a coloring $c:\mathcal V\to I$ defined by declaring $\mathcal V_i=c^{-1}(\{i\})$. The weight map $\lambda:\mathcal V\to\mathbb Z_{>0}$ is defined by
\begin{equation}
	\lambda(\bs\omega) = \wt(\bs\omega)(h_i) \qquad\text{for all}\qquad \bs\omega\in \mathcal V_i.
\end{equation}
In particular,
\begin{equation}
	\sum_{i\in I}\sum_{\bs\omega\in \mathcal V_i} \lambda(\bs\omega)\omega_i = \wt(\bs\pi).
\end{equation}
The set of arrows $\mathcal A=\mathcal A(\bs\pi)$ is defined as the set of ordered pairs of 
$q$-factors, say $(\bs\omega_{i,a,r},\bs\omega_{j,b,s})$, such that
\begin{equation}\label{e:arrowsdefp}
	a=bq^m \quad\text{for some}\quad m\in \mathscr R_{i,j}^{r,s}.
\end{equation} 
In representation theoretic terms, this is equivalent to saying:
\begin{equation*}
	L_q(\bs\omega_{i,a,r})\otimes L_q(\bs\omega_{j,b,s}) \quad\text{is reducible and highest-$\ell$-weight.}
\end{equation*}
Note that, in the case of the actual $q$-factorization, we necessarily have $m\notin \mathscr R_i^{r,s}$ when $i=j$. The value of the exponent $\epsilon:\mathcal A\to\mathbb Z_{>0}$ at an arrow satisfying \eqref{e:arrowsdefp} is set to be $m$. Quite clearly, $G= (\mathcal V,\mathcal A)$ with the above choice of coloring, weight, and exponent is a pseudo $q$-factorization graph and $\bs\pi_G =\bs\pi$. We refer to $G$ as a pseudo $q$-factorization graph over $\bs\pi$. In the case this construction was performed using the $q$-factorization of $\bs\pi$, then $G$ will be called the  $q$-factorization graph of $\bs\pi$ and we denote it by $G(\bs\pi)$.

It is now natural to seek for the classification of the prime $q$-factorization graphs, i.e., those for which $L_q(\bs\pi_G)$ is prime.
In type $A_1$, this is the case if and only if the $q$-factorization graph of $\bs\pi$ has a single vertex, which is also equivalent to saying that the graph is connected. For higher rank, the story is much more complicated. We still have the following proposition which will be proved in   \Cref{ss:hlwtp}.

\begin{prop}\label{p:primeconect}
	Let $\bs\pi\in\mathcal P^+$. If $G_1, \cdots, G_k$ are the connected components of $G(\bs\pi)$ and $\bs\pi^{(j)}\in\mathcal P^+, 1\le j\le k$, are such that $\bs\pi = \prod_{j=1}^k \bs\pi^{(j)}$ and $G_j = G(\bs\pi^{(j)})$, then
	\begin{equation*}
		L_q(\bs\pi)\cong L_q(\bs\pi^{(1)})\otimes\cdots\otimes L_q(\bs\pi^{(k)}).
	\end{equation*}  
	In particular,  $G(\bs\pi)$ is connected if $L_q(\bs\pi)$ is prime.	
	\endd
\end{prop}

However, even for type $A_2$, the converse is not true and counter examples can be found in \cite{mosi}, for instance.

We also introduce duality notions for pre-factorization graphs. Given a graph $G$, we denote by $G^-$ the graph obtained from $G$ by reversing all the arrows\footnote{The underlying directed graph is usually called the transpose of $G$.} and keeping the rest of structure of (pre)-factorization graph. In light of \eqref{e:redsim}, $G^-$ is a factorization graph as well, which we refer to as the arrow-dual of $G$. Similarly, the graph $G^*$, called the color-dual of $G$, obtained  by changing the coloring according to the rule $i\mapsto i^*$ for all $i\in I$, is a factorization graph. Moreover,
\begin{equation}\label{e:dualgraphs}
	\bs\pi_{G^-,v,a^{-1}} =\bs\pi_{G,v,a}^- \qquad\text{and}\qquad \bs\pi_{G^*,v,aq^{-r^\vee h^\vee}} =\bs\pi_{G,v,a}^*.
\end{equation}

Given $\bs\pi,\bs\pi'\in\mathcal P^+$, the graph $G(\bs\pi\bs\pi')$ may have no relation to $G(\bs\pi)$ and $G(\bs\pi')$. However, if $\bs\pi$ and $\bs\pi'$ have dissociate $q$-factorizations, then $G(\bs\pi)$ and $G(\bs\pi')$ determine a cut of $G(\bs\pi\bs\pi')$. We will consider several times the situation that the corresponding cut-set is a singleton. More generally, given pseudo $q$-factorization graphs $G$ and $G'$, let $G\otimes G'$ be the unique pseudo $q$-factorization graph whose set of vertices is $\mathcal V\sqcup\mathcal V'$ preserving the original coloring and weight map. The following is trivially established.

\begin{lem}\label{tpqfg}
	If $G$ and $G'$ are pseudo $q$-factorization graphs over $\bs\pi$ and $\bs\pi'$, respectively, then $G\otimes G'$ is a pseudo $q$-factorization graph over $\bs\pi\bs\pi'$. Moreover, if $G=G(\bs\pi)$ and $G'=G(\bs\pi')$, then $G\otimes G'=G(\bs\pi\bs\pi')$ if and only if $\bs\pi$ and $\bs\pi'$ have dissociate $q$-factorizations.\hfil\qed
\end{lem}

If $G$ and $G'$ are pseudo $q$-factorization graphs over $\bs\pi$ and $\bs\pi'$, respectively, we shall say that the pair $(G,G')$ or, equivalently, that $G\otimes G'$ is simple if so is $L_q(\bs\pi)\otimes L_q(\bs\pi')$. Otherwise we say it is reducible. Evidently, $G\otimes G'=G'\otimes G$, regardless if $L_q(\bs\pi)\otimes L_q(\bs\pi')$ is simple or not. 

\subsection{Main Conjectures and Results}\label{ss:main}
Throughout this section, we let $G=G(\bs\pi)=(\mathcal V,\mathcal A)$ be a $q$-factorization graph. We say $G$ is prime if $L_q(\bs\pi)$ is prime. 
One could hope that the complexity of prime graphs is incremental in the sense that all prime graphs with $N+1$ vertices are obtained by adding a vertex in some particular ways to some prime graph with $N$ vertices. In other words:

\begin{con}\label{p:buildupprimes}
	If $G$ is prime and $\#\mathcal V>1$, there exists $v\in \mathcal V$ such that $G_{\mathcal V\setminus\{v\}}$ is also prime.\endd
\end{con}

\begin{prop}\label{p:buildupprimesmono}
	The conclusion of \Cref{p:buildupprimes} holds for every $v\in \partial G$.
		In particular, the conjecture holds if $G$ is a tree.\endd
\end{prop}

The ``in particular'' part of the proposition follows from its first statement since $\partial G\ne \emptyset$ if $G$ is a tree. The difficulty in \Cref{p:buildupprimes} arises when  $\partial G= \emptyset$. \Cref{p:buildupprimesmono} will be proved as an application of \Cref{l:vertexremoval} below which explores the role of sinks and sources.  Together with general combinatorial properties of trees  (\Cref{l:elemtree}(c)).  \Cref{p:buildupprimesmono} implies:

\begin{cor}\label{c:subpt=>p}
	If $G$ is a prime tree, every of its proper connected subgraphs are  prime.\hfill\qed
\end{cor}

Trees are the simplest kind of directed graphs and, among them, totally ordered lines are the simplest. We have:

\begin{thm}\label{gencp}
	If $G$ is a totally ordered line, then $G$ is prime.\endd
\end{thm}

For $\lie g$ of type $A_2$, this was the main result of \cite{cp:fact}. It will be proved here for general $\lie g$, by a  different argument, as a Corollary of \Cref{p:killhlw}. In particular, differently than the proof in \cite{cp:fact}, our proof does not use any information about the elements belonging to the sets $\mathscr R_{i,j}^{r,s}$. For $\lie g$ of type $A$, we also prove the following generalization, which is the main result of the present paper.
\begin{thm}\label{t:toto}
	If $\lie g$ is of type $A$, every totally ordered $q$-factorization graph is prime. \endd
\end{thm}

In particular, a $q$-factorization graph afforded by a tournament is prime.
After \Cref{t:toto}, it becomes natural the purely combinatorial problem of classifying all $q$-factorization graphs which are totally ordered  since this leads to the explicit construction of a family of Drinfeld polynomials whose corresponding simple modules are prime.  
We shall not pursue a general answer for this combinatorial problem here. However, for illustrative purposes, we do present two results in this direction.  One is \Cref{ex:tour} below, where we describe an infinite family of examples of $q$-factorization graphs afforded by tournaments. The other is the following proposition whose proof, given in \Cref{ss:totol}, is essentially a byproduct of some technical lemmas extracted from the proof of \Cref{t:toto} in \Cref{ss:arith}. In particular, it solves this combinatorial problem for type $A_2$ since, in that case, $I=\partial I$.

\begin{prop}\label{p:totopartial}
	Suppose $\lie g$ is of type $A$ and $\bs\pi\in\mathcal P^+$ is such that $G=G(\bs\pi)$ is totally ordered. If $c(\mathcal V_G)\subseteq \partial I$, then $G$ is a line whose vertices are alternately colored.  
	%In particular, if $\lie g$ is of type $A_2$, such lines are the only totally ordered $q$-factorization graphs.\endd   
\end{prop}

We have the following rephrasing of \Cref{c:qftp} in the language  of cuts: $G$ is prime if and only if every nontrivial cut of $G$ is reducible. 
Several partial results in the direction of proving our main results provide criteria for checking the reducibility of certain special  cuts, which we deem to be interesting results in their own right. One of these, which is an immediate consequence of \Cref{c:source<-sink}, is stated here in the language of cuts:

\begin{thm}\label{critredcutext}
	Let $(G',G'')$ be a cut of $G$, and suppose there exist vertices $v'$ of $G'$ and $v''$ of $G''$ satisfying the following conditions:
	\begin{enumerate}[(i)]
		\item $v'$ and $v''$ are adjacent in $G$;
		\item $v'$ and $v''$ are extremal in $G'$ and $G''$, respectively;
		\item $v'$ is extremal in $G$ only if $v'$ is an isolated vertex of $G'$ and similarly for $v''$.
	\end{enumerate}
		Then, $(G',G'')$ is reducible. \endd  
\end{thm}

The following corollary about triangles is immediate.

\begin{cor}\label{c:redtrcut}
	Suppose $G$ is a triangle and let $(G',G'')$ be a cut such that $G'$ is a singleton containing an extremal vertex of $G$. Then, $(G',G'')$ is reducible.\hfil\qed  
\end{cor}

This corollary implies there is only one cut for a triangle which may be simple: the one whose singleton contains the vertex which is not extremal. \Cref{t:toto} implies this is not so for type $A$. However, the present proof  utilizes the precise description of the sets $\mathscr R_{i,j}^{r,s}$ and, hence, in order to extend it to other types, it requires a case by case analysis, which will appear elsewhere.

We also have the following general criterion for primality.

\begin{prop}\label{p:critprimedual}
	$G$ is prime if, for any cut $(G',G'')$ of $G$, there exist $\bs{\varpi}'\in\mathcal{V}_{G'},\;\bs{\varpi}''\in\mathcal{V}_{G''}$ such that one of the following two conditions holds:
	\begin{enumerate}[(i)]
		\item $(\bs{\varpi}'',\bs{\varpi}')\in\mathcal{A}_G$  and $L_q(\bs{\omega}')\otimes L_q(\bs{\omega}'')^*$ is simple for all $(\bs{\omega}',\bs{\omega}'')\in\mathcal{N}^+_{G'}(\bs{\varpi}')\times\mathcal{N}^-_{G''}(\bs{\varpi}'')\setminus\{(\bs{\varpi}',\bs{\varpi}'')\}$.
		\item  $(\bs{\varpi}',\bs{\varpi}'')\in\mathcal{A}_G$ and $L_q(\bs{\omega}')^*\otimes L_q(\bs{\omega}'')$ is simple for all $(\bs{\omega}',\bs{\omega}'')\in\mathcal{N}^-_{G'}(\bs{\varpi}')\times\mathcal{N}^+_{G''}(\bs{\varpi}'')\setminus\{(\bs{\varpi}',\bs{\varpi}'')\}$. \endd
	\end{enumerate}
\end{prop}

\Cref{p:critprimedual} will be proved as a corollary to \Cref{p:neighdualcrit}.

\subsection{Examples}\label{ss:ex}
The first example provides a family of tournaments for type $A$ and, hence, a family of simple prime modules.

\begin{ex}\label{ex:tour}
	Given $N>1$, let $\lie g$ be of type $A_n, n\ge {3N-4}$, identify $I$ with the integer interval $[1,n]$ as usual, and consider
	\begin{equation*}
		\bs\pi = \prod\limits_{i=1}^{N}\bs{\omega}_{i+N-2, q^{3(i-1)}}.
	\end{equation*} 
	Checking that $G(\bs\pi)$ is a tournament with $N$ vertices amounts to showing that
	\begin{equation*}
		3(j-i)\in\mathscr R_{i+N-2, j+N-2}^{1,1} \quad\text{for all}\quad 1\le i<j\le N,
	\end{equation*}
	which, by \Cref{t:krredsets}, is equivalent to
	\begin{equation*}
		3(j-i) = 2 + d(i+N-2, j+N-2) -2p \quad\text{with}\quad  -d([i+N-2,j+N-2], \partial I)\le p\le 0.
	\end{equation*}
	Since $d(i+N-2, j+N-2)=j-i$ and $3(j-i)=2 + (j-i) -2(1-(j-i))$, we need to check
	\begin{equation*}
		-d([i+N-2,j+N-2], \partial I)\le 1-(j-i)\le 0.
	\end{equation*}
	The second inequality is immediate from  $1\leq i<j\leq N$. On the other hand,
	\begin{equation*}
		d([i+N-2,j+N-2], \partial I) = \min\{ (i+N-2)-1, n-(j+N-2) \}\ge N-2,
	\end{equation*}
	from where the first inequality easily follows. 
	
	Although the above family affords triangles only for rank at least $5$, it is easy to build $q$-factorization graphs which are triangles for $n\ge 3$. For instance:	
	\begin{equation*}
		 \begin{tikzcd}[column sep=small,row sep=tiny]
			& \arrow[swap,dl,"6"]  \stackrel{3}{3} \arrow[dr,"3"]& \\  
			\stackrel{3}{1}  & & \arrow[swap,ll,"3"]\stackrel{3}{2} 
		\end{tikzcd} \qquad\text{and}\qquad 
	\begin{tikzcd}[column sep=small,row sep=tiny]
		& \arrow[swap,dl,"7"]  \stackrel{3}{2} \arrow[dr,"3"]& \\  
		\stackrel{3}{1}  & & \arrow[swap,ll,"4"]\stackrel{3}{3} 
	\end{tikzcd}
	\end{equation*}	
	\endd
\end{ex}

\begin{ex}\label{ex:snake}
	We now examine prime snake modules for type $A$ from the perspective of $q$-factorization graphs. The notion of snake and snake modules was introduced in \cite{muyou:path} while that of prime snakes was introduced in \cite{muyou:tsystem} and revised in \cite{dll:clusna}. We now rephrase these definitions in terms of the sets $\mathscr R_{i,j}^{r,s}$. A (type $A$) snake of length $k$ is a sequence $(i_j,m_j)\in\mathcal I\times \mathbb Z, 1\le j\le k$, such that
	\begin{equation*}
		m_{j+1}-m_j = 2+d(i_j,i_{j+1})-2p_j \quad\text{for some}\quad p_j\in\mathbb Z_{\le 0} \quad\text{and all}\quad 1\le j<k.
	\end{equation*}
	The snake is said to be prime if $-d([i_j,i_{j+1}],\partial I)\le p_j$ for all $1\le j<k$. 	In other words, the snake is prime if and only if
	\begin{equation*}
		m_{j+1}-m_j \in \mathscr R_{i_j,i_{j+1}}^{1,1} \quad\text{for all}\quad  1\le j<k.
	\end{equation*}
	Given a snake and $a\in\mathbb F^\times$, the associated snake module is $L_q(\bs\pi)$ with 
	\begin{equation}\label{e:snakepi}
		\bs\pi = \prod_{j=1}^k \bs\omega_{i_j,aq^{m_j}}.
	\end{equation}
	For a general snake, $G(\bs\pi)$ may be disconnected and, hence, not prime. However, if the snake is prime and we regard the definition of $\bs\pi$ as a pseudo $q$-factorization, the associated   pseudo $q$-factorization graph is totally ordered. It is then easy to see that $G(\bs\pi)$, the actual $q$-factorization graph, is also totally ordered and, hence, prime by \Cref{t:toto}. Thus, \Cref{t:toto}, together with \Cref{p:primeconect},  recovers \cite[Proposition 3.1]{muyou:tsystem}.
	
	The following is the $q$-factorization graph arising from the following prime snake for type $A_5$: $(4,-2),    (3,1), (2,4), (3,7)$.
	\begin{equation*}
		\xymatrix@R-1pc @C-0.5pc{
			&  \stackrel{1}{3} \ar[dl]_3\ar[dr]^7\\ \stackrel{1}{2} \ar[rr]^{3} & & \stackrel{1}{3} \ar[rr]^{3} & & \stackrel{1}{4}	}
	\end{equation*}
	%Note every proper connected subgraph is prime and this is not a tournament nor a tree. 
	\endd
\end{ex}

\begin{ex}\label{ex:skew}
	Snake modules also arise in the study of the so called skew representations associated to skew tableaux $\lambda\backslash\mu$ \cite[Section 4]{homo}, which we now review. Fix $m\in\mathbb Z_{\ge 0}$, as well as $\lambda = (\lambda_1,\lambda_2,\dots,\lambda_{m+n+1})\in\mathbb Z^{m+n+1}$ and $\mu = (\mu_1,\dots,\mu_m)\in\mathbb Z^m$ such that
	\begin{equation*}
		\lambda_i\ge \lambda_{i+1}, \quad \mu_l\ge \mu_{l+1}, \quad\text{and}\quad \lambda_k\ge \mu_k \ge\lambda_{k+n+1}
	\end{equation*} 
	for all $1\le i\le m+n, 1\le l<m, 1\le k\le  m$. Set $\mu_0 = +\infty, \mu_{m+1}= -\infty$, and, for $1\le i\le n+1$ and $1\le l\le m+1$, let $\nu_{i,l}$ be the middle value among $\mu_{l-1},\mu_l$, and $\lambda_{i+l-1}$. The skew module of $U_q(\tlie g)$ associated to the skew tableaux $\lambda\setminus\mu$ is the simple module $L_q(\bs\pi^{\lambda,\mu})$, where
	\begin{equation}\label{homopolc}
		\bs\pi^{\lambda,\mu}_i(u) = %\prod_{l=1}^{m+1}\ \prod_{k=1}^{\nu_{i,l}-\nu_{i+1,l}} \left(1-q^{2(\nu_{i+1,l}-j+k)}u \right) = 
		\prod_{l=1}^{m+1} \bs\omega_{i,q^{\nu_{i,l}+\nu_{i+1,l}-2l+1-i},\nu_{i,l}-\nu_{i+1,l}}.
	\end{equation}  
	For $\bs\mu=\emptyset$, this is the Drinfeld polynomial of an evaluation module. With a little patience, one can check each factor of the above definition is a $q$-factor of $\bs\pi^{\lambda,\mu}$. Moreover, each connected component of  $G(\bs\pi^{\lambda,\mu})$ is totally ordered and, hence, corresponds to a prime simple module by \Cref{t:toto}.
	
	In order to explore specific examples, let us organize the table:
	\begin{table}[h!]
		\centering
		\begin{tabular}{ |c |c|c|c|c|  }
			\hline
			\multicolumn{5}{|c|}{Values of $\nu_{i,l}$} \\
			\hline
			$l\backslash i$ & 1 & 2 & $\cdots$ & $n$+1\\ 	\hline
			1   & $\nu_{1,1}$    & $\nu_{2,1}$ &  $\cdots$ & $\nu_{n+1,1}$\\ \hline
			$\vdots$ &     &    & & $\vdots$ \\ 		\hline
			$m+1$ & $\nu_{1,m+1}$ & $\cdots$ & $\cdots$ & $\nu_{n+1,m+1}$\\ \hline
		\end{tabular}
	\end{table}\\
	Plugging this information in \eqref{homopolc}, each row will produce at most one $q$-factor for each $i\in I$. Moreover the centers of the associated $q$-strings are of the form $q^k$ for some exponent $k\in\mathbb Z$. We can then form a table with the corresponding exponents and lengths:
	\begin{table}[h!]
		\centering
		\begin{tabular}{ |c |c|c|c|c|  }
			\hline
			\multicolumn{5}{|c|}{Exponents and Lengths} \\ 			\hline
			$l\backslash i$ & 1 & 2 & $\cdots$ & $n$\\ 	\hline
			1   & $k_{1,1}\ |\ r_{1,1}$    &  $k_{2,1}\ |\ r_{2,1}$  &  $\cdots$ &   $k_{n,1}\ |\ r_{n,1}$ \\ \hline
			$\vdots$ &     &    & & $\vdots$ \\ 		\hline
			$m+1$ &  $k_{1,m+1}\ |\ r_{1,m+11}$  & $\cdots$ & $\cdots$ &  $k_{n,m+1}\ |\ r_{n,m+11}$ \\ \hline
		\end{tabular}
	\end{table}\\
	For instance, if $\lambda = (20,16,10,7,2,0)$ and $\mu = (17,5)$, so $m=2$ and $n=3$, we have
	\begin{table}[h!]
		%\centering
		\begin{tabular}{ |c |c|c|c|c|  }
			\hline
			\multicolumn{5}{|c|}{Values of $\nu_{i,l}$} \\
			\hline
			$l\backslash i$ & 1 & 2 & 3 & 4\\ 	\hline
			1 & 20  & 17 & 17 & 17 \\ \hline
			2 & 16  & 10 & 7  & 5  \\ \hline
			3 & 5 & 5 & 2 & 0\\ \hline
		\end{tabular} \qquad
		%\centering
		\begin{tabular}{ |c|c|c|c|  }
			\hline
			\multicolumn{4}{|c|}{Exponents and Lengths} \\ 			\hline
			$l\backslash i$ & 1 & 2 & 3 \\ 	\hline
			1   & $35\ |\ 3$    &  $31\ |\ 0$  &   $30\ |\ 0$ \\ \hline
			2 & $22\ |\ 6$    & $12\ |\ 3$   & $6\ |\ 2$ \\ 	\hline
			3 &  $4\ |\ 0$  & $0\ |\ 3$ &  $-6\ |\ 2$ \\ \hline
		\end{tabular}
	\end{table}\\
	Organizing the vertices of  $G(\bs\pi^{\lambda,\mu})$ following the rows of the last table, we get:
	\begin{equation*}
		\xymatrix@R-2pc{
			\stackrel{3}{1}& &  \\  \stackrel{6}{1} \ar[r]^{10}& \stackrel{3}{2}\ar[r]^{6} & \stackrel{2}{3}\ar[ddl]_{6} \\ \\ & \stackrel{3}{2} \ar[r]^{6} & \stackrel{2}{3}	}
	\end{equation*}
	Thus, this example leads to a graph with two connected components: a singleton and an oriented line.
	If $\lambda = (6,6,6,4,2,1,1)$ and $\mu = (5)$, so $m=1$ and $n=5$, then 
	\begin{table}[h!]
		%\centering
		\begin{tabular}{ |c |c|c|c|c|c|c|  }
			\hline
			\multicolumn{7}{|c|}{Values of $\nu_{i,l}$} \\
			\hline
			$l\backslash i$ & 1 & 2 & 3 & 4 & 5 & 6\\ 	\hline
			1 & 6 & 6  & 6 & 5 & 5 & 5\\ \hline
			2 & 5  & 5 & 4  & 3 & 1 & 1 \\ \hline
		\end{tabular} \qquad
		%\centering
		\begin{tabular}{ |c|c|c|c|c|c|  }
			\hline
			\multicolumn{6}{|c|}{Exponents and Lengths} \\ 			\hline
			$l\backslash i$ & 1 & 2 & 3 & 4 &5 \\ 	\hline
			1   & $10\ |\ 0$    &  $9\ |\ 0$  &   $7\ |\ 1$ &   $5\ |\ 0$ &   $4\ |\ 0$\\ \hline
			2 & $6\ |\ 0$    & $4\ |\ 1$   & $0\ |\ 2$ & $-3\ |\ 2$ & $-7\ |\ 0$\\ 	\hline
		\end{tabular}
	\end{table}\\
	and  $G(\bs\pi^{\lambda,\mu})$ is connected:
	\begin{equation*}
		\xymatrix@R-1pc@C-0.5pc{
			&  \stackrel{1}{3} \ar[dl]_3\ar[dr]^7\\ \stackrel{1}{2} \ar[rr]^{4} & & \stackrel{2}{3} \ar[rr]^{3} & & \stackrel{2}{4}	}
	\end{equation*}
	Although the underlying directed graph is the same as the one in \Cref{ex:snake}, $L_q(\bs\pi^{\lambda,\mu})$ is not a snake module. Indeed, $\bs\pi^{\lambda,\mu}$ can be constructed as in \eqref{e:snakepi} by using the sequence: $$(4,-4), (4,-2), (3,-1), (3,1), (2,4),  (3,7).$$ 
	However, this is not a snake because $(i_2,m_2)=(4,-2), (i_3,m_3) = (3,-1)$, and $m_3-m_2 = 1 \notin\mathscr R_{4,3}^{1,1}$. 
	One can easily check no reordering of this sequence is a snake.	\endd
\end{ex}

\section{Highest-$\ell$-weight Criteria}\label{s:hlwc}

In this section we prove several criteria for deciding whether a tensor product of simple module is highest-$\ell$-weight or not. In particular, the results proved here can be regarded as the backbone of the arguments in the proof of \Cref{t:toto}. Moreover, they will also be prominently used to prove the main results of \cite{mosi}.

\subsection{Background on Highest-$\ell$-weight Tensor Products}\label{ss:hlwtp}

The following is easily established.

\begin{lem}\label{l:tensV}
	Let $m\in\mathbb Z_{>0}$ and $V_k\in\ \wcal C, 1\le k\le m$ . Then, $V_1\otimes\cdots\otimes V_m$ is highest-$\ell$-weight (resp. simple) only if $V_k$ is highest-$\ell$-weight (resp. simple) for all $1\le k\le m$. \qed
\end{lem}

%\begin{proof}
%	Fix a highest-$\ell$-weight vector $v_k$ of $V_k$, for each $k=1,\cdots, m$. Assume that $V_1\otimes\cdots\otimes V_m$ is highest-$\ell$-weight, that is,
%	\begin{equation*}
	%		\mathcal{U}_q(\mathcal{L}(\mathfrak{g}))(v_1\otimes\cdots\otimes v_m)=V_1\otimes\cdots\otimes V_m.
	%	\end{equation*} 
%	Now, suppose that $V_k$ is not highest-$\ell$-weight, for some $k$, there is,
%	\begin{equation*}
	%		\mathcal{U}_q(\mathcal{L}(\mathfrak{g}))v_k\subsetneq V_k.
	%	\end{equation*}	
%	Hence, we obtain that
%	\begin{equation*}
	%		\mathcal{U}_q(\mathcal{L}(\mathfrak{g}))(v_1\otimes\cdots\otimes v_m)\subset V_1\otimes\cdots\otimes V_{k-1}\otimes\mathcal{U}_q(\mathcal{L}(\mathfrak{g}))v_k\otimes V_{k+1}\otimes\cdots\otimes V_m\subsetneq V_1\otimes\cdots\otimes V_k, 
	%	\end{equation*}
%	yielding a contradiction. \\
%	
%	Now, assume that $V_1\otimes\cdots\otimes V_m$ is simple, while $V_k$ is reducible for some $k$. Let $U$ be a non-trivial submodule of $V_k$. Therefore,
%	\begin{equation*}
	%		V_1\otimes\cdots\otimes V_{k-1}\otimes U\otimes V_{k+1}\otimes\cdots\otimes V_m
	%	\end{equation*} 
%	would be a non-trivial submodule of $V_1\otimes\cdots\otimes V_m$, as well, yielding a contradiction again. \\
%\end{proof}	

\begin{lem}\label{l:subsandq}
	Let $\bs\pi,\bs\pi'\in\mathcal P^+$, $V=L_q(\bs\pi)\otimes L_q(\bs\pi')$, and $W=L_q(\bs\pi')\otimes L_q(\bs\pi)$.
	\begin{enumerate}[(a)]
		\item $V$ contains a submodule isomorphic to $L_q(\bs\varpi), \bs\varpi\in\mathcal P^+$, if and only there exists an epimorphism $W\to L_q(\bs\varpi)$.
		\item If $W$ is highest-$\ell$-weight, the submodule of $V$ generated by its top weight space is simple.
		\item If $V$ is not highest-$\ell$-weight, there exists an epimorphism $V\to L_q(\bs\varpi)$ for some $\bs\varpi\in\mathcal P^+$ such that $\bs\varpi<\bs\pi\bs\pi'$. 
	\end{enumerate}
\end{lem}

\begin{proof}
	Assume we have a monomprphism $L_q(\bs\varpi)\to V$. It follows from \eqref{e:dualses} and \eqref{e:dualtp} that we have an epimorphism
	\begin{equation*}
		L_q(\bs\pi')^*\otimes L_q(\bs\pi)^* \to L_q(\bs\varpi)^*.
	\end{equation*}
	In particular, letting $\psi = \sigma\circ \tau_{q^{r^\vee h^\vee}}$ and using   \eqref{e:tppb} and \eqref{e:cinvses}, we get an epimorphism
	\begin{equation*}
		(L_q(\bs\pi')^*)^\psi\otimes (L_q(\bs\pi)^*)^\psi \to (L_q(\bs\varpi)^*)^\psi.
	\end{equation*}
	One easily checks using  \eqref{e:tau_a},\eqref{e:dualDpoli}, and \eqref{e:kappapb} that the domain of the latter epimorphism is isomorphic to $W$ and $ (L_q(\bs\varpi)^*)^\psi\cong L_q(\bs\varpi)$. The converse in part (a) is proved by reversing this argument. Part (b) is immediate from (a) since the assumption on $W$ implies we have an epimorphism $W\to L_q(\bs\pi\bs\pi')$ and the top weight space of $V$ is one-dimensional and equal to $V_{\bs\pi\bs\pi'}$. 
	
	For proving (c), let $V'$ be the submodule of $V$ generated by its top weight space. The assumption is equivalent to saying that $V'$ is a proper submodule. Since $V$ is finite-dimensional, the set of proper submodules of $V$ containing $V'$ is nonempty and contains a maximal element, say $U$, which is necessarily also maximal in the set of all proper submodules of $V$. Hence, $V/U\cong L_q(\bs\varpi)$ for some $\bs\varpi\in\mathcal P^+$ and, since $V_{\bs\pi\bs\pi'}\subseteq U$, we have $\bs\varpi\ne \bs\pi\bs\pi'$. Since $\bs\varpi$ is an $\ell$-weight of $V$ and  every $\ell$-weight of $V$  is smaller than $\bs\pi\bs\pi'$, we have $\bs\varpi<\bs\pi\bs\pi'$. 
\end{proof}

The following fact is well-known (a proof can be found in \cite{Moura}).

\begin{prop}\label{vnvstar}
	Let $V$ be finite-dimensional $U_q(\tlie g)$-module. Then, $V$ is simple if and only if $V$ and $V^*$ are highest-$\ell$-weight.
\end{prop}

\begin{cor}\label{c:vnvstar}
	Let $\bs{\pi}, \bs\varpi\in\mathcal{P}^+$. Then, $L_q(\bs{\pi})\otimes L_q(\bs\varpi)$ is simple if and only if both $L_q(\bs{\pi})\otimes L_q(\bs\varpi)$ and $L_q(\bs\varpi)\otimes L_q(\bs{\pi})$ are highest-$\ell$-weight.
\end{cor}

\begin{proof}
	If $U:=L_q(\bs{\pi})\otimes L_q(\bs\varpi)$ is simple, \Cref{sinter} implies  $U\cong W:=L_q(\bs{\varpi})\otimes L_q(\bs\pi)$. In particular, $U$ and $W$ are both highest-$\ell$-weight. Conversely, assume  $U$ and $W$ are both highest-$\ell$-weight. Since $U$ is highest-$\ell$-weight, \Cref{vnvstar}, \eqref{e:dualtp}, and \eqref{e:dualDpoli}, imply that it suffices to show
	\begin{equation}\label{e:vnvstar}
		L_q(\bs\varpi^*)\otimes L_q(\bs\pi^*) \quad\text{is highest-$\ell$-weight.}
	\end{equation}
	Using \eqref{e:cinvses} with $V_2=W$, it follows from \eqref{e:tppb} and \eqref{e:kappapb} that $W^\sigma\cong L_q(\bs\varpi^\sigma)\otimes L_q(\bs\pi^\sigma)$ is highest-$\ell$-weight. Setting $a= q^{-r^\vee h^\vee}$, \eqref{e:kappapb} implies $\bs\pi^*=(\bs\pi^\sigma)^{\tau_a}$ and similarly for $\bs\varpi$. The proof of \eqref{e:vnvstar}  is then completed by using  \eqref{e:tppb}, \eqref{e:tau_a} and \eqref{e:cinvses} with $V_2 = W^\sigma$ and $f=\tau_a$.
\end{proof}

\Cref{irred} admits the following analogue.

\begin{thm}\label{cyc}
	Let $S_1,\cdots, S_m$ be simple $U_q(\tlie g)$-modules. Then, $S_1\otimes\cdots\otimes S_m$ is highest-$\ell$-weight if and only if  $S_i\otimes S_j$ is highest-$\ell$-weight for all $1\leq i<j\leq m$. 
\end{thm}

\begin{proof}
	The ``if'' part  is the main result of \cite{Hernandez2}. We now prove the converse by induction on $m$. Thus, suppose $S_1\otimes\cdots\otimes S_m$ is highest-$\ell$-weight and note there is nothing to prove if $m\le 2$. Assume $m>2$ and note \Cref{l:tensV} implies 
	$S_i\otimes\cdots\otimes S_j$ is highest-$\ell$-weight for all $1\le i\le j\le m$. Together with the induction hypothesis, this implies $S_i\otimes S_j$ is highest-$\ell$-weight for all $1\le i<j\le m$ except if $(i,j)=(1,m)$. 
	
	To prove that $S_1\otimes S_m$ is also highest-$\ell$-weight, thus completing the proof, let $S$ be a simple quotient of $T:=S_2\otimes\cdots\otimes S_{m-1}$ and consider the associated epimorphism $\pi:T\to S$. This implies we have an epimorphism
	\begin{equation*}
		S_1\otimes T\otimes S_m\xrightarrow{\quad\operatorname{id}_{S_1}\otimes\pi\otimes\operatorname{id}_{S_m}\quad}S_1\otimes S\otimes S_m
	\end{equation*}
	which, together with the assumption that $S_1\otimes\cdots\otimes S_m$ is highest-$\ell$-weight, implies $S_1\otimes S\otimes S_m$ is highest-$\ell$-weight as well. The inductive argument is completed if $m>3$. 
	
	If $m=3$, let $\lambda_i\in P^+$ be the highest weight of $S_i, 1\le i\le 3$. Since $S_1\otimes S_2$ is highest-$\ell$-weight, \Cref{l:subsandq}(b) implies the submodule $M$ generated by the top weight space of $S_2\otimes S_1$ is simple and we have a monomorphism
	\begin{equation*}
		M\otimes S_3 \to S_2\otimes S_1\otimes S_3
	\end{equation*} 
	If $S_1\otimes S_3$ were not highest-$\ell$-weight, an application of \Cref{l:subsandq}(c) would gives us an epimorphism
	\begin{equation*}
		S_1\otimes S_3 \to N
	\end{equation*}
	where $N$ is a simple module whose highest weight $\lambda$ satisfies $\lambda<\lambda_1+\lambda_3$.  \Cref{c:nonzeromorph} says these maps can be used to obtain a nonzero map
	\begin{equation*}
		M\otimes S_3 \to S_2\otimes N.
	\end{equation*}
	The highest weight of $M\otimes S_3$ is $\lambda_1+\lambda_2+\lambda_3$ while that of $S_2\otimes N$ is $\lambda_2+\lambda$. To reach a contradiction, it then suffices to show $M\otimes S_3$ is highest-$\ell$-weight.
	
	Indeed, we have an epimorphism $S_1\otimes S_2\to M$ and, hence, an epimorphism
	\begin{equation*}
		S_1\otimes S_2\otimes S_3 \to M\otimes S_3.
	\end{equation*}
	The assumption that $S_1\otimes S_2\otimes S_3$ is highest-$\ell$-weight then implies that so is  $M\otimes S_3$, as desired. 
\end{proof}

The next corollary will be used often during the main proofs.

\begin{cor}\label{l:hlwquot}
	Given $\bs{\pi},\widetilde{\bs{\pi}}\in\mathcal{P}^+$, $L_q(\bs{\pi})\otimes L_q(\widetilde{\bs{\pi}})$ is highest-$\ell$-weight if there exist 
	$\bs\pi^{(k)}\in\mathcal P^+, 1\le k\le m$, $\widetilde{\bs\pi}^{(k)}\in\mathcal P^+, 1\le k\le \tilde m$, such that  
	\begin{equation*}
		\bs\pi = \prod_{k=1}^m \bs\pi^{(k)}, \quad  \widetilde{\bs\pi} = \prod_{k=1}^{\tilde m} \widetilde{\bs\pi}^{(k)}, 
	\end{equation*}
	and the following tensor products are highest-$\ell$-weight:
	\begin{equation*}
		L_q(\bs\pi^{(k)})\otimes L_q(\bs\pi^{(l)}), \quad L_q(\widetilde{\bs\pi}^{(k)})\otimes L_q(\widetilde{\bs\pi}^{(l)}), \quad\text{for } k<l,\text{ and}\quad L_q(\bs\pi^{(k)})\otimes L_q(\widetilde{\bs\pi}^{(l)}) \quad\text{for all }k,l.
	\end{equation*}
	Moreover, if all these tensor products are irreducible, then so is $L_q(\bs{\pi})\otimes L_q(\widetilde{\bs{\pi}})$.
\end{cor}

\begin{proof}
	It follows from \Cref{cyc} that
	\begin{equation*}
		W:= L_q(\bs\pi^{(1)})\otimes \cdots\otimes L_q(\bs\pi^{(m)}), \quad \tilde W:= L_q(\widetilde{\bs\pi}^{(1)})\otimes \cdots\otimes L_q(\widetilde{\bs\pi}^{(\tilde m)}), \quad\text{and}\quad W\otimes \tilde W
	\end{equation*}
	are highest-$\ell$-weight. Therefore, we have epimorphisms $p:W\to L_q(\bs\pi), \tilde p:\tilde W\to L_q(\widetilde{\bs\pi})$, and, hence, $p\otimes \tilde p:W\otimes \tilde W\to   L_q(\bs{\pi})\otimes L_q(\widetilde{\bs\pi})$. For the last claim, the extra assumption implies we reach the same conclusion with the tensor products above in reversed order. Hence, we are done by \Cref{c:vnvstar}.
\end{proof}

We are now able to prove  \Cref{p:primeconect}.

\begin{proof}[Proof of  \Cref{p:primeconect}]
	Let $G'$ and $G''$ be non empty unions of distinct connected components of $G(\bs\pi)$.  Let also $\bs\pi',\bs\pi''\in\mathcal P^+$ be such that $G'=G(\bs\pi')$ and $G''=G(\bs\pi'')$. If $\bs\omega'$ is a vertex of $G'$ and $\bs\omega''$ is a vertex of $G''$, they belong to different connected components of $G$ and, hence, $L_q(\bs\omega')\otimes L_q(\bs\omega'')$ is simple. It then follows from  \Cref{l:hlwquot} that $L_q(\bs\pi')\otimes L_q(\bs\pi'')$ is simple. An obvious inductive argument proves  \Cref{p:primeconect}.
\end{proof}

\subsection{Proof of \Cref{p:buildupprimesmono}}

\begin{lem}\label{l:vertexremoval}
	Let $\bs\pi\in\mathcal P^+$  and suppose $\bs\omega$ is an extremal vertex in $G(\bs\pi)$. Let $\bs\varpi = \bs\pi\bs\omega^{-1}$ and assume there exists a non-trivial factorization $\bs\varpi= \bs{\varpi}^{(1)}\bs{\varpi}^{(2)}$ such that
	\begin{equation*}
		L_q(\bs{\varpi})\cong L_q(\bs{\varpi}^{(1)})\otimes L_q(\bs{\varpi}^{(2)})
	\end{equation*}
	and  every $q$-factor of $\bs\pi$ adjacent to $\bs\omega$  in $G(\bs\pi)$ lies in $G(\bs\varpi^{(1)})$. Then,
	\begin{equation*}
		L_q(\bs{\pi})\cong L_q(\bs{\omega}\bs{\varpi}^{(1)})\otimes L_q(\bs{\varpi}^{(2)}).
	\end{equation*}
\end{lem}

\begin{proof}
	Up to arrow dualization, we can assume $\bs\omega$ is a source and, hence,
	\begin{equation*}
		L_q(\bs\omega)\otimes L_q(\bs\omega') \quad\text{is highest-$\ell$-weight for every other $q$-factor $\bs\omega'$ of $\bs\pi$.}
	\end{equation*}
	By  \Cref{c:vnvstar}, it suffices to prove that both 
	\begin{equation*}
		L_q(\bs{\omega}\bs{\varpi}^{(1)})\otimes L_q(\bs{\varpi}^{(2)}) \qquad\text{and}\qquad L_q(\bs{\varpi}^{(2)})\otimes L_q(\bs{\omega}\bs{\varpi}^{(1)})  
	\end{equation*}
	are highest-$\ell$-weight. 	
	Consider the tensor product
	\begin{equation*}
		W_1=L_q(\bs{\omega})\otimes L_q(\bs{\varpi}^{(1)})\otimes L_q(\bs{\varpi}^{(2)}).
	\end{equation*}
	We claim that $W_1$ is highest-$\ell$-weight. Indeed, we are assuming that $L_q(\bs{\varpi}^{(1)})\otimes L_q(\bs{\varpi}^{(2)})$ is simple and the hypothesis on $G(\bs\varpi^{(1)})$ together with  \Cref{l:hlwquot} implies $L_q(\bs{\omega})\otimes L_q(\bs{\varpi}^{(2)})$ is also simple. 
	By condition (i) and  \Cref{l:hlwquot},   $L_q(\bs{\omega})\otimes L_q(\bs{\varpi}^{(1)})$ is  is highest-$\ell$-weight. Hence,  \Cref{cyc} implies $W_1$ is highest-$\ell$-weight, as well as its quotient $L_q(\bs{\omega}\bs{\varpi}^{(1)})\otimes L_q(\bs{\varpi}^{(2)})$. These facts also imply $L_q(\bs{\varpi}^{(2)})\otimes L_q(\bs{\omega}) \otimes L_q(\bs{\varpi}^{(1)})$ is highest-$\ell$-weight, showing that $L_q(\bs{\varpi}^{(2)})\otimes L_q(\bs{\omega}\bs{\varpi}^{(1)})$ is highest-$\ell$-weight.
\end{proof}

\begin{lem}\label{l:vertexremoval'}
	Suppose $L_q(\bs\pi)$ is prime and that $\bs\omega$ is an extremal vertex in $G(\bs\pi)$. 
	Let $\bs\varpi = \bs\pi\bs\omega^{-1}$. Then, either $L_q(\bs\varpi)$ is prime or there exists a non-trivial factorization $\bs\varpi= \bs{\varpi}^{(1)}\bs{\varpi}^{(2)}$ such that
	\begin{equation*}
		L_q(\bs{\varpi})\cong L_q(\bs{\varpi}^{(1)})\otimes L_q(\bs{\varpi}^{(2)})
	\end{equation*}
	and both $\bs{\varpi}^{(1)}$ and $\bs{\varpi}^{(2)}$ contain $q$-factors of $\bs\pi$ adjacent to $\bs\omega$ in $G(\bs\pi)$.
\end{lem}

\begin{proof}
	Immediate from  \Cref{l:vertexremoval}.
\end{proof}

We can now prove \Cref{p:buildupprimesmono}. 	Write $G=G(\bs\pi)$, let $\bs\omega=v$, and $\bs\varpi=\bs\pi\bs\omega^{-1}$. In particular, $G_{\mathcal V\setminus\{v\}}=G(\bs\varpi)$. 	
Since $G$ is connected by \Cref{p:primeconect}, $v$ must be monovalent and, hence, there exists a unique $q$-factor $\bs\omega'$ of $\bs\varpi$ such that $L_q(\bs\omega)\otimes L_q(\bs\omega')$ is reducible. In particular,  $\bs\omega$ is extremal in $G$. The proposition then follows immediately from  \Cref{l:vertexremoval'}.

\subsection{A Key Highest-$\ell$-weight Criterion and \Cref{gencp}} We now establish a criterion for a tensor product to be highest-$\ell$-weight which is the heart of the proof of \Cref{gencp} and will also be used to deduce further criteria which will be used in the proof of \Cref{t:toto}.

\begin{prop}\label{p:killhlw}
	Let $\bs{\lambda},\bs{\nu}\in\mathcal{P}^+$ and  $V=L_q(\bs{\lambda})\otimes L_q(\bs{\nu})$. Then, $V$ is highest-$\ell$-weight provided there exists  $\bs{\mu}\in\mathcal{P}^+$ such that one of the following conditions holds:
	\begin{enumerate}[(i)]
		\item $L_q(\bs{\lambda}\bs{\mu})\otimes L_q(\bs{\nu})$ and $L_q(\bs{\lambda})\otimes L_q(\bs{\mu})$ are both highest-$\ell$-weight;
		\item $L_q(\bs{\lambda})\otimes L_q(\bs{\mu}\bs{\nu})$ and $L_q(\bs{\mu})\otimes L_q(\bs{\nu})$ are both highest-$\ell$-weight. 	
	\end{enumerate}
\end{prop}

\begin{proof} We write the details only in case (i) holds since the other case is similar. So, assume that $V$ is not highest-$\ell$-weight. In particular, there exists $\bs{\xi}\in\mathcal{P}^+$ such that $\bs{\xi}<\bs{\lambda}\bs{\nu}$ together with an epimorphism $V\xrightarrow{f} L_q(\bs{\xi})$. Therefore,there also exists an epimorphism
	\begin{equation*}
		L_q(\bs{\mu})\otimes V=L_q(\bs{\mu})\otimes L_q(\bs{\lambda})\otimes L_q(\bs{\nu})\xrightarrow{\operatorname{id}_{L_q(\bs{\mu})}\otimes f}L_q(\bs{\mu})\otimes L_q(\bs{\xi}).
	\end{equation*}
	On the other hand, since $L_q(\bs{\lambda})\otimes L_q(\bs{\mu})$ is highest-$\ell$-weight, there exist monomorphisms
	\begin{equation*}
		L_q(\bs{\lambda}\bs{\mu})\xrightarrow{g} L_q(\bs{\mu})\otimes L_q(\bs{\lambda})
		\qquad\text{and}\qquad 
		L_q(\bs{\lambda}\bs{\mu})\otimes L_q(\bs{\nu})\xrightarrow{g\otimes\operatorname{id}_{L_q(\bs{\nu})}} L_q(\bs{\mu})\otimes L_q(\bs{\lambda})\otimes L_q(\bs{\nu}).
	\end{equation*}
	\Cref{c:nonzeromorph} implies the composition
	\begin{equation*}
		L_q(\bs{\lambda}\bs{\mu})\otimes L_q(\bs{\nu})\xrightarrow{g\otimes\operatorname{id}_{L_q(\bs{\nu})}}L_q(\bs{\mu})\otimes L_q(\bs{\lambda})\otimes L_q(\bs{\nu})\xrightarrow{\operatorname{id}_{L_q(\bs{\mu})}\otimes f}L_q(\bs{\mu})\otimes L_q(\bs{\xi})
	\end{equation*}
	is non-zero. Then, since $L_q(\bs{\lambda}\bs{\mu})\otimes L_q(\bs{\nu})$ is highest-$\ell$-weight, the image of its highest-$\ell$-weight vector under this composition must be a non-zero vector with $\ell$-weight $\bs{\lambda}\bs{\mu}\bs{\nu}$. However, since $\bs{\xi}<\bs{\lambda}\bs{\nu}$, we also have  $\bs{\xi}\bs{\mu}<\bs{\lambda}\bs{\mu}\bs{\nu}$ and, hence, there is no vector in $L_q(\bs{\mu})\otimes L_q(\bs{\xi})$ with $\ell$-weight $\bs{\lambda}\bs{\mu}\bs{\nu}$, yielding a contradiction. 	
\end{proof}

One of the applications of the above result gives a partial answer to the following question. Suppose $\bs\pi,\bs\pi',\bs\varpi\in\mathcal P^+$ are such that $L_q(\bs\pi)\otimes L_q(\bs\pi')$ is highest-$\ell$-weight and $\bs\varpi$ divides $\bs\pi$. Under which further assumptions $L_q(\bs\pi\bs\varpi^{-1})\otimes L_q(\bs\pi')$ is also highest-$\ell$-weight? A similar question can be made in the case $\bs\varpi$ divides $\bs\pi'$.

\begin{cor}\label{c:killsinksource}
	Suppose $G$ and $G'$ are pseudo-$q$-factorization graphs over $\bs{\pi},\bs{\pi}'\in\mathcal{P}^+$, respectively. Assume $L_q(\bs{\pi})\otimes L_q(\bs{\pi}')$ is highest-$\ell$-weight and let $\bs{\omega}\in\mathcal{V}_G$ and $\bs{\omega}'\in\mathcal{V}_{G'}$. Then,
	\begin{enumerate}[(a)]
		\item If $\bs{\omega}$ is a sink in $\mathcal{V}_G$, $L_q(\bs{\pi}\bs{\omega}^{-1})\otimes L_q(\bs{\pi}')$ is highest-$\ell$-weight.
		\item If $\bs{\omega}'$ is a source in $\mathcal{V}_{G'}$, $L_q(\bs{\pi})\otimes L_q(\bs{\pi}'(\bs{\omega}')^{-1})$ is highest-$\ell$-weight.
	\end{enumerate}
\end{cor}

\begin{proof} If $\bs{\omega}$ is a sink in $\mathcal{V}_G$, \Cref{l:hlwquot} implies that $L_q(\bs{\pi}\bs{\omega}^{-1})\otimes L_q(\bs{\omega})$ is highest-$\ell$-weight, showing that (i) of \Cref{p:killhlw} holds with $\bs{\lambda}=\bs{\pi}\bs{\omega}^{-1}$, $\bs{\mu}=\bs{\omega}$ and $\bs{\nu}=\bs{\pi}'$. Part (b) follows similarly. 	
\end{proof}

\begin{cor}\label{c:source<-sink} Let $\bs{\pi}',\bs{\pi}''\in\mathcal{P}^+$ with dissociate factorizations and $\bs{\pi}=\bs{\pi}'\bs{\pi}''$. Let also 
	$G=G(\bs{\pi}), G'=G(\bs{\pi}'), G''=G(\bs{\pi}'')$, and suppose $\bs{\omega}',\bs{\omega}''\in\mathcal{P}^+$ satisfy 
	\begin{equation*}
		\bs{\omega}' \text{ is a source in } G', \quad \bs{\omega}'' \text{ is a sink in } G'', \quad\text{and}\quad (\bs{\omega}'',\bs{\omega}') \in \mathcal{A}_G.
	\end{equation*}
	Then, $L_q(\bs{\pi}')\otimes L_q(\bs{\pi}'')$ is not highest-$\ell$-weight.
\end{cor}

\begin{proof}
	The first two assumptions, together with \Cref{l:hlwquot}, imply that
	\begin{equation*}
		L_q(\bs{\omega}')\otimes L_q(\bs{\pi}'(\bs{\omega}')^{-1})\quad\textrm{and}\quad L_q(\bs{\pi}''(\bs{\omega}'')^{-1})\otimes L_q(\bs{\omega}'') \quad\text{are highest-$\ell$-weight.}
	\end{equation*}
	On the other hand, the last assumption implies that $L_q(\bs{\omega}')\otimes L_q(\bs{\omega}'')$ is not highest-$\ell$-weight. Letting $\bs\lambda=\bs\omega',\bs\mu=\bs\pi'(\bs\omega')^{-1}$ and $\bs\nu=\bs\omega''$,  \Cref{p:killhlw}(a) implies  $L_q(\bs{\pi}')\otimes L_q(\bs{\omega}'')$
	is not highest-$\ell$-weight. Then, an application of \Cref{p:killhlw}(b) with $\bs\mu = \bs{\pi}''(\bs{\omega}'')^{-1}, \bs\nu = \bs{\omega}''$, and $\bs\lambda=\bs\pi'$ completes the proof. 
\end{proof}

\Cref{critredcutext} is easily deduced from \Cref{c:source<-sink}. Moreover, we can also give the:

\begin{proof}[Proof of \Cref{gencp}]
	Let $\bs\pi\in\mathcal P^+$ and assume $G(\bs\pi)$ is a totally ordered line. By \Cref{p:qftp}, we need to show that $L_q(\bs\pi')\otimes L_q(\bs\pi'')$ is reducible for any nontrivial decomposition $\bs\pi=\bs\pi'\bs\pi''$ such that $G=G(\bs\pi')$ and $G'=G(\bs\pi'')$ are connected and, hence, also totally ordered line. Without loss of generality, assume $G'$ contains the sink of $G$. Then, if $\bs\omega'$ is the source of $G'$ and $\bs\omega''$ is the sink of $G''$, the fact that $G, G'$, and $G''$ are totally ordered lines implies that $(\bs\omega'',\bs\omega') \in\mathcal A_G$ and, hence $L_q(\bs\pi')\otimes L_q(\bs\pi'')$ is not simple by  \Cref{c:source<-sink}. 
\end{proof}

\subsection{Highest-$\ell$-weight Criteria via Monotonic Paths}
Recall \eqref{e:Neibor}.

\begin{lem}\label{l:hlwlpm}
	Let $\bs\pi\in\mathcal{P}^+$ and suppose $G$ is a pseudo $q$-factorization graph over $\bs\pi$. Given $\bs\varpi\in\mathcal{V}_G$, consider
	\begin{equation*}
		\bs\pi_\pm=\prod\limits_{\bs{\omega}\in\mathcal{N}_G^\pm(\bs\varpi)}\bs{\omega}'.
	\end{equation*}	
	Then, the following tensor products are highest-$\ell$-weight: 
	\begin{equation*}
		L_q(\bs\pi_+)\otimes L_q(\bs\pi\bs\pi_+^{-1})\quad\textrm{and}\quad L_q(\bs\pi\bs\pi_-^{-1})\otimes L_q(\bs\pi_-).
	\end{equation*}
\end{lem}

\begin{proof}
	For the first tensor product, \Cref{l:hlwquot}. implies it suffices to shows $L_q(\bs{\omega})\otimes L_q(\bs{\omega}')$ is highest-$\ell$-weight  for any $\bs{\omega}\in\mathcal{N}_G^+(\bs\varpi)$ and  any $\bs{\omega}'\in\mathcal{V}_G\setminus\mathcal{N}_G^+(\bs\varpi)$. Indeed, if this failed for some choice of such $\bs\omega$ and $\bs\omega'$, it would follow that $a:=(\bs{\omega}',\bs{\omega})\in\mathcal{A}_G$. Since $\bs{\omega}\in\mathcal{N}_G^+(\bs\varpi)$, we can chose $\rho\in\mathscr P^+_{\bs\varpi,\bs\omega}$ and it would follow that $\rho*a\in \mathscr P^+_{\bs\varpi,\bs\omega'}$, contradicting the assumption $\bs\omega'\notin \mathcal{N}_G^+(\bs\varpi)$. The second tensor product is treated similarly. 
\end{proof}

The next lemma will play a role in the proofs of \Cref{p:neighdualcrit} and \Cref{t:toto}.

\begin{lem}\label{c:killpi-pi'+}
	Let $\bs{\pi},\bs{\pi'}\in\mathcal{P}^+$ and suppose $G$ and $G'$ are  pseudo $q$-factorization graphs over  $\bs{\pi}$ and $\bs{\pi'}$, respectively. Let $\bs{\varpi}\in\mathcal{V}_G$ and $\bs{\varpi}'\in\mathcal{V}_{G'}$ and consider
	\begin{equation*}
		\bs{\pi}_+=\prod\limits_{\bs{\omega}\in\mathcal{N}^+_G(\bs{\varpi})}\bs{\omega}\quad\textrm{and}\quad\bs{\pi}'_-=\prod\limits_{\bs{\omega}\in\mathcal{N}^-_{G'}(\bs{\varpi}')}\bs{\omega}.
	\end{equation*}
	If $V=L_q(\bs{\pi})\otimes L_q(\bs{\pi}')$ is highest-$\ell$-weight, so are the following tensor products:
	\begin{equation*}
		L_q(\bs{\pi}_+)\otimes L_q(\bs{\pi}'),\quad L_q(\bs{\pi})\otimes L_q(\bs{\pi}'_-),\quad\textrm{and}\quad L_q(\bs{\pi}_+)\otimes L_q(\bs{\pi}'_-).
	\end{equation*}
\end{lem}

\begin{proof}
	Let $\bs\lambda=\bs\pi_+, \bs\mu=\bs\pi\bs\pi_+^{-1}$, and $\bs\nu = \bs\pi'$. By assumption, $L_q(\bs\lambda\bs\mu)\otimes L_q(\bs\nu)$ is highest-$\ell$-weight, while \Cref{l:hlwlpm} implies that so is  $L_q(\bs\lambda)\otimes L_q(\bs\mu)$. The claim about the first tensor product then follows from \Cref{p:killhlw}. The other two cases are treated similarly.
\end{proof}

\subsection{A Highest-$\ell$-weight Criterion from Duality} We now deduce the main technical part behind  the proof of \Cref{p:critprimedual} which will also be used for proving \Cref{t:toto}.

\begin{prop}\label{p:killdualhlw}
	Let $\bs{\lambda},\bs{\mu},\bs{\nu}\in\mathcal{P}^+$. Let also $V = L_q(\bs{\lambda})\otimes L_q(\bs{\nu})^*$,
	\begin{equation*}
		T_1=L_q(\bs{\lambda}\bs{\mu})\otimes L_q(\bs{\nu}),\quad U_1=L_q(\bs{\lambda})\otimes L_q(\bs{\mu}), \quad W_1=L_q(\bs{\mu})\otimes L_q(\bs{\nu}),
	\end{equation*}	
	\begin{equation*}
		T_2=L_q(\bs{\lambda})\otimes L_q(\bs{\mu}\bs{\nu}),\quad U_2=L_q(\bs{\mu})\otimes L_q(\bs{\nu}),\quad W_2=L_q(\bs{\lambda})\otimes L_q(\bs{\mu}).
	\end{equation*}
	Then, $W_i$ is highest-$\ell$-weight provided $T_i$ and $U_i$ are highest-$\ell$-weight,  $i\in\{1,2\}$, and $V$ is simple.
\end{prop}

\begin{proof}
	We write the details for $i=1$ only. Since $V$ is simple, we have $V\cong L_q(\bs{\nu})^*\otimes L_q(\bs{\lambda})$. 	
	If $W_1$ were not highest-$\ell$-weight, there would exist $\bs{\xi}\in\mathcal{P}^+$ such that $\bs{\xi}<\bs{\mu}\bs{\nu}$, together with an epimorphism $W_1\xrightarrow{f} L_q(\bs{\xi})$. Then, \eqref{e:frobrec} implies there would also exist monomorphisms
	\begin{equation*}
		L_q(\bs{\mu})\xrightarrow{g} L_q(\bs{\xi})\otimes L_q(\bs{\nu})^*	\qquad\text{and}\qquad
		L_q(\bs{\mu})\otimes L_q(\bs{\lambda})\xrightarrow{g\otimes\operatorname{id}_{L_q(\bs{\lambda})}} L_q(\bs{\xi})\otimes L_q(\bs{\nu})^*\otimes L_q(\bs{\lambda})\cong L_q(\bs{\xi})\otimes	V.
	\end{equation*}
	On the other hand, since $U_1$ is highest-$\ell$-weight, there exists a monomorphism
	\begin{equation*}
		L_q(\bs{\lambda}\bs{\mu})\xrightarrow{h} L_q(\bs{\mu})\otimes L_q(\bs{\lambda}).
	\end{equation*}
	Therefore, the following composition would also be injective
	\begin{equation*}
		L_q(\bs{\lambda}\bs{\mu})\xrightarrow{h} L_q(\bs{\mu})\otimes L_q(\bs{\lambda})\xrightarrow{g\otimes\operatorname{id}_{L_q(\bs{\lambda})}}L_q(\bs{\xi})\otimes V,
	\end{equation*}
	yielding a monomorphism
	\begin{equation*}
		L_q(\bs{\lambda}\bs{\mu})\hookrightarrow L_q(\bs{\xi})\otimes L_q(\bs{\lambda})\otimes L_q(\bs{\nu})^*.
	\end{equation*} 
	Finally, by \eqref{e:frobrec}, this implies there would exist a non-zero homomorphism
	\begin{equation*}
		T_1=L_q(\bs{\lambda}\bs{\mu})\otimes L_q(\bs{\nu})\rightarrow L_q(\bs{\xi})\otimes L_q(\bs{\lambda}).
	\end{equation*}
	Since $\bs{\xi}<\bs{\mu}\bs{\nu}$ and, therefore, $\bs{\lambda}\bs{\xi}<\bs{\lambda}\bs{\mu}\bs{\nu}$, this yields a contradiction with the assumption that $T_1$ is highest-$\ell$-weight. 	
\end{proof}

\begin{cor}\label{c:killsinksourcedual}
	Suppose $G$ and $G'$ are pseudo-$q$-factorization graphs over $\bs{\pi},\bs{\pi}'\in\mathcal{P}^+$, respectively, and assume $L_q(\bs{\pi})\otimes L_q(\bs{\pi}')$ is highest-$\ell$-weight. 	
	\begin{enumerate}[(a)]
		\item If $\bs{\omega}\in\mathcal{V}_G$ is a source in $G$ such that $L_q(\bs{\omega}')^*\otimes L_q(\bs{\omega})$ is simple  for any $\bs{\omega}'\in\mathcal{V}_{G'}$, then  $L_q(\bs{\pi}\bs{\omega}^{-1})\otimes L_q(\bs{\pi}')$ 	is highest-$\ell$-weight. 
		\item  If $\bs{\omega}'\in\mathcal{V}_{G'}$ is a sink in $G'$ such that $L_q(\bs{\omega}')\otimes{} ^*L_q(\bs{\omega})$ is simple for any $\bs{\omega}\in\mathcal{V}_G$, then $L_q(\bs{\pi})\otimes L_q(\bs{\pi}'(\bs{\omega}')^{-1})$ is highest-$\ell$-weight.
	\end{enumerate} 
\end{cor}

\begin{proof}
	If $\bs{\omega}\in\mathcal{V}_G$ is a source in $G$,  \Cref{l:hlwquot} implies $L_q(\bs{\omega})\otimes L_q(\bs{\pi}\bs{\omega}^{-1})$ is highest-$\ell$-weight. In its turn, since $L_q(\bs{\omega}')^*\otimes L_q(\bs{\omega})$ is simple for any $\bs{\omega}'\in\mathcal{V}_{G'}$,  \Cref{l:hlwquot} implies  $L_q(\bs{\pi}')^*\otimes L_q(\bs{\omega})$ is simple. Part (a) then follows from  \Cref{p:killdualhlw} with $\bs{\lambda}=\bs{\omega}$, $\bs{\mu}=\bs{\pi}\bs{\omega}^{-1}, \bs{\nu}=\bs{\pi}'$, and $i=1$. Part (b) is proved similarly.
\end{proof}

 The latter criteria leads to the following criterion for proving that a tensor product is not highest-$\ell$-weight.
\begin{prop}\label{p:neighdualcrit}
	Assume $\bs{\pi},\bs{\pi}'\in\mathcal{P}^+$ have dissociate $q$-factorizations and let $G=G(\bs{\pi}), G'=G(\bs{\pi}'), G''=G(\bs{\pi}\bs{\pi}')$. Suppose that there exist $\bs{\varpi}\in\mathcal{V}_G,\;\bs{\varpi}'\in\mathcal{V}_{G'}$ such that $(\bs{\varpi}',\bs{\varpi})\in\mathcal{A}_{G''}$ and
	\begin{equation}\label{e:neighdualcrit}
		L_q(\bs{\omega})\otimes L_q(\bs{\omega}')^* \quad\text{is simple}\quad \forall\  (\bs{\omega},\bs{\omega}')\in\mathcal{N}^+_G(\bs{\varpi})\times\mathcal{N}^-_{G'}(\bs{\varpi}')\setminus\{(\bs{\varpi},\bs{\varpi}')\}.
	\end{equation}
	Then, $L_q(\bs{\pi})\otimes L_q(\bs{\pi}')$ is not highest-$\ell$-weight. 
\end{prop}

\begin{proof}
	Suppose $
	L_q(\bs{\pi})\otimes L_q(\bs{\pi}')$ is highest-$\ell$-weight and let $\bs\pi_+,\bs\pi_-'$ be defined as in \Cref{c:killpi-pi'+}. In particular, $L_q(\bs{\pi}_+)\otimes L_q(\bs{\pi}'_-)$ is also highest-$\ell$-weight. On the other hand, it follows from \eqref{e:neighdualcrit} and \Cref{l:hlwquot} that
	\begin{equation*}
		L_q(\bs{\pi}_+)\otimes L_q(\bs{\pi}'_-(\bs{\varpi}')^{-1})^* \quad\text{is simple.}
	\end{equation*}
	\Cref{p:killdualhlw} with  $\bs{\lambda}=\bs{\pi}_+, \bs{\mu}=\bs{\varpi}',  \bs{\nu}=\bs{\pi}'_-(\bs{\varpi}')^{-1}$ and $i=2$ then implies that
	\begin{equation*}
		L_q(\bs{\pi}_+)\otimes L_q(\bs{\varpi}')  \quad\text{is highest-$\ell$-weight.}
	\end{equation*}
	Assumption \eqref{e:neighdualcrit} also implies $L_q(\bs{\omega})\otimes L_q(\bs{\varpi}')^*$ is simple for all $\bs{\omega}\in\mathcal{N}^+_G(\bs{\varpi})\setminus\{\bs{\omega}\}$ and, hence, it follows from \Cref{l:hlwquot}  that 
	\begin{equation*}
		L_q(\bs{\pi}_+\bs{\varpi}^{-1})\otimes L_q(\bs{\varpi}')^* \quad\text{is simple.}
	\end{equation*}
	Therefore, Proposition \ref{p:killdualhlw} with $\bs{\lambda}=\bs{\pi}_+\bs{\varpi}^{-1}, \bs{\mu}=\bs{\varpi}, \bs{\nu}=\bs{\varpi}'$, and $i=1$ implies that
	\begin{equation*}
		L_q(\bs{\varpi})\otimes L_q(\bs{\varpi}') \quad\text{is highest-$\ell$-weight,}
	\end{equation*}
	which contradicts the assumption $(\bs{\varpi}',\bs{\varpi})\in\mathcal{A}_{G''}$. 
\end{proof}

We are ready for:

\begin{proof}[Proof of \Cref{p:critprimedual}]
	Assume that $G$ is not prime, so we have a nontrivial factorization 
	\begin{equation}\label{eq:pipi'}
		L_q(\bs{\pi})\cong L_q(\bs{\pi}')\otimes L_q(\bs{\pi}'').
	\end{equation}
	By \Cref{p:qftp}, $\bs{\pi}'$ and $\bs{\pi}''$ have dissociate $q$-factorizations and, hence, if $G'=G(\bs\pi')$ and $G''=G(\bs\pi'')$, $(G',G'')$ is a cut of $G$. Therefore, by assumption there exist $\bs{\varpi}'\in\mathcal{V}_{G'}$ and $\bs{\varpi}''\in\mathcal{V}_{G''}$ such that either (i) or (ii) holds. If it is (i), \Cref{p:neighdualcrit} implies that $L_q(\bs{\pi}')\otimes L_q(\bs{\pi}'')$ is not highest-$\ell$-weight, yielding a contradiction. If it is (ii), the same conclusion is reached by interchanging the roles of $\bs\pi'$ and $\bs\pi''$.   
\end{proof}

\subsection{Preservation of Highest-$\ell$-weight Property  under Vertex Removal}\label{ss:remove}

We now present a couple of  criteria, in terms of the sets $\mathscr R_{i,j}^{r,s}$, for removing a vertex from a pseudo $q$-factorization graph in a way that its tensor product with another graph remains associated to a highest-$\ell$-weight tensor product.    They will play a role in the proof of \Cref{p:killfactormid} which, in turn, will play a small role in the proof of \Cref{t:toto}.

\begin{lem}\label{p:killfactorp}
	Let $\bs\pi,\bs\pi'\in\mathcal P^+$ be such that $L_q(\bs\pi)\otimes L_q(\bs\pi')$ is highest-$\ell$-weight and let $G$ and $G^{\prime}$ pseudo $q$-factorization graphs over $\bs\pi$ and $\bs\pi'$, respectively. If $\bs\omega_{i,a,r}$ is a vertex in $G$, then  $L_q(\bs\pi\bs\omega_{i,a,r}^{-1})\otimes L_q(\bs\pi')$ is also highest-$\ell$-weight if there exists $s\in\mathbb{Z}_{>0}$ such that, for each other vertex $\bs\omega_{i',aq^m,r'}$ of $G\otimes G'$, we have
	\begin{equation}\label{e:killfactorp}
		m-d_i(r+s)\notin\mathscr{R}_{i,i'}^{s,r'} \quad\text{as well as}\quad |m-d_is|\notin\mathscr{R}_{i,i'}^{r+s,r'} \quad\text{if}\quad\bs\omega_{i',aq^m,r'}\in\mathcal{V}_G.
	\end{equation}
	Similarly, if $\bs\omega_{i,a,r}$ is a vertex in $G'$, then  $L_q(\bs\pi)\otimes L_q(\bs\pi'\bs\omega_{i,a,r}^{-1})$ is also highest-$\ell$-weight if there exists $s\in\mathbb{Z}_{>0}$ such that, for each other vertex $\bs\omega_{i',aq^m,r'}$ of $G\otimes G'$, we have
	\begin{equation*}
		-d_i(r+s)-m\notin\mathscr{R}_{i,i'}^{s,r'} \quad\text{as well as}\quad 
		|m+d_is|\notin\mathscr{R}_{i,i'}^{r+s,r'} \quad\textrm{if}\quad\bs\omega_{i',aq^m,r'}\in\mathcal{V}_{G^{\prime}}.
	\end{equation*}
\end{lem}

\begin{proof}
	We write down the details for the first statement only, as the second is similar. Given $s\in\mathbb Z_{\ge 0}$, consider 
	\begin{equation*}
		\bs\omega_{i,aq_i^s,r+s} = \bs\omega_{i,a,r}\bs\omega_{i,aq_i^{r+s},s}.
	\end{equation*}
	In particular,
	\begin{equation}\label{killfactorpoly}
		\bs\omega_{i,aq_i^{r+s},s}\bs\pi= \bs\omega_{i,aq_i^s,r+s}\bs\pi\bs\omega_{i,a,r}^{-1}.
	\end{equation}
	\Cref{cyc}, the assumption that  $L_q(\bs\pi)\otimes L_q(\bs\pi')$ is highest-$\ell$-weight, and the first condition in \eqref{e:killfactorp} imply
	\begin{equation*}
		L_q( \bs\omega_{i,aq_i^{r+s},s})\otimes L_q(\bs\pi)\otimes L_q(\bs\pi')
	\end{equation*}
	is highest-$\ell$-weight and, moreover, in light of \eqref{killfactorpoly}, we have an epimorphism
	\begin{equation}\label{killfactorepi}
		L_q( \bs\omega_{i,aq_i^{r+s},s})\otimes L_q(\bs\pi)\otimes L_q(\bs\pi') \to L_q( \bs\omega_{i,aq_i^s,r+s}(\bs\pi\bs\omega_{i,a,r}^{-1}))\otimes L_q(\bs\pi').
	\end{equation}
	The second condition in \eqref{e:killfactorp}  together with \Cref{l:hlwquot} implies 
	\begin{equation}\label{killfactorsimple}
		L_q( \bs\omega_{i,aq_i^s,r+s}(\bs\pi\bs\omega_{i,a,r}^{-1})) \cong  L_q( \bs\omega_{i,aq_i^s,r+s})\otimes L_q(\bs\pi\bs\omega_{i,a,r}^{-1}),
	\end{equation}
	from where the lemma follows.  Indeed, if $L_q(\bs\pi\bs\omega_{i,a,r}^{-1})\otimes L_q(\bs\pi')$ were not highest-$\ell$-weight, \eqref{killfactorsimple} would imply that neither would be the right-hand side of \eqref{killfactorepi}, yielding a contradiction.
\end{proof}

\begin{lem}\label{p:killfactormidp}
	Let $\bs\pi,\bs\pi'\in\mathcal P^+$ be such that $L_q(\bs\pi)\otimes L_q(\bs\pi')$ is highest-$\ell$-weight and let $G$ and $G^{\prime}$ pseudo $q$-factorization graphs over $\bs\pi$ and $\bs\pi'$, respectively. If $\bs\omega_{i,a,r}$ is a vertex in $G$, then  $L_q(\bs\pi\bs\omega_{i,a,r}^{-1})\otimes L_q(\bs\pi')$ is also highest-$\ell$-weight if there exist $s,s'\in\mathbb{Z}_{>0}$ such that, for each other vertex $\bs\omega_{i',aq^m,r'}$ of $G\otimes G'$, we have
	\begin{equation}\label{p:killfactormidpGG'}
		m-d_i(r+s')\notin\mathscr{R}_{i,i'}^{s',r'}, 
	\end{equation}
	\begin{equation}\label{p:killfactormidpG}
		-d_i(r+s)-m\notin\mathscr{R}_{i,i'}^{s,r'}, %\quad\textrm{and}\quad |m+s|\notin\mathscr{R}_{i,i'}^{r+s,r'},
		\quad\textrm{and}\quad  |m+d_i(s-s')|\notin \mathscr R_{i,i'}^{r+s+s',r'}, 		
		\quad\textrm{if}\quad\bs\omega_{i',aq^m,r'}\in\mathcal{V}_G,
	\end{equation}
	as well as
	\begin{equation}\label{p:killfactormidpG'}
		m+d_i(r+s)\notin\mathscr{R}_{i,i'}^{s,r'},\quad\textrm{if}\quad\bs\omega_{i',aq^m,r'}\in\mathcal{V}_{G'}.
	\end{equation}
	Similarly,if $\bs\omega_{i,a,r}$ is a vertex in $G'$, then  $L_q(\bs\pi)\otimes L_q(\bs\pi'\bs\omega_{i,a,r}^{-1})$ is also highest-$\ell$-weight if there exists $s,s'\in\mathbb{Z}_{>0}$ such that, for each other vertex $\bs\omega_{i',aq^m,r'}$ of $G\otimes G'$, we have
	\begin{equation*}%\label{p:killfactormidpGG'}
		-m-d_i(r+s')\notin\mathscr{R}_{i,i'}^{s',r'}, 
	\end{equation*}
	\begin{equation*}%\label{p:killfactormidpG}
		m-d_i(r+s)\notin\mathscr{R}_{i,i'}^{s,r'}, %\quad\textrm{and}\quad |m+s|\notin\mathscr{R}_{i,i'}^{r+s,r'},
		\quad\textrm{and}\quad  |m+d_i(s'-s )|\notin \mathscr R_{i,i'}^{r+s+s',r'}, 		
		\quad\textrm{if}\quad\bs\omega_{i',aq^m,r'}\in\mathcal{V}_{G'},
	\end{equation*}
	as well as
	\begin{equation*}%\label{p:killfactormidpG'}
		d_i(r+s)-m\notin\mathscr{R}_{i,i'}^{s,r'},\quad\textrm{if}\quad\bs\omega_{i',aq^m,r'}\in\mathcal{V}_{G}.
	\end{equation*}
\end{lem}

\begin{proof}
	We write down the details for the first part only as the second is similar. Given $s\in\mathbb Z_{\ge 0}$, consider 
	\begin{equation*}
		\bs\omega_{i,aq_i^{-s},r+s} = \bs\omega_{i,a,r}\bs\omega_{i,aq_i^{-r-s},s}.
	\end{equation*}
	In particular,
	\begin{equation}\label{killfactormidpoly}
		\bs\omega_{i,aq_i^{-r-s},s}\bs\pi= \bs\omega_{i,aq_i^{-s},r+s}\bs\pi\bs\omega_{i,a,r}^{-1}.
	\end{equation}	
	\Cref{l:hlwquot} and the the first assumption in \eqref{p:killfactormidpG} imply $L_q(\bs\pi)\otimes L_q(\bs\omega_{i,aq_i^{-r-s},s})$ is highest-$\ell$-weight while \eqref{p:killfactormidpG'} implies the same holds for $L_q(\bs\omega_{i,aq_i^{-r-s},s})\otimes L_q(\bs\pi')$. Since $L_q(\bs\pi)\otimes L_q(\bs\pi')$  is also highest-$\ell$-weight, \Cref{cyc} implies that so is
	\begin{equation*}
		L_q(\bs{\pi})\otimes L_q(\bs{\omega}_{i,aq_i^{-r-s},s})\otimes L_q(\bs\pi')
	\end{equation*}
	and, moreover,  we have an epimorphism
	\begin{equation*}
		L_q(\bs{\pi})\otimes L_q(\bs{\omega}_{i,aq_i^{-r-s},s})\otimes L_q(\bs\pi') \to L_q(\bs{\pi}\bs{\omega}_{i,aq_i^{-r-s},s})\otimes L_q(\bs\pi').
	\end{equation*}
	Hence, in light of \eqref{killfactormidpoly}, 
	\begin{equation*}
		L_q(\bs\omega_{i,aq_i^{-s},r+s}\bs\pi\bs\omega_{i,a,r}^{-1})\otimes L_q(\bs\pi') \quad\text{is highest-$\ell$-weight.}
	\end{equation*}
	Letting $\bs{\tilde\pi} = \bs\omega_{i,aq_i^{-s},r+s}\bs\pi\bs\omega_{i,a,r}^{-1}$, we have just proved that $L_q(\bs{\tilde\pi})\otimes L_q(\bs\pi')$ is highest-$\ell$-weight and we want to show 
	\begin{equation*}
		L_q(\tilde{\bs\pi}\bs\omega_{i,aq_i^{-s},r+s}^{-1})\otimes L_q(\bs\pi') \quad\text{is highest-$\ell$-weight.}
	\end{equation*}
	This follows from \Cref{p:killfactorp} if the pseudo $q$-factor $\bs\omega_{i,aq_i^{-s},r+s}$ of $\tilde{\bs\pi}$ satisfies the corresponding conditions in \eqref{e:killfactorp}. This is exactly what \eqref{p:killfactormidpGG'} and the second condition in \eqref{p:killfactormidpG} guarantee.	
\end{proof}

\section{Totally Ordered Graphs}\label{ss:toto}
In this section we prove  \Cref{t:toto} and, hence, assume $\lie g$ is of type $A$.

\subsection{Further Vertex Removal Criteria}\label{ss:removeA}
We continue the line of investigation initiated in \Cref{ss:remove} and obtain further criteria with the same spirit, but specific for type $A$.

\begin{prop}\label{p:killfactorA}
	Let $\bs\pi,\bs\pi'\in\mathcal P^+$ be such that $L_q(\bs\pi)\otimes L_q(\bs\pi')$ is highest-$\ell$-weight and let $G$ be a pseudo $q$-factorization graph over $\bs\pi$. If $\bs\omega_{i,a,r}$ is a vertex in $G$, then  $L_q(\bs\pi\bs\omega_{i,a,r}^{-1})\otimes L_q(\bs\pi')$ is also highest-$\ell$-weight if, for each other vertex $\bs\omega_{i',aq^m,r'}$ of $G$, we have
	\begin{equation}\label{killfactoropt}
		-m>\max\mathscr R_{i,i'}^{r,r'} \qquad\text{or}\qquad r + d(i,i')\ge r'-m.
	\end{equation}
	Similarly, if $G'$ is a pseudo $q$-factorization graph over $\bs\pi'$ and $\bs\omega_{i,a,r}$ is a vertex in $G'$, then  $L_q(\bs\pi)\otimes L_q(\bs\pi'\bs\omega_{i,a,r}^{-1})$ is also highest-$\ell$-weight if, for each other vertex $\bs\omega_{i',aq^m,r'}$ of $G'$,  we have
	\begin{equation}
		m>\max\mathscr R_{i,i'}^{r,r'} \qquad\text{or}\qquad r+ d(i,i')\ge r'+m.
	\end{equation}
\end{prop}

\begin{proof}
	We write down the details for the first statement only, as the second is similar. We check that \eqref{killfactoropt} implies there exists $s$ satisfying \eqref{e:killfactorp} for any pseudo $q$-factorization graph $G'$ over $\bs\pi'$ and the proposition then follows from \Cref{p:killfactorp}. Indeed, letting
	\begin{equation*}
		M = \max\{m\in\mathbb Z:\exists\ i'\in I, r'\in\mathbb Z_{>0} \text{ s.t. } \bs\omega_{i',aq^m,r'}\in\mathcal V_{G\otimes G'}\},
	\end{equation*}
	the first condition in \eqref{e:killfactorp} is satisfied provided $s\ge M-r$ (this does not depend on $\lie g$ being of type $A$).

	Let $\bs\omega_{i',aq^m,r'}$ be a vertex of $G$ distinct from $\bs\omega_{i,a,r}$. By choosing $s$ sufficiently large, we can also assume $s>m$ and $r+s>r'$.  In particular, $|s-m|=s-m$. If the second condition in \eqref{e:killfactorp} were not satisfied, we would have
	\begin{equation*}
		s-m =  s+r+r'+d(i,i')- 2p \quad\text{for some}\quad -d([i,i'],\partial I)\le p< r'=\min\{r',s+r\},
	\end{equation*}
	and, hence,
	\begin{equation*}
		-m =r+r'+d(i,i')- 2p.
	\end{equation*}
	The condition $-d([i,i'],\partial I)\le p$ contradicts the first option in \eqref{killfactoropt} while the condition $p< r'$ contradicts the second.
\end{proof}

\begin{prop}\label{p:killfactormid}
	Let $\bs\pi,\bs\pi'\in\mathcal P^+$ be such that $L_q(\bs\pi)\otimes L_q(\bs\pi')$ is highest-$\ell$-weight and let $G$ and $G'$ be pseudo $q$-factorization graphs over $\bs\pi$ and $\bs\pi'$, respectively. Suppose there exists a vertex $\bs\omega_{i,a,r}$ of $G$ such that for any vertex $\bs\omega_{i',aq^m,r'}$ of $G'$, we have
	\begin{equation}\label{e:killfactormid}
		m = r+r'+d(i,i')-2p \quad\text{with either}\quad p<r-d([i,i'],\partial I) \quad\text{or}\quad p\ge r+r' .  
	\end{equation} 
	Then, $L_q(\bs\pi\bs\omega_{i,a,r}^{-1})\otimes L_q(\bs\pi')$ is also highest-$\ell$-weight. 
	Similarly, if  there exists a vertex $\bs\omega_{i,a,r}$ of $G'$ such that, for any vertex $\bs\omega_{i',aq^m,r'}$ of $G$, we have
	\begin{equation*}
		m = r+r'+d(i,i')-2p \quad\text{with either}\quad p<r-d([i,i'],\partial I) \quad\text{or}\quad p\ge r+r',
	\end{equation*}
	then $L_q(\bs\pi)\otimes L_q(\bs\pi'\bs\omega_{i,a,r}^{-1})$ is also highest-$\ell$-weight.
\end{prop}

\begin{proof}
	We write down the details for the first part only as the second is similar. We check there exists $s$ such that \eqref{p:killfactormidpG} and \eqref{p:killfactormidpG'} are satisfied, so the proposition follows from \Cref{p:killfactormidp}. It suffices to choose $s$ so that
	\begin{align}\label{eq:boundminus}
		s&\geq\max\{-r-m: \exists\ i'\in I, r'\in\mathbb Z_{>0} \text{ s.t. } \bs\omega_{i',aq^m,r'}\in\mathcal V_G\},\\ \label{eq:boundminus'}
		s&\geq\max\{r': \exists\ i'\in I, m\in\mathbb Z \text{ s.t. } \bs\omega_{i',aq^m,r'}\in\mathcal V_{G'}\}, \quad\text{and}\\ \label{eq:minusf}
		2s&\geq\max\{-r-m+r': \exists\ i'\in I \text{ s.t. } \bs\omega_{i',aq^m,r'}\in\mathcal V_G\}.
	\end{align}
	Indeed,  \eqref{eq:boundminus} implies that, if $\bs\omega_{i',aq^m,r'}\in\mathcal{V}_G$, then $-r-s-m\le 0$ and, hence, the first condition in \eqref{p:killfactormidpG} is satisfied. In its turn, \eqref{eq:boundminus'} implies \eqref{p:killfactormidpG'}. Indeed, if \eqref{p:killfactormidpG'} did not hold, there would exist $\bs\omega_{i',aq^m,r'}\in\mathcal{V}_{G'}$ such that
	\begin{equation*}
		m+r+s = r'+s+d(i,i')-2p' \quad\text{with}\quad -d([i,i'],\partial I)\le p'<\min\{r',s\}=r'.
	\end{equation*}
	But this would imply 
	\begin{equation*}
		m = r+r'+d(i,i') - 2p \quad\text{with}\quad p=r+p',
	\end{equation*}
	yielding a contradiction with \eqref{e:killfactormid}.
	
	Let $\tilde{\bs\pi}$ be as in the proof of \Cref{p:killfactormidp} with this choice of $s$. As seen in that proof, \eqref{p:killfactormidpGG'} and the second condition in \eqref{p:killfactormidpG} are equivalent to requesting that the pseudo $q$-factor $\bs\omega_{i,aq^{-s},r+s}$ of $\tilde{\bs\pi}$ satisfies the corresponding conditions in \eqref{e:killfactorp}. In light of \Cref{p:killfactorA}, this follows if
	\begin{equation*}
		(r+s)+d(i,i')\ge r'-(m+s)
	\end{equation*}	
	for every pseudo $q$-factor $\bs\omega_{i',aq^m,r'}$ of $\bs\pi$ different than $\bs\omega_{i,a,r}$. This is guaranteed by \eqref{eq:minusf}.
\end{proof}

\subsection{Some Combinatorics}\label{ss:arith} 
In this section, we deduce a few technical lemmas concerned with arithmetic relations among the elements of $\mathscr R_{i,j}^{r,s}$. In particular, they are useful for detecting whether a pseudo $q$-factorization graph is a tournament.  

\begin{lem}\label{l:simptpdual}
	If $i,j\in I, r,s\in\mathbb Z_{>0}, m\in\mathscr R_{i,j}^{r,s}\setminus \mathscr R_{i,j,[i,j]}^{r,s}$, and $a\in\mathbb F^\times$, then 
	\begin{equation*}
		L_q(\bs\omega_{i,aq^m,r})\otimes L_q(\bs\omega_{j,a,s})^*\quad\text{is simple}.
	\end{equation*}
\end{lem}

\begin{proof}
	The assumptions imply $m=r+s+d(i,j)-2p$ for some $-d([i,j],\partial I)\le p<0$, while the claim follows if we show that $		m+h^\vee \notin\mathscr R_{i,j^*}^{r,s}$. Since
	\begin{equation*}
		0<m+h^\vee = r+s+d(i,j^*) - 2p' \quad\text{with}\quad p' = p + \frac{d(i,j^*)-d(i,j)-h^\vee}{2},
	\end{equation*}
	it suffices to show $p'< -d([i,j^*],\partial I)$.
	
	Without loss of generality, assume $I$ has been identified with $\{1,\dots,n\}$ so that $i\le j$ and recall that $j^*=n+1-j$.  Suppose first that
	$d([i,j],\partial I)=d(i,\partial I)$. It follows that $d([i,j^*],\partial I)=d(i,\partial I)$ and either $d(i,\partial I)=i-1$ or $i=j$ and $d(i,\partial I)=n-i$. In the former case, we have $j^*\ge i$ and
	\begin{equation*}
		d(i,j^*)-d(i,j)-h^\vee = (n+1-j-i)-(j-i)-(n+1) = -2j. 
	\end{equation*}
	Therefore, since $p\le -1$, we see that $$p'\le -1 -j\le -1-i = -1-(d([i,j^*],\partial I)+1) = -d([i,j^*],\partial I)-2,$$
	thus completing the proof in this case.
	In the latter case, $j^*\le i=j$ and we have
	\begin{equation*}
		d(i,j^*)-d(i,j)-h^\vee = (i-(n+1-i))-(n+1) = -2(n+1-i) = -2d([i,j^*],\partial I)-2, 
	\end{equation*}
	which also completes the proof.
	
	It remains to consider the case $d([i,j],\partial I)=d(j,\partial I)=n-j$, which implies $j^*\le i$ and $d([i,j^*],\partial I)=d(j^*,\partial I)=n-j$. Hence,
	\begin{equation*}
		d(i,j^*)-d(i,j)-h^\vee = (i-(n+1-j))-(j-i)-(n+1) = -2(n+1-i),
	\end{equation*}
	and we get,
	\begin{equation*}
		p'\le -1 - (d([i,j^*],\partial I)+1+(j-i))< - d([i,j^*],\partial I),
	\end{equation*}
	as desired.
\end{proof}

\begin{lem}\label{l:lineps}
	Let $N\in\mathbb Z_{> 0}$ and $(m_k,r_k,i_k)\in \mathbb Z_{\ge 0}\times \mathbb Z_{> 0}\times I, 1\le k\le N$. Suppose 
	\begin{equation}\label{e:lineps}
		|m_k-m_{k-1}|\in\mathscr{R}^{r_{k-1},r_k}_{i_{k-1},i_k} \qquad\text{for all}\qquad 1< k\leq N.
	\end{equation}
	\begin{enumerate}[(a)]
		\item For all $1\le k,l\le N$, there exists $p_{l,k}\in\mathbb Z$ such that $m_l-m_k=r_l+r_k+d(i_l,i_k)-2p_{l,k}$. 
		\item If $m_k>m_{k-1}$ for all $1< k\leq N$, then $p_{N,1}<\operatorname{min}\{r_1,r_N\}$, and 
		\begin{equation*}
			p_{N,1}<p_{l,k}<\operatorname{min}\{r_k,r_l\},\quad\textrm{for all}\quad 1\leq k<l\leq N,\quad\textrm{with}\quad (k,l)\neq(1,N).
		\end{equation*}
		Similarly, if  $m_k<m_{k-1}$ for all $1< k\leq N$, then $p_{1,N}<\operatorname{min}\{r_1,r_N\}$, and 
		\begin{equation*}
			p_{1,N}<p_{k,l}<\operatorname{min}\{r_k,r_l\},\quad\textrm{for all}\quad 1\leq k<l\leq N,\quad\textrm{with}\quad (k,l)\neq(1,N).
		\end{equation*}
	\end{enumerate}
\end{lem}

\begin{proof}
	The equality $m_k-m_l=r_k+r_l+d(i_k,i_l)-2p_{k,l}$ clearly defines $p_{k,l}\in\mathbb Q$ and, moreover, one can easily check that
	\begin{equation}\label{e:strangeprel}
		p_{k,l}+p_{j,k} = p_{j,l} + r_k +  d^{i_k}_{i_j,i_l} \quad\text{for all}\quad 1\le j,k,l\le N, 
	\end{equation}
	and
	\begin{equation}\label{e:pklplk}
		p_{k,l} + p_{l,k} = r_k + r_l +d(i_k,i_l) \quad\text{for all}\quad 1\le k,l\le N.
	\end{equation}
	We will use these to show $p_{k,l}\in\mathbb Z$ by induction on $|k-l|\ge 1$ (note we also have $p_{k,k}=r_k$). If $|k-l|=1$, \eqref{e:lineps} implies either $p_{k,l}$ or $p_{l,k}$ is an integer. Then, \eqref{e:pklplk} implies the same is true for the other one. If $|k-l|>1$, the inductive step easily follows from \eqref{e:strangeprel} by choosing $j$ in between $k$ and $l$. 
	
	We prove (b) in the case $m_k>m_{k-1}$ by induction on $N> 1$ (the other case is similar). For $N=2$, we have $p_{N,1}=p_{2,1}$ and, hence, the first claim follows from \eqref{e:lineps}, while the second claim is vacuous. Note  the first claim is a consequence of the second for $N>2$, in which case,  the inductive hypothesis implies
	\begin{equation*}
		p_{N-1,1}<p_{l,k}<\operatorname{min}\{r_k,r_l\},\quad\textrm{for all}\quad 1\leq k<l\leq N-1,\quad\textrm{with}\quad (k,l)\neq(1,N-1).
	\end{equation*}
	as well as $p_{N-1,1}<\operatorname{min}\{r_1,r_{N-1}\}$. By \eqref{e:strangeprel}, we have
	\begin{equation}\label{eq:hypind}
		p_{N,1}=p_{N,N-1}+p_{N-1,1}-r_{N-1}-d_{i_1,i_N}^{i_{N-1}} \quad\text{and}\quad p_{N,1}=p_{N,2}+p_{2,1}-r_{2}-d_{i_1,i_N}^{i_{1}}.
	\end{equation}
	Moreover, 
	\begin{equation}\label{eq:hypind'}
		p_{N,N-1}=p_N<\operatorname{min}\{r_N,r_{N-1}\}\leq r_{N-1} \quad\text{and}\quad p_{2,1}<\operatorname{min}\{r_2,r_1\}\leq r_2.
	\end{equation}
	Therefore,
	\begin{equation*}
		p_{N,1}<p_{N-1,1}-d_{i_1,i_N}^{i_{N-1}}\leq p_{N-1,1}<p_{l,k}<\operatorname{min}\{r_k,r_l\},
	\end{equation*} 
	for all $1\leq k<l\leq N-1$ with $(k,l)\neq(1,N-1)$. 	
	Thus, it remains to show that
	\begin{equation*}
		p_{N,1}<p_{N,k}<\operatorname{min}\{r_k,r_N\} \quad\text{for all}\quad 1< k<N.
	\end{equation*}
	Applying the inductive hypothesis to the sequence  $(m_k,r_k,i_k), 1< k\le N$, we have
	\begin{equation*}
		p_{N,1}<p_{l,k}<\operatorname{min}\{r_k,r_l\} \quad\text{for all}\quad 1< k<l<N, \ (k,l)\ne (2,N).
	\end{equation*}
	The second claims in \eqref{eq:hypind} and \eqref{eq:hypind'} imply $p_{N,1}<p_{N,2}$, thus completing the proof.
\end{proof}

%\begin{lem}\label{l:positiveppath}
%	Assume $m_k>m_{k-1}$ for all $1< k\leq N$ in \Cref{l:lineps}. 
%	If $m_N-m_1\in\mathscr R^{r_1,r_N}_{i_1,i_N,[i_1,i_N]}$, then $m_l-m_k\in\mathscr R^{r_k,r_l}_{i_k,i_l,[i_k,i_l]}$	for all $1\leq k<l\leq N$. In particular, if $J\subseteq I$ is connected,
%	\begin{equation*}
%		\bs\pi = \prod_{k:i_k\in J} \bs\omega_{i_k,aq^{m_k},r_k},
%	\end{equation*}
%	and $G$ is the pseudo $q$-factorization graph for $U_q(\tlie g)_J$ associated to this pseudo $q$-factorization of $\bs\pi$, then $G$ is a tournament. 
%\end{lem}
%
%\begin{proof}
%	The second claim  is immediate from the first. By assumption, the number $p_{N,1}$ from \Cref{l:lineps}(a) satisfies $0\le p_{N,1}<\min\{r_1,r_N\}$ and we want to show $0\le p_{l,k}<\min\{r_k,r_l\}$, which follows from \Cref{l:lineps}(b).	
%\end{proof}

\begin{lem}\label{l:positiveppath}
	Assume $m_k>m_{k-1}$ for all $1< k\leq N$ in \Cref{l:lineps}. 
	\begin{enumerate}[(a)]
		\item If $p_{N,1}\ge -d([i_k,i_l],\partial I)-1$ for some  $1\leq k<l\leq N, (l,k)\ne (1,N)$, then $m_l-m_k\in\mathscr R^{r_k,r_l}_{i_k,i_l}$. In particular, this is the case if $m_N-m_1\in\mathscr R^{r_1,r_N}_{i_1,i_N}$ and $d([i_k,i_l],\partial I)\ge d([i_1,i_N],\partial I)$.
		\item If $m_N-m_1\in\mathscr R^{r_1,r_N}_{i_1,i_N,[i_1,i_N]}$, then $m_l-m_k\in\mathscr R^{r_k,r_l}_{i_k,i_l,[i_k,i_l]}$	for all $1\leq k<l\leq N$.
	\end{enumerate}
\end{lem}

\begin{proof}
	The initial assumption in (a), together with \Cref{l:lineps}(b), implies $-d([i_k,i_l],\partial I)-1\le p_{N,1}<p_{l,k}$. A second application of \Cref{l:lineps}(b) then implies $-d([i_k,i_l],\partial I)\le p_{l,k}<\operatorname{min}\{r_k,r_l\}$, thus proving the first claim in (a). The assumptions in the second part of (a) imply $p_{N,1}\ge - d([i_1,i_N],\partial I)\ge -d([i_k,i_l],\partial I)$, showing the second part follows from the first.
	
	The assumption in (b) implies $0\le p_{N,1}<\min\{r_1,r_N\}$ and we want to show $0\le p_{l,k}<\min\{r_k,r_l\}$, which follows from \Cref{l:lineps}(b).	
\end{proof}

 Assume, for instance, that the assumption in (b) of the last lemma holds. Then, if $J\subseteq I$ is connected and 
\begin{equation*}
	\bs\pi = \prod_{k:i_k\in J} \bs\omega_{i_k,aq^{m_k},r_k},
\end{equation*}
the pseudo $q$-factorization graph $G$ for $U_q(\tlie g)_J$ associated to this pseudo $q$-factorization of $\bs\pi$ is a tournament.

\subsection{The Main Lemma}\label{ss:totol}
Fix $\bs{\pi}\in\mathcal{P}^+$ such that its $q$-factorization graph $G=G(\bs{\pi})=(\mathcal V, \mathcal A)$ is totally ordered
%We shall use the notation fixed in the  \hyperlink{pr:totopartial}{first paragraph of the proof of} \Cref{p:totopartial}.
and let $N=\#\mathcal V$. Let $\bs\omega_{i,a,r}$ be the sink and let $m_l,r_l\in\mathbb Z_{\ge 0}, 1\le l\le N$, be such that
$0=m_1<m_2<\cdots<m_N$ and 
\begin{equation*}
	\mathcal V = \{\bs\omega_{i_l,aq^{m_l},r_l}: 1\le l\le N\}.
\end{equation*}
To shorten notation, set $\bs\omega^{(l)}=\bs\omega_{i_l,aq^{m_l},r_l}, 1\le l\le N$. Note 
\begin{equation*}
	\mathcal{A}\subseteq \{(\bs\omega^{(l)},\bs\omega^{(k)}): 1\leq k<l\leq N\},   
\end{equation*}
and 
\begin{equation*}
	(\bs\omega^{(l)},\bs\omega^{(k)}) \in\mathcal A \quad\Rightarrow\quad m_l-m_k\in\mathscr{R}_{i_k,i_l}^{r_k,r_l}.
\end{equation*}
Since $\lie g$ is of type $A$, the latter is equivalent to
\begin{equation*}
	m_l-m_k=r_k+r_l+d(i_k,i_l)-2p_{l,k},\quad\textrm{for some}\quad -d([i_k,i_l],\partial I)\leq p_{l,k}<\operatorname{min}\{r_k,r_l\}.
\end{equation*}
\Cref{l:lineps} %{l:positiveppath} 
 implies such an expression exists for $m_l-m_k$ for all $1\le k<l\le N$ for some  $p_{l,k}\in\mathbb Z$ and, moreover, 
\begin{equation}\label{e:plkub}
	p_{l,k}<\operatorname{min}\{r_k,r_l\} \quad\text{for all}\quad 1\le k<l\le N.
\end{equation}	
Furthermore,  \Cref{t:krredsets} and \eqref{e:areqfact} imply
\begin{equation}\label{e:p>0intne}
	p_{l,k}\geq 0 \quad\Leftrightarrow\quad m_l-m_k \in\mathscr R_{i_k,i_l,[i_k,i_l]}^{r_k,r_l} \quad\Rightarrow\quad i_l\neq i_k.
\end{equation}

Let us make a brief interlude and use the setup we have just fixed to give the:

\begin{proof}[Proof of \Cref{p:totopartial}]%\hypertarget{pr:totopartial}{}
%	We begin by fixing some notation and making some general comments which we also be used in \Cref{ss:totol}.
%	Let $N=\#\mathcal V$, let $\bs\omega_{i,a,r}$ be the sink of $G=(\mathcal V,\mathcal A)$, and let $m_l,r_l\in\mathbb Z_{\ge 0}, 1\le l\le N$, be such that
%	$0=m_1<m_2<\cdots<m_N$ and 
%	\begin{equation*}
%		\mathcal V = \{\bs\omega_{i_l,aq^{m_l},r_l}: 1\le l\le N\}.
%	\end{equation*}
%	To shorten notation, set $\bs\omega^{(l)}=\bs\omega_{i_l,aq^{m_l},r_l}, 1\le l\le N$. Note 
%	\begin{equation*}
%		\mathcal{A}\subseteq \{(\bs\omega^{(l)},\bs\omega^{(k)}): 1\leq k<l\leq N\},   
%	\end{equation*}
%	and 
%	\begin{equation*}
%		(\bs\omega^{(l)},\bs\omega^{(k)}) \in\mathcal A \quad\Rightarrow\quad m_l-m_k\in\mathscr{R}_{i_k,i_l}^{r_k,r_l}.
%	\end{equation*}
%	Since $\lie g$ is of type $A$, the latter is equivalent to
%	\begin{equation*}
%		m_l-m_k=r_k+r_l+d(i_k,i_l)-2p_{l,k},\quad\textrm{for some}\quad -d([i_k,i_l],\partial I)\leq p_{l,k}<\operatorname{min}\{r_k,r_l\}.
%	\end{equation*}
%	\Cref{l:lineps} %{l:positiveppath} 
%	implies such an expression exists for $m_l-m_k$ for all $1\le k<l\le N$ for some  $p_{l,k}\in\mathbb Z$ and, moreover, 
%	\begin{equation}\label{e:plkub}
%		p_{l,k}<\operatorname{min}\{r_k,r_l\} \quad\text{for all}\quad 1\le k<l\le N.
%	\end{equation}	
%	Furthermore,  \Cref{t:krredsets} and \eqref{e:areqfact} imply
%	\begin{equation}\label{e:p>0intne}
%		p_{l,k}\geq 0 \quad\Leftrightarrow\quad m_l-m_k \in\mathscr R_{i_k,i_l,[i_k,i_l]}^{r_k,r_l} \quad\Rightarrow\quad i_l\neq i_k.
%	\end{equation}
	In light of \eqref{e:p>0intne}, 
	the assumption  $c(\mathcal V)\subseteq \partial I$  implies 
	\begin{equation}\label{e:altcol}
		(\bs\omega^{(l)},\bs\omega^{(k)}) \in\mathcal A \quad\text{only if}\quad i_k\ne i_l.
	\end{equation}
	Note the claim about the vertices being alternately colored is immediate from this. 	
	Since a totally ordered tree is a line, if $G$ were not a line, it would contain a cycle. In that case, let $v$ be the maximal element of $\mathcal V$ which is part of a cycle. Suppose $a = (v,w)$ is the first arrow of this cycle and $a'=(v,w')$ is the last:
	\begin{equation*}
		\begin{tikzcd}
			\cdots w & \arrow[swap,l,"a"] v \arrow[r,"a'"] & w' \cdots
		\end{tikzcd} 
	\end{equation*}
	Set $e=\pi(a)$ and $e'=\pi(a')$. Since $G$ is totally ordered, we must have either $w\prec w'$ or $w\succ w'$. Without loss of generality, we assume it is the latter. This means there exists a (simple) monotonic path $\rho\in\mathscr P_{w,w'}$  and, moreover, $\rho*e\mathscr\in\mathscr P_{v,w'}$ is a monotonic path. Furthermore, $e'*\rho*e$ is a cycle based on $v$ and, by construction, the vertices in this cycle satisfy all the assumptions in \Cref{l:positiveppath}(b). As commented after that lemma, this implies the subgraph determined by this subset of vertices is a tournament which, by \eqref{e:altcol}, have all of its vertices differently colored. This yields a contradiction since $\#\partial I=2$ and there are no cycles with less than three vertices.	
\end{proof}

The next lemma is the heart of the proof of \Cref{t:toto}.

\begin{lem}\label{l:totordopparrow}
	Let $\bs{\pi}',\bs{\pi}''\in\mathcal{P}^+$ have dissociate $q$-factorizations and assume $\bs\pi=\bs\pi'\bs\pi''$. Let also $G'=G(\bs{\pi}')$ and $G''=G(\bs{\pi}'')$. Assume 
	\begin{equation}\label{e:pmcomp}
		L_q(\bs{\pi}')\otimes L_q(\bs{\pi}'') \quad\text{is highest-$\ell$-weight}
	\end{equation}
	and that  $1\leq j'< j''\leq N$ are such that $\bs{\omega}^{(j')}\in\mathcal{V}_{G'}$, $\bs{\omega}^{(j'')}\in\mathcal{V}_{G''}$, and
	\begin{equation}\label{eq:newsubd}
		m_{j''}-m_{j'}\in{\mathscr{R}_{i_{j'},i_{j''},J}^{r_{j'},r_{j''}}},\quad\textrm{where}\quad J=[i_{j'},i_{j''}].
	\end{equation}
	Then, there exist $1\leq k''< k'\leq N$ such that $K:=[i_{k'},i_{k''}]\subsetneqq J$,  $\bs{\omega}^{(k')}\in\mathcal{V}_{G'}$, $\bs{\omega}^{(k'')}\in\mathcal{V}_{G''}$, and $m_{k'}-m_{k''}\in{\mathscr{R}_{i_{k'},i_{k''},K}^{r_{k'},r_{k''}}}$.
\end{lem}

\begin{proof}
	Assume, by contradiction, that there does not exist such pair $(k',k'')$ and consider: 
	\begin{equation*}
		\mathcal{I}_{G'}^+=\{1\leq l\leq N:\bs{\omega}^{(l)}\in\mathcal{N}^+_{G'}(\bs{\omega}^{(j')})\}=\{1\leq l\leq N:\bs{\omega}^{(l)}\in\mathcal{V}_{G'},\; l\geq j'\},
	\end{equation*}
	\begin{equation*}
		\mathcal{I}_{G''}^-=\{1\leq l\leq N:\bs{\omega}^{(l)}\in\mathcal{N}^-_{G''}(\bs{\omega}^{(j'')})\}=\{1\leq l\leq N:\bs{\omega}^{(l)}\in\mathcal{V}_{G''},\; l\leq j''\}.
	\end{equation*}
	\Cref{c:killpi-pi'+}, together with \eqref{e:pmcomp}, implies 
	\begin{equation}\label{e:Mhlw}
		\begin{aligned}
			L_q(\bs\pi_+)\otimes L_q(\bs\pi_-) \quad\text{is highest-$\ell$-weight, where}\\
			\bs\pi_+:=\prod_{l\in\mathcal{I}^+_{G'}} \bs{\omega}^{(l)}\quad\textrm{and}\quad\bs\pi_-:=\prod\limits_{l\in\mathcal{I}^-_{G''}}
			\bs{\omega}^{(l)}.	
		\end{aligned}
	\end{equation}
	If $\mathcal{I}^+_{G'}=\{j'\}$ and $\mathcal{I}^-_{G''}=\{j''\}$, then $\bs\pi_+=\bs{\omega}^{(j')}, \bs\pi_-=\bs{\omega}^{(j'')}$, and \eqref{e:Mhlw}  contradicts \eqref{eq:newsubd}. Thus, henceforth assume that 
	\begin{equation*}
		\text{either}\quad \#\mathcal{I}^+_{G'}>1 \quad\text{or}\quad \#\mathcal{I}^-_{G''}>1.
	\end{equation*}
	
	Set
	\begin{equation*}
		\mathcal{I}_{G'}^{++}=\{l\in\mathcal{I}_{G'}^+:l>j'', p_{l,j''}<0\}, \qquad\mathcal{I}_{G''}^{--}=\{l\in\mathcal{I}_{G''}^-: l<j', p_{j',l}<0\},
	\end{equation*}
	\begin{equation*}
		\bs\pi_{++}=\prod\limits_{l\in\mathcal{I}_{G''}^{++}}\bs{\omega}^{(l)}\quad\textrm{and}\quad\bs\pi_{--}=\prod\limits_{l\in\mathcal{I}_{G''}^{--}}\bs{\omega}^{(l)}.
	\end{equation*}
	\Cref{l:lineps} %{l:positiveppath} 
	implies
	\begin{equation}\label{eq:negH}
		p_{l',l''}<p_{l',j''}<0,\quad\textrm{for all}\quad l'\in\mathcal{I}_{G'}^{++},\;l''\in\mathcal{I}_{G''}^-\setminus\{j''\}
	\end{equation}
	and
	\begin{equation}\label{eq:negK}
		p_{l',l''}<p_{j',l''}<0,\quad\textrm{for all}\quad l''\in\mathcal{I}_{G''}^{--},\;l'\in\mathcal{I}_{G'}^+\setminus\{j'\}.
	\end{equation}
	Moreover, \Cref{l:lineps} %{l:positiveppath} 
	also implies
	\begin{equation}\label{eq:Hpm}
		l'\in\mathcal{I}_{G'}^+\setminus\mathcal{I}_{G'}^{++}\quad\Rightarrow\quad l'<l\quad\textrm{for all}\quad l\in\mathcal{I}_{G'}^{++}.
	\end{equation}
	Indeed, if it could be $l<l'$ for some $l\in\mathcal{I}_{G'}^{++}$, it would follow from \Cref{l:lineps} %{l:positiveppath} 
	 that
	\begin{equation*}
		p_{l',j''} < p_{l,j''} <0
	\end{equation*}
	which contradicts the assumption $l'\notin\mathcal{I}_{G'}^{++}$. Similarly,
	\begin{equation}\label{eq:Kpm}
		l''\in\mathcal{I}_{G''}^-\setminus\mathcal{I}_{G''}^{--}\quad\Rightarrow\quad l''>l\quad\textrm{for all}\quad l\in\mathcal{I}_G''^{--}.
	\end{equation}

	Note that  \eqref{eq:negH},  together with \Cref{l:simptpdual}, implies 
	\begin{equation}\label{e:doubred1}
		L_q(\bs{\omega}^{(l')})\otimes L_q(\bs{\omega}^{(l'')})^*\quad\text{is simple for all}\quad l'\in\mathcal{I}_{G'}^{++},\;l''\in\mathcal{I}_{G''}^-
	\end{equation}
	and, similarly, \eqref{eq:negK} implies 
	\begin{equation}\label{e:doubred2}
		L_q(\bs{\omega}^{(l')})\otimes L_q(\bs{\omega}^{(l'')})^* \quad\text{is simple for all}\quad l''\in\mathcal{I}_{G''}^{--},\;l'\in\mathcal{I}_{G'}^+.
	\end{equation}
	In their turn, \eqref{eq:Hpm} and \eqref{eq:Kpm} imply
	\begin{equation}\label{e:doubred3}
		L_q(\bs{\omega}^{(l)})\otimes L_q(\bs{\omega}^{(l')})\quad\textrm{is highest-$\ell$-weight for all}\quad l'\in\mathcal{I}_{G'}^+\setminus\mathcal{I}_{G'}^{++},\;l\in\mathcal{I}_{G'}^{++}
	\end{equation}
	and
	\begin{equation}\label{e:doubred4}
		L_q(\bs{\omega}^{(l')})\otimes L_q(\bs{\omega}^{(l)}) \quad\textrm{is highest-$\ell$-weight for all}\quad l'\in\mathcal{I}_{G'' }^-\setminus\mathcal{I}_{G''}^{--},\;l\in\mathcal{I}_{G''}^{--}.
	\end{equation}
	We will check that these facts, together with \eqref{e:Mhlw},  \Cref{l:hlwquot},  and \Cref{p:killdualhlw}, imply 
	\begin{equation}\label{e:doubred}
		\begin{aligned}
			M=L_q(\bs{\varpi}')\otimes L_q(\bs{\varpi}'') \quad\text{is highest-$\ell$-weight, where}\\
			\bs{\varpi}'=\bs\pi_+(\bs\pi_{++})^{-1}\quad\textrm{and}\quad\bs{\varpi}''=\bs\pi_-(\bs\pi_{--})^{-1}.
		\end{aligned}
	\end{equation}
	Moreover, \Cref{c:sJs} implies that $M_J=L_q(\bs{\varpi}'_J)\otimes L_q(\bs{\varpi}''_J)$ is also highest-$\ell$-weight. Using the initial assumption of the proof, we will see that this contradicts \eqref{eq:newsubd}, thus completing the proof.
	
	To check \eqref{e:doubred}, we first use \Cref{p:killdualhlw} with $i=1$, $\bs\lambda=\bs\pi_{++},\bs\mu = \bs\varpi'$, and $\bs\nu = \bs\pi_-$. In the terminology of  \Cref{p:killdualhlw}, \eqref{e:Mhlw} means $T_1$ is highest-$\ell$-weight, \eqref{e:doubred1} and \Cref{l:hlwquot} imply $V$ is simple, while  \eqref{e:doubred3} and \Cref{l:hlwquot} imply $U_1$ is highest-$\ell$-weight. Hence, $W_1 = L_q(\bs\varpi')\otimes L_q(\bs\pi_-)$ is  highest-$\ell$-weight. A second application of \Cref{p:killdualhlw} with $i=2$, $\bs\lambda=\bs\varpi',\bs\mu = \bs\varpi''$, and $\bs\nu = \bs\pi_{--}$, together with \eqref{e:doubred2}, \eqref{e:doubred4}, \Cref{l:hlwquot}, and \Cref{l:hlwquot} gives \eqref{e:doubred}.

	Consider the following sets:
	\begin{equation*}
		\mathcal{J}_{G'}=(\mathcal{I}^+_{G'}\setminus\mathcal{I}^{++}_{G'})\cap\{1\leq l\leq N: i_l\in  J\} \quad\textrm{and}\quad \mathcal{J}_{G''}=(\mathcal{I}^-_{G''}\setminus\mathcal{I}^{--}_{G''})\cap\{1\leq l\leq N: i_l\in J\}.
	\end{equation*}
	Note 
	\begin{equation}
		\bs\varpi'_J = \prod_{l\in\mathcal{J}_{G'}}\bs\omega^{(l)}_J \quad\text{and}\quad \bs\varpi''_J = \prod_{l\in\mathcal{J}_{G''}}\bs\omega^{(l)}_J. 
	\end{equation}
	If $\mathcal{J}_{G'}=\{j'\}$ and $\mathcal{J}_{G'}=\{j''\}$, then $\bs{\varpi}'_J=\bs{\omega}^{(j')}_J$, $\bs{\varpi}''_J=\bs{\omega}^{(j'')}_J$ and $M_J=L_q(\bs{\omega}^{(j')}_J)\otimes L_q(\bs{\omega}^{(j'')}_J)$, yielding a contradiction between \eqref{e:doubred} and \eqref{eq:newsubd}. Thus, we must have
	\begin{equation*}
		\text{either} \quad \#\mathcal{J}_{G'}>1 \quad\text{or}\quad \#\mathcal{J}_{G''}>1.
	\end{equation*}
	Consider also 
	\begin{equation*}
		\mathcal{J}^+_{G'}:=\{l\in\mathcal J_{G'}: l>j''\},\qquad\mathcal{J}^-_{G'}:=\{l\in\mathcal J_{G'}: l<j''\},
	\end{equation*}
	\begin{equation*}
		\mathcal{J}^+_{G''}:=\{l\in\mathcal J_{G''}: l>j'\},\qquad\mathcal{J}^-_{G''}:=\{l\in\mathcal J_{G''}: l<j'\}.
	\end{equation*}
	Obviously, $j'\in\mathcal{J}^-_{G'}, j''\in\mathcal{J}^+_{G''}$, $\mathcal{J}_{G'}$ is the disjoint union of $\mathcal{J}^\pm_{G'}$, and similarly for $\mathcal{J}_{G''}$.  We claim 
	\begin{equation*}
		\#\mathcal{J}^+_{G'}\leq 1\quad\textrm{and}\quad\#\mathcal{J}^-_{G''}\leq 1.
	\end{equation*}
	Indeed, by definition of let $\mathcal{J}^+_{G'}$, we have
	\begin{equation}\label{e:J+G'}
		l\in\mathcal{J}^+_{G'} \quad\Rightarrow\quad i_l\in J,\quad l>j'',\quad\textrm{and}\quad p_{l,j''}\geq 0.
	\end{equation}
	In particular, together with \eqref{e:plkub} and \eqref{e:p>0intne}, this implies  
	\begin{equation*}
		m_l-m_{j''}\in{\mathscr{R}_{i_{j''},i_l}^{r_{j''},r_l}}_{[i_{j''},i_l]}, \quad i_{j''}\ne i_l, \quad\text{and}\quad [i_{j''},i_l]\subseteq J.
	\end{equation*}
	If it were $[i_{j''},i_l]\subsetneqq J$, then $k'=l$ and $k''=j''$ would be a pair of indices satisfying the conclusion of the lemma, contradicting the initial assumption in the proof. Hence, we must have 
	\begin{equation*}
		i_l=i_{j'} \quad\text{for all}\quad l\in\mathcal{J}^+_{G'}.
	\end{equation*}
	If it were $\#\mathcal{J}^+_{G'}> 1$, let $l,l'\in \mathcal{J}^+_{G'}$ with $l>l'$. Then, since $i_{l'}=i_l=i_{j'}$ and $G$ is a $q$-factorization graph, we must have $p_{l,l'}<0$. However, \Cref{l:lineps} %{l:positiveppath} 
	implies that
	\begin{equation*}
		p_{l,j''}<p_{l,l'}<0,
	\end{equation*} 
	contradicting \eqref{e:J+G'}. Similar arguments can be used to show that $\#\mathcal{J}^-_{G''}\leq 1$ and that $i_l=i_{j''}$ if $l\in\mathcal{J}^-_{G''}$. Henceforth, let $j^+$ denote the unique element of $\mathcal{J}^+_{G'}$, if it exists, and let $j^-$ be  the unique element of $\mathcal{J}^-_{G''}$, if it exists. In particular,
	\begin{equation}
		\begin{aligned}
			\mathcal{J}^-_{G'}=&\ \mathcal{J}_{G'}\setminus\{j^+\},\quad \mathcal{J}^+_{G''}=\mathcal{J}_{G''}\setminus\{j^-\}, \quad i_{j^+}=i_{j'}, \quad i_{j^-} = i_{j''},\\
			&\text{and}\quad j^-<j'\le l\le j''<j^+ \quad\textrm{for all}\quad l\in\mathcal{J}^-_{G'}\cup\mathcal{J}^+_{G''}.
		\end{aligned}
	\end{equation}
	Moreover, since $p_{j'',j'}\geq 0$ by \eqref{eq:newsubd},  \Cref{l:lineps}, %{l:positiveppath} 
	and \eqref{e:p>0intne} imply that 
	\begin{equation}\label{e:existarrow}
		0\leq p_{l,l'}<\operatorname{min}\{r_l,r_{l'}\} \quad\text{and}\quad i_l\ne i_{l'} \quad\textrm{for all}\quad l,l'\in\mathcal{J}^-_{G'}\cup\mathcal{J}^+_{G''},\quad l>l'.
	\end{equation}
	It follows that a pair $(k',k'')$ such that $k'\in\mathcal{J}^-_{G'}, k''\in\mathcal{J}^+_{G''}$, and $k'>k''$ satisfies the conclusion of the lemma and, hence, does not exist by the initial assumption of the proof. Thus, we must have
	\begin{equation*}
		l<l'\quad\textrm{for all}\quad l\in\mathcal{J}^-_{G'},\;l'\in\mathcal{J}^+_{G''}. 
	\end{equation*}
	Note also that
	\begin{equation}
		\bs\varpi'_J = \bs\omega^{(j^+)}\prod_{l\in\mathcal{J}_{G'}^-}\bs\omega^{(l)}_J \quad\text{and}\quad \bs\varpi''_J = \bs\omega^{(j^-)} \prod_{l\in\mathcal{J}_{G''}^+}\bs\omega^{(l)}_J,
	\end{equation}
	where we set $\bs\omega^{(j^\pm)}=1$ if $j^\pm$ does not exist. 	
	Let us check that
	\begin{equation}\label{e:moresing}
		\mathcal{J}^-_{G'}=\{j'\}\quad\textrm{and}\quad\mathcal{J}^+_{G''}=\{j''\}.
	\end{equation}
	Indeed, assume $\mathcal{J}^-_{G'}\setminus\{j'\}\ne\emptyset$, choose $l'\in\mathcal{J}^-_{G'}\setminus\{j'\}$ and $l''\in\mathcal{J}^+_{G''}$
	such that $d(i_{l'},i_{l''})$ is minimal and let $$\overline{J}=[i_{l'},i_{l''}]\subsetneqq J.$$
	The choice of $(l',l'')$ implies $\bs\varpi'_{\overline J}=\bs\omega^{(l')}_{\overline J}$, while 
	\begin{equation*}
		\bs\varpi''_{\overline J} = 
		\begin{cases}
			\bs\omega^{(l'')}_{\overline{J}}, & \text{if } l''\ne j'',\\
			\bs\omega^{(j'')}_{\overline{J}}\bs\omega^{(j^-)}_{\overline{J}}, & \text{if } l''= j''.
		\end{cases}
	\end{equation*}
	As commented after \eqref{e:doubred}, $M_J$ is highest-$\ell$-weight and, hence, so is
	\begin{equation*}
		M_{\overline{J}}:=L_q(\bs{\varpi}'_{\overline{J}})\otimes L_q(\bs{\varpi}''_{\overline{J}})=L_q(\bs{\omega}^{(l')}_{\overline{J}})\otimes L_q(\bs{\varpi}''_{\overline{J}}).
	\end{equation*}
	If $l''\ne j''$, we have 
	\begin{equation*}
		M_{\overline{J}} = L_q(\bs{\omega}^{(l')}_{\overline{J}})\otimes L_q(\bs\omega^{(l'')_{\overline{J}}}),
	\end{equation*} 
	which is not highest-$\ell$-weight by \eqref{e:existarrow}, yielding a contradiction. If $l''=j''$ (so $i_{l''}=i_{j''}$),  \eqref{e:plkub} and \eqref{e:p>0intne} imply $p_{j'',j^-}<0$ and, hence,
	\begin{equation*}
		L_q(\bs\varpi''_{\overline J})\cong L_q(\bs\omega^{(j'')}_{\overline{J}})\otimes L_q(\bs\omega^{(j^-)}_{\overline{J}}).
	\end{equation*}
	Therefore,
	\begin{equation*}
		M_{\overline{J}}\cong L_q(\bs{\omega}^{(l')}_{\overline{J}})\otimes L_q(\bs{\omega}^{(j'')}_{\overline{J}})\otimes L_q(\bs{\omega}^{(j^-)}_{\overline{J}}),
	\end{equation*}
	yielding a contradiction with  \eqref{e:existarrow} again.  This proves the first claim in \eqref{e:moresing} and the second is proved similarly.

	We have shown $\mathcal{J}_{G'}=\{j',j^+\}$ and $\mathcal{J}_{G''}=\{j'',j^-\}$, where we understand $j^\pm$ has not being listed if it does not exist. In particular, 
	\begin{equation*}
		J\cap\supp(\bs\varpi') = \{i_{j'}\} \quad\text{and}\quad J\cap\supp(\bs\varpi') = \{i_{j''}\},
	\end{equation*}
	which implies
	\begin{equation*}
		M_J=L_q((\bs{\omega}^{(j')}\bs{\omega}^{(j^+)})_J)\otimes L_q((\bs{\omega}^{(j'')}\bs{\omega}^{(j^-)})_J).
	\end{equation*}
	Since $p_{j^+,j'}<0$ and $p_{j'',j^-}<0$, it follows that
	\begin{equation*}
		M_J\cong L_q(\bs{\omega}^{(j^+)}_J) \otimes L_q(\bs{\omega}^{(j')}_J)\otimes L_q(\bs{\omega}^{(j'')}_J)\otimes L_q(\bs{\omega}^{(j^-)}_J).
	\end{equation*}
	However, $L_q(\bs{\omega}^{(j')}_J)\otimes L_q(\bs{\omega}^{(j'')}_J)$ is not highest-$\ell$-weight by \eqref{eq:newsubd}, yielding the promised contradiction. 
\end{proof}

\subsection{Proof of \Cref{t:toto}}\label{ss:ptoto}
Let $\bs\pi',\bs\pi''\in\mathcal{P}^+\setminus\{\bs 1\}$ be such that $\bs{\pi}=\bs\pi'\bs\pi''$ and set
\begin{equation*}
	U=L_q(\bs\pi')\otimes L_q(\bs\pi'')\quad\textrm{and}\quad V=L_q(\bs\pi'')\otimes L_q(\bs\pi')
\end{equation*}
In light of \Cref{c:vnvstar}, \Cref{t:toto} follows if we show that either $U$ or $V$ is not highest-$\ell$-weight. Moreover, by Corollary \ref{c:sJs}, we can assume $\bs\pi'$ and $\bs\pi''$ have dissociate $q$-factorizations. The case $N=1$ is obvious, while the case $N=2$ follows from the definition of $q$-factorization graph and \eqref{e:krhwtp}, since $G$ is connected. Thus, henceforth, $N\ge 3$. 	We shall assume $U$ and $V$ are highest-$\ell$-weight and reach a contradiction.

We will use the notation fixed before \Cref{l:totordopparrow}. Let also $G'=G(\bs\pi')=(\mathcal V',\mathcal A')$ and $G''=G(\bs\pi'')=(\mathcal V'',\mathcal A'')$.
Without loss of generality, assume $\bs\omega:=\bs\omega^{(N)}\in\mathcal V''$ ($\bs\omega$ is the source of $G$). We claim 
\begin{equation}\label{e:V''>1}
	\#\mathcal V''>1 \quad\text{and, hence,}\quad  \bs\pi''\bs\omega^{-1}\ne\bs 1.
\end{equation}
Indeed, if this were not the case, it would follow that $\bs\nu:=\bs\pi''\in\mathcal V$ and $\bs\pi'=\bs\pi\bs\nu^{-1}$. \Cref{l:totop} then implies $G'$ is also totally ordered and, letting $\bs\lambda=\bs\omega^{(N-1)}$ be the source of $G'$, it would follow that
\begin{equation}\label{e:fakearrow}
	(\bs\nu,\bs\lambda)\in\mathcal A.
\end{equation}	
Set also $\bs\mu=\bs\pi'\bs\lambda^{-1}$ and note $\bs\mu\in\mathcal P^+\setminus\{\bf 1\}$ since $N\ge 3$ and we are assuming $\#\mathcal V''=1$. By assumption, $L_q(\bs\lambda\bs\mu)\otimes L_q(\bs\nu)=U$ is highest-$\ell$-weight. On the other hand,  \Cref{l:hlwquot} implies $L_q(\bs\lambda)\otimes L_q(\bs\mu)$ is also highest-$\ell$-weight. Together with  \Cref{p:killhlw}, this implies $L_q(\bs\lambda)\otimes L_q(\bs\nu)$ is highest-$\ell$-weight as well, yielding a contradiction with \eqref{e:fakearrow} and \eqref{e:krhwtp}.

Note also that \Cref{c:killsinksource} implies that
\begin{equation*}
	\tilde U:=L_q(\bs\pi')\otimes L_q(\bs\pi''\bs\omega^{-1}) \quad\text{is highest-$\ell$-weight.}
\end{equation*}
Since $G(\bs\pi\bs\omega^{-1})$ is totally ordered by \Cref{l:totop}, an inductive argument on $N$ then implies 
\begin{equation}\label{e:V'nothlw}
	\tilde V:=L_q(\bs\pi''\bs\omega^{-1})\otimes L_q(\bs\pi') \quad\text{is not highest-$\ell$-weight.}
\end{equation}
Let
\begin{equation*}
	\mathcal I' = \{j: \bs\omega^{(j)}\in\mathcal V' \},\qquad  \mathcal I'' = \{j: \bs\omega^{(j)}\in\mathcal V'' \},
\end{equation*}
and 
\begin{equation*}
	\mathcal{I}'_>=\{j\in\mathcal I':  p_{N,j}\geq 0\}.
\end{equation*}
Let us show $\mathcal I'_>\ne \emptyset$. Set $p_{j,k}=r_j+r_k+d(i_j,i_k)-p_{k,j}$ for $j<k$, so we can write
\begin{equation*}
	m_j-m_k=r_j+r_k+d(i_j,i_k)-2p_{j,k}.
\end{equation*}
If it were  $\mathcal I'_>=\emptyset$, then $p_{j,N}=r_j+r_N+d(i_j,i_N)-p_{N,j}>r_j+r_N$ for all $1\le j<N$. Together with \Cref{p:killfactormid}, this contradicts \eqref{e:V'nothlw}.

If  $j\in\mathcal I'_>$ and $k>j$, it follows from  \Cref{l:lineps} %{l:positiveppath} 
that $p_{k,j}>p_{N,j}\geq 0$ and, hence, $i_j\neq i_k$ by \eqref{e:p>0intne}. This shows
\begin{equation*}
	i_j\ne i_k \quad\text{for all}\quad j,k\in\mathcal I'_>,\ j\ne k,
\end{equation*}
and, therefore, there exists  unique $j'\in\mathcal{I}'_>$ such that
\begin{equation*}
	0<d(i_{j'},i_N)=\operatorname{min}\{d(i_j,i_N): j\in\mathcal{I}'_>\}.
\end{equation*}
Set
\begin{equation*}
	\mathcal{I}''_>=\{j\in\mathcal I'': j'<j \textrm{ and } p_{j,j'}\geq 0\}
\end{equation*}
and note $N\in\mathcal I''_>$. Proceeding as above, one easily checks that
\begin{equation*}
	l\in\mathcal I''_>\ \text{ and }\ j'<k<l \quad\Rightarrow\quad i_k\ne i_l,
\end{equation*}
which implies $i_k\ne i_l$ for all $k,l\in\mathcal I_>'', k\ne l$. Let then $j''\in\mathcal I_>''$ be the unique element such that
\begin{equation*}
	0<d(i_{j'},i_{j''})=\operatorname{min}\{d(i_{j'},i_j): j\in\mathcal{I}''_>\}
\end{equation*}
and set $J=[i_{j'},i_{j''}]$. By construction  \eqref{eq:newsubd} holds. 	Since $U$ is highest-$\ell$-weight, \Cref{l:totordopparrow} then implies there exist $j_1''<j_1'$ such that $J_1:=[i_{j_1'},i_{j_1''}]\subsetneqq J$,  $\bs{\omega}^{(j_1')}\in\mathcal{V}_{G'}$, $\bs{\omega}^{(j_1'')}\in\mathcal{V}_{G''}$, and $m_{j_1'}-m_{j_1''}\in{\mathscr{R}_{i_{j_1'},i_{j_1''},J_1}^{r_{j_1'},r_{j_1''}}}$.

Since  $V$ is also highest-$\ell$-weight, \Cref{l:totordopparrow} with $V$ in place of $U$, $j_1''$ in place of $j'$ and $j_1'$ in place of $j''$, would imply there exist $j_2'<j_2''$ such that $J_2:=[i_{j_2'},i_{j_2''}]\subsetneqq J_1$,  $\bs{\omega}^{(j_2')}\in\mathcal{V}_{G'}$, $\bs{\omega}^{(j_2'')}\in\mathcal{V}_{G''}$, and $m_{j_2''}-m_{j_2'}\in{\mathscr{R}_{i_{j_2'},i_{j_2''},J_2}^{r_{j_2'},r_{j_2''}}}$. The same lemma  with $j_2'$ in place of $j'$ and $j_2''$ in place of $j''$ and so on would give rise to an infinite sequence $J\supsetneqq J_1\supsetneq J_2 \supsetneqq\cdots$ and, hence, the desired contradiction.\qed

\bibliographystyle{amsplain}

\end{document}